\documentclass[aos,preprint]{imsart}

\usepackage{amsmath,amssymb,amsfonts}
\usepackage{amsthm}
\usepackage{needspace}
\usepackage[numbers]{natbib}
\usepackage{hyperref}
\hypersetup{hidelinks,pdfauthor={Nikolaos Gkoumas, Nickos Papadatos, Samis Trevezas}}
\usepackage{tabularx}

\startlocaldefs
\newcommand{\EE}{\mathbb{E}}
\newcommand{\PP}{\mathbb{P}}
\newcommand{\Var}{\mathbb{V}\mathrm{ar}}

\newtheorem{theorem}{Theorem}
\newtheorem{prop}{Proposition}
\newtheorem{proposition}[prop]{Proposition}
\newtheorem{cor}[theorem]{Corollary}
\newtheorem{lem}{Lemma}
\newtheorem{lemma}[lem]{Lemma}
\theoremstyle{definition}
\newtheorem{myrem}{Remark}

\makeatletter
\renewcommand*\l@section{\@dottedtocline{1}{0em}{2.6em}}
\renewcommand*\l@subsection{\@dottedtocline{2}{1.5em}{3.6em}}
\renewcommand{\nocontentsline}{%
    \let\@@addcontentsline\addcontentsline
    \def\addcontentsline##1##2##3{\let\addcontentsline\@@addcontentsline}%
}
\makeatother
\endlocaldefs

\begin{document}

\begin{frontmatter}

\title{Hermite Spectra and Kernel Factorization for Gaussian-Weighted Tests of Normality}
\runtitle{Hermite spectra for normality tests}

\begin{aug}
\author[A]{\fnms{Nikolaos}~\snm{Gkoumas}}
\author[A]{\fnms{Nickos}~\snm{Papadatos}}
\author[A]{\fnms{Samis}~\snm{Trevezas}\ead[label=e1]{strevezas@math.uoa.gr}}
\address[A]{Department of Mathematics, National and Kapodistrian University of Athens\printead[presep={,\ }]{e1}}
\end{aug}

\begin{abstract}
We study the asymptotic null spectrum of Gaussian-weighted tests of univariate normality with estimated location and scale. Closed Hermite coefficients give equations both away from and at the unperturbed poles. Every positive eigenvalue of the fully standardized covariance is simple, and the odd and even eigenvalues alternate strictly. For each unperturbed even pole except the largest, there is exactly one weight parameter at which it belongs to the perturbed even spectrum. These parameters are strictly ordered and converge to the boundary of the parameter interval. A signed total-positivity argument proves simplicity more generally for consecutive Gaussian covariance corrections and positive even integrable weights. We also derive the limiting quadratic forms directly from degenerate kernels at standardized observations. A derivative summability condition controls the complete spectral tail, including signed kernels, and is verified for characteristic-function kernels with a finite fourth weight moment. A real Hilbert-space factorization identifies the observation and Fourier spectra. Numerical calculations assess asymptotic calibration.
\end{abstract}



\end{frontmatter}
\hypersetup{pdfsubject={Preprint}}

\tableofcontents

\section{Introduction}

Empirical characteristic functions have long been used in goodness-of-fit testing; see Feuerverger and Mureika~\cite{FeuervergerMureika1977}. The test of Epps and Pulley~\cite{EppsPulley1983} compares the empirical and normal characteristic functions in a Gaussian-weighted quadratic norm. Its multivariate extension is due to Baringhaus and Henze~\cite{BaringhausHenze1988}; related composite normality tests were studied by Cs{\"o}rg{\H{o}}~\cite{Csorgo1986}. Under the normal null, the limiting distributions are infinite weighted sums of independent chi-square variables. The weights are the eigenvalues of a covariance operator and determine asymptotic critical values.

The general setting is the spectral theory of degenerate $U$- and $V$-statistics; see Hoeffding~\cite{hoeffding1994class} and Gregory~\cite{Gregory1977}. Parameter substitution has been studied by Randles~\cite{Randles1982}, Pierce~\cite{Pierce1982}, de~Wet and Randles~\cite{DeWetRandles1987}, Shieh~\cite{Shieh1997}, and Cupari{\'c} et~al.~\cite{CuparicMilosevicObradovic2022}. Arcones~\cite[Theorems~2.3 and 3.3]{Arcones2007} gives generalized one-sample $U$-statistic substitution results and a normality-test application that includes the BHEP limit for block size one. For weakly dependent observations, see Leucht~\cite{Leucht2012}. In the normality problem, the covariance corrections due to location and scale estimation are given in Henze~\cite{Henze1990}. Explicit calculations for a particular weight occur in Baringhaus~\cite{baringhaus1996fibonacci}. Ebner and Henze~\cite{EbnerHenze2023} use a Gaussian Hermite basis and a finite-rank Fredholm determinant to compute the Epps--Pulley spectrum. These spectral representations provide the starting point for the calculations below. More recently, Ebner et~al.~\cite{EbnerJimenezGameroMilosevic2025} studied Rayleigh--Ritz approximation of covariance spectra, including the BHEP operator in dimensions one, two and three, and its use for asymptotic calibration.

We give closed expressions for the coefficients of these corrections and resolve two further questions about the spectrum. The chosen parameter makes the unperturbed eigenvalues geometric and yields an affine relation between the two even coefficient sequences. A bordered-matrix criterion supplements the resolvent equation at its poles. For every unperturbed even pole except the largest, there is exactly one parameter at which it is an even eigenvalue of the corrected covariance. The exceptional parameters are strictly increasing and tend to one. Thus the poles cannot be excluded, but their coincidences can be classified pole by pole. This strengthens the counterexample to the exclusion assertion in Ebner and Henze~\cite[Theorem~2.2]{EbnerHenze2023}; see Theorem~\ref{main-thm:unique_pole_sequence} and Section~\ref{main-subsec:EH_equivalence}.

We also prove that every positive eigenvalue of the fully standardized operator is simple, with strict alternation of the odd and even eigenvalues. Simplicity within the odd subspace follows already from its rank-one equation; it does not settle simplicity in the even subspace or exclude a common eigenvalue of the two subspaces. Theorem~\ref{main-thm:simple_full_spectrum} settles both questions. Its proof conjugates the covariance kernel to a strictly totally positive kernel and uses exterior powers. The same argument applies after subtraction of any consecutive initial block of Gaussian covariance directions, for an arbitrary positive even integrable weight; see Theorem~\ref{supp-thm:consecutive_corrections}. The total-positivity method is classical; the required signed kernel and its application to these covariance corrections are proved here.

The observation-space part derives the limit directly from the residual sums. Standardization replaces the empirical eigenfunction sums, to first order, by their projections away from the normal location and scale scores. The parameter-substitution results of Cupari{\'c} et~al.~\cite{CuparicMilosevicObradovic2022} already give $U$-statistic limits in which the original kernel trace is subtracted. The identification of the projected normal summand also agrees with Arcones~\cite{Arcones2007}; Section~\ref{supp-sec:Arcones_comparison} gives the exact correspondence. Our concern is the justification of the entire observation-space spectral expansion after standardization, not a new general parameter-substitution principle. We give a tail condition with the infinite tail taken before the sample-size limit. Proposition~\ref{main-prop:derivative_tail_criterion} controls the empirical eigenfunction sums uniformly in sample size by Gaussian norms of the functions and their first derivatives. Its summability consequence applies also to signed kernels. Corollary~\ref{main-cor:weighted_kernel_full_tail} verifies it for the complete weighted characteristic-function kernel, of either parity, whenever the weight has a finite fourth moment.

For characteristic-function kernels, the projected observation-space operator and the Fourier covariance are products of a Hilbert--Schmidt operator and its adjoint, in opposite orders. They have the same positive spectrum. The factorization uses the real space of Hermitian-symmetric functions, which fixes the multiplicities in the chi-square series. Inner-product kernel representations are familiar from Hofmann et~al.~\cite{HofmannScholkopfSmola2008} and Sejdinovic et~al.~\cite{SejdinovicSriperumbudurGrettonFukumizu2013}; here their role is to identify the operators in the two statistical derivations. If the score space is generated by kernel eigenfunctions, projection removes those eigendirections. For the odd Gaussian kernel, the location score has nonzero coordinates along every odd eigenfunction, and all odd eigenvalues change instead.

The two parts concern the same limiting eigenvalues: the factorization identifies the operator obtained from residual sums with the covariance operator treated by the secular equations. The derivative estimate justifies this identification through the infinite observation-space expansion; it does not rely on the subsequent explicit solution of the Fourier eigenvalue problem.

The numerical study examines the finite-sample accuracy of asymptotic calibration for the fully standardized Epps--Pulley $V$-statistic. We approximate the eigenvalues by the Rayleigh--Ritz method in the Gaussian Hermite basis, using the coefficients of Proposition~\ref{main-prop:overlaps_closed_form}. The geometric spectrum gives a bound on the truncation error. Section~\ref{supp-sec:numerical_use} of the Supplement reports the first ten eigenvalues, asymptotic percentage points and rejection frequencies under the normal null, based on 50,000 replications for each sample size.

Section~\ref{main-sec:positioning} gives the kernel representation and the relation between the weight parameters. Section~\ref{main-sec:spectral_preliminaries} develops the Hermite and perturbation calculations. Section~\ref{main-sec:data_domain_projection} proves the projection, tail and factorization results. Section~\ref{main-sec:full_studentization} gives the full spectrum, its simplicity and parity ordering, and the classification of coincidences at even poles. One-parameter results, auxiliary proofs and numerical details are in the Supplement. Section~\ref{main-sec:discussion} distinguishes the scope of the spectral and statistical conclusions. All statistical results are univariate.

\section{Background and kernel representation}
\label{main-sec:positioning}
\label{main-sec:background_kernel_view}

\subsection{Weighted distance and its kernel}
\label{main-subsec:cf_and_ecf}
\label{main-subsec:EPBHEP_as_weighted_distance}
\label{main-subsec:general_weighted_norm_identity}
Let $\phi(t)=\EE[\exp\!\left\{itX\right\}]$ be the characteristic function of $F$, and let $\phi_0$ be that of a target law $F_0$ symmetric about zero. Thus $\phi_0$ is real and even. For a positive, even, integrable weight $\theta$, define
\begin{equation}\label{main-eq:Delta_theta_def}
d_\theta^2(F,F_0)=\int_{\mathbb R}\theta(t)|\phi(t)-\phi_0(t)|^2\,dt.
\end{equation}
This quantity is at most $4\|\theta\|_{L^1}$. Its square root is a distance, and continuity and uniqueness of characteristic functions give
\begin{equation}\label{main-eq:characterize_phi0_with_DF=0}
d_\theta^2(F,F_0)=0\quad\Longleftrightarrow\quad\phi=\phi_0
\quad\Longleftrightarrow\quad F=F_0.
\end{equation}
Write
\begin{align}
c_1(u)&=\int_{\mathbb R}\theta(t)\phi_0(t)\cos(tu)\,dt,\label{main-eq:c1_def}\\
c_2(u)&=\int_{\mathbb R}\theta(t)\cos(tu)\,dt,\qquad u\in\mathbb R,\label{main-eq:c2_def}
\end{align}
and
\begin{equation}\label{main-eq:UB_def_general}
\mathrm{UB}=\int_{\mathbb R}\theta(t)\phi_0^2(t)\,dt.
\end{equation}
The following identity gives the observation-space kernel.
\begin{proposition}\label{main-prop:weighted_L2_kernel_rep}
For independent $X,Y$ with law $F$,
\begin{equation}\label{main-eq:Delta_identity}
d_\theta^2(F,F_0)=\mathrm{UB}-\EE[2c_1(X)-c_2(X-Y)].
\end{equation}
Equivalently, the symmetric kernel
\begin{equation}\label{main-eq:Ktheta_def}
K_\theta(x,y)=c_2(x-y)-c_1(x)-c_1(y)+\mathrm{UB}
\end{equation}
satisfies
\begin{equation}\label{main-eq:Delta_kernel_rep}
d_\theta^2(F,F_0)=\EE[K_\theta(X,Y)].
\end{equation}
\end{proposition}
\begin{proof}
Expanding the squared modulus gives
\begin{equation*}
|\phi-\phi_0|^2=|\phi|^2-2\phi_0\Re\phi+\phi_0^2.
\end{equation*}
Independence gives $|\phi(t)|^2=\EE[\cos(t(X-Y))]$, while $\Re\phi(t)=\EE[\cos(tX)]$. Fubini's theorem, justified by integrability of $\theta$, and \eqref{main-eq:c1_def}, \eqref{main-eq:c2_def} and \eqref{main-eq:UB_def_general} yield \eqref{main-eq:Delta_identity}. The definition \eqref{main-eq:Ktheta_def} gives \eqref{main-eq:Delta_kernel_rep}.
\end{proof}
\begin{cor}\label{main-cor:UB_inequality}
Under the preceding assumptions, $\EE[2c_1(X)-c_2(X-Y)]\le\mathrm{UB}$, with equality if and only if $F=F_0$.
\end{cor}
\begin{proof}
Combine \eqref{main-eq:Delta_identity} with \eqref{main-eq:characterize_phi0_with_DF=0}.
\end{proof}
For independent observations, $U_n(\theta)=2\{n(n-1)\}^{-1}\sum_{i<j}K_\theta(X_i,X_j)$ is therefore unbiased for \eqref{main-eq:Delta_theta_def}.

\subsection{The Gaussian kernel and its parametrization}
\label{main-subsec:Gaussian_kernel_Krho}
From now on $F_0=\Phi$, with density $f_0(x)=(2\pi)^{-1/2}\exp\!\left\{-x^2/2\right\}$ and characteristic function
\begin{equation}\label{main-eq:phi0_normal}
\phi_0(t)=\exp\!\left\{-t^2/2\right\}.
\end{equation}
The notation $L^2(\Phi)=L^2(f_0)$ will be used with the scalar field specified. For $0<\rho<1$, set
\begin{equation}\label{main-eq:theta_rho_weight}
\theta_\rho(t)=\frac{1}{\sqrt{2\pi\rho}}
\exp\!\left\{-\frac{(1-\rho)^2t^2}{2\rho}\right\}.
\end{equation}
Its integral is $(1-\rho)^{-1}$. Define $\mu_F(\rho)=d_{\theta_\rho}^2(F,\Phi)$.
\begin{cor}\label{main-lem:ineq}
For every $0<\rho<1$, $\mu_F(\rho)\ge0$, with equality if and only if $F=\Phi$.
\end{cor}
This follows from \eqref{main-eq:characterize_phi0_with_DF=0}. Using \eqref{main-eq:phi0_normal} in \eqref{main-eq:c1_def}--\eqref{main-eq:UB_def_general} gives
\begin{align*}
c_1(u)&=a_2(\rho)\exp\!\left\{-b_2(\rho)u^2\right\},\qquad
c_2(u)=a_1(\rho)\exp\!\left\{-b_1(\rho)u^2\right\},\\
\mathrm{UB}&=a_3(\rho).
\end{align*}
Thus $K_\rho=K_{\theta_\rho}$ is
\begin{align}\label{main-kernel-K}
K_\rho(x,y)&=a_1(\rho)\exp\!\left\{-b_1(\rho)(x-y)^2\right\}\nonumber\\
&\quad-a_2(\rho)\left(\exp\!\left\{-b_2(\rho)x^2\right\}
+\exp\!\left\{-b_2(\rho)y^2\right\}\right)+a_3(\rho),
\end{align}
where
\begin{align*}
a_1(\rho)&=\frac1{1-\rho},\qquad
 a_2(\rho)=\frac1{\sqrt{1-\rho+\rho^2}},\qquad
 a_3(\rho)=\frac1{\sqrt{1+\rho^2}},\\
b_1(\rho)&=\frac{\rho}{2(1-\rho)^2},\qquad
 b_2(\rho)=\frac{\rho}{2(1-\rho+\rho^2)}.
\end{align*}
For the kernel \eqref{main-kernel-K}, the associated statistic is $U_n(\rho)=2\{n(n-1)\}^{-1}\sum_{i<j}K_\rho(X_i,X_j)$, with $\EE U_n(\rho)=\mu_F(\rho)$.

\subsection{The Epps--Pulley scale and the two quadratic forms}
\label{main-subsec:EPBHEP_exact}
The normalized Gaussian weight of Epps and Pulley~\cite{EppsPulley1983} is
\begin{equation}\label{main-eq:EP_varphi_beta}
\varphi_\beta(t)=\frac{1}{\beta\sqrt{2\pi}}
\exp\!\left\{-\frac{t^2}{2\beta^2}\right\},\qquad\beta>0.
\end{equation}
The parameters are related by
\begin{equation}\label{main-eq:beta_rho}
\beta=\beta(\rho)=\frac{\sqrt\rho}{1-\rho},\qquad
\rho=\frac{2\beta^2+1-\sqrt{4\beta^2+1}}{2\beta^2}.
\end{equation}
Comparing \eqref{main-eq:theta_rho_weight} and \eqref{main-eq:EP_varphi_beta} gives
\begin{equation}\label{main-eq:theta_vs_varphi}
\varphi_\beta=(1-\rho)\theta_\rho.
\end{equation}
Hence, writing $K_\beta$ for the kernel induced by $\varphi_\beta$,
\begin{align}
\mu_F(\rho)&=(1-\rho)^{-1}d_{\varphi_\beta}^2(F,\Phi),\label{main-eq:scaling_mu_EP}\\
K_\rho(x,y)&=(1-\rho)^{-1}K_{\beta(\rho)}(x,y).\label{main-eq:kernel_scaling_rho_beta}
\end{align}
The choice of $\rho$ will give geometric eigenvalues for the unperturbed operator, without changing the statistic except by this deterministic scale.

For the raw observations and for residuals $U_{n,j}$, respectively, let
\begin{equation*}
\phi_n(t)=\frac1n\sum_{j=1}^n\exp\!\left\{itX_j\right\},\qquad
\psi_n(t)=\frac1n\sum_{j=1}^n\exp\!\left\{itU_{n,j}\right\}.
\end{equation*}
Expansion of the squared modulus gives
\begin{equation*}
\Delta_n(\theta):=\int\theta(t)|\phi_n(t)-\phi_0(t)|^2\,dt
=\frac1{n^2}\sum_{j,k=1}^nK_\theta(X_j,X_k).
\end{equation*}
Thus the known-parameter statistic is $T_{n,\beta}^{(0)}=n\Delta_n(\varphi_\beta)$. Replacing $X_j$ by standardized residuals gives $T_{n,\beta}=n\int|\psi_n-\phi_0|^2\varphi_\beta$, the classical composite-null statistic; see Epps and Pulley~\cite{EppsPulley1983}, Henze~\cite{Henze1990}, and Baringhaus and Henze~\cite{BaringhausHenze1988}.

For any symmetric $K$ and any observations, including dependent residuals, put
\begin{equation*}
V_n(K)=\frac1{n^2}\sum_{j,k}K(X_j,X_k),\qquad
U_n(K)=\frac2{n(n-1)}\sum_{j<k}K(X_j,X_k).
\end{equation*}
The exact algebraic identity is
\begin{equation}\label{main-eq:UV_identity}
nU_n(K)=\frac{n}{n-1}\left\{nV_n(K)-\frac1n\sum_jK(X_j,X_j)\right\}.
\end{equation}
Under $F=\Phi$, $K_\beta$ and $K_\rho$ are canonical: $\EE[K_\beta(x,X)]=0$. The distinction between the two forms in \eqref{main-eq:UV_identity} is therefore retained at their null scale; see Hoeffding~\cite{hoeffding1994class} and Gin{\'e} and Nickl~\cite{gine2021mathematical}. It will give the original-trace subtraction in Corollary~\ref{main-cor:full_U_V_limit}.

\section{Spectral and perturbation preliminaries}
\label{main-sec:spectral_preliminaries}

The limit theory in the parameter-estimation settings is formulated on a frequency-domain Hilbert space.
Indeed, the statistics studied later are weighted $L^2$ norms of (centered or rescaled)
empirical characteristic functions, hence their Gaussian limits are random elements in a
weighted $L^2$ space.  The corresponding covariance operators are compact and positive, and
their nonzero eigenvalues determine the limiting quadratic-form distributions.
This section introduces the operator notation and gives the explicit
diagonalization of the unperturbed Gaussian kernel operator used in the finite-rank perturbation arguments.

\subsection{Operator notation}
We work with real separable Hilbert spaces unless a complexification is stated. For $g,h\in\mathcal H$, write $(g\otimes h)f=\langle f,h\rangle_{\mathcal H}g$. Positive eigenvalues of compact self-adjoint operators are counted with multiplicity. The notation $\|\cdot\|_1$ denotes the trace norm. We use the spectral theorem, the min--max principle and the trace-class property of square-integrable covariance operators; the statements and references are collected in Section~\ref{supp-sec:supp_operator_facts}. The distinction between a real covariance space and a complex function space is made explicitly below.

\subsection{The frequency-domain Hilbert space}
\label{main-subsec:frequency_Hilbert}
Fix $\rho\in(0,1)$ and let $\beta=\beta(\rho)$ as in \eqref{main-eq:beta_rho}. Spectral calculations with real kernels are made in $L^2(\varphi_\beta;\mathbb R)$. We also use its complexification, with inner product
\begin{equation*}
\langle f,g\rangle_\beta:=\int_{\mathbb R}f(t)\overline{g(t)}\varphi_\beta(t)\,dt,
\qquad \|f\|_\beta^2:=\langle f,f\rangle_\beta.
\end{equation*}
The empirical characteristic-function differences belong to the closed real subspace
\begin{equation}
\label{main-eq:Hermitian_space}
\mathcal H_\beta:=\{f\in L^2(\varphi_\beta;\mathbb C):f(-t)=\overline{f(t)}\text{ for almost every }t\}.
\end{equation}
The inner product above is real on $\mathcal H_\beta$, because the weight is even. All Gaussian limits of characteristic-function processes below are real Gaussian elements of $\mathcal H_\beta$.

For real $f$, put $f_{\mathrm e}(t)=\{f(t)+f(-t)\}/2$ and $f_{\mathrm o}(t)=\{f(t)-f(-t)\}/2$. The map
\begin{equation}
\label{main-eq:parity_isometry}
J_\beta f:=f_{\mathrm e}+i f_{\mathrm o}
\end{equation}
is a surjective real isometry from $L^2(\varphi_\beta;\mathbb R)$ onto $\mathcal H_\beta$. Indeed, the real and imaginary parts of a function in $\mathcal H_\beta$ are respectively even and odd, and parity makes these parts orthogonal. A real integral kernel satisfying $C(-s,-t)=C(s,t)$ preserves parity. Its integral operators on these two real spaces are therefore intertwined by $J_\beta$. Their eigenvalues, including multiplicities, agree. In particular, an even real eigenfunction remains unchanged, whereas an odd real eigenfunction is multiplied by $i$ in $\mathcal H_\beta$.

Finally, $\varphi_1=f_0$ and the change of variables $t=\beta x$ gives
\begin{equation*}
\int_{\mathbb R}|f(t)|^2\varphi_\beta(t)\,dt
=\int_{\mathbb R}|f(\beta x)|^2 f_0(x)\,dx.
\end{equation*}
This change of scale will be used for the Hermite basis.

\subsection{The unperturbed Gaussian kernel operator}
\label{main-subsec:Gaussian_kernel_operator_A0}

Define the Gaussian kernel
\begin{equation*}
K_0(s,t):=\exp\!\left\{-\frac{(s-t)^2}{2}\right\},\qquad s,t\in\mathbb R,
\end{equation*}
and the integral operator $A_0:L^2(\varphi_\beta)\to L^2(\varphi_\beta)$ by
\begin{equation}
\label{main-eq:oneparam_A0}
(A_0 f)(s):=\int_{-\infty}^{\infty} K_0(s,t)\,f(t)\,\varphi_\beta(t)\,dt.
\end{equation}
Since $0\le K_0\le 1$ and $\varphi_\beta$ is a probability density, we have
$K_0\in L^2(\varphi_\beta\otimes\varphi_\beta)$, hence $A_0$ is Hilbert--Schmidt and therefore compact.
Moreover, $K_0(s,t)=K_0(t,s)$ implies that $A_0$ is self-adjoint, and $K_0$ is positive definite (as the Fourier
transform of a Gaussian density), so $A_0$ is positive.
The parametrization \eqref{main-eq:beta_rho} gives geometric eigenvalues and explicit Hermite eigenfunctions for $A_0$ (Proposition~\ref{main-prop:Hermite_basis}).

\subsection{Hermite polynomials and Mehler's formula}
\label{main-subsec:Hermite_mehler}

\noindent\textbf{Two Hermite normalizations.}
There are two standard Hermite families in the literature.  The \emph{probabilists' Hermite polynomials}
$\{\mathrm{He}_n:n\ge 0\}$ are orthogonal with respect to the standard normal density and are defined by
\begin{equation*}
\mathrm{He}_n(x)=(-1)^n \exp\!\left\{x^2/2\right\}\,\frac{d^n}{dx^n}\,\exp\!\left\{-x^2/2\right\},\qquad n=0,1,2,\dots,
\end{equation*}
or, equivalently, by the generating function
\begin{equation*}
\exp\!\left\{tx-t^2/2\right\}=\sum_{n=0}^{\infty}\mathrm{He}_n(x)\,\frac{t^n}{n!}.
\end{equation*}
The \emph{physicists' Hermite polynomials} $\{H_n\}$ correspond to the weight $\exp\!\left\{-x^2\right\}$ and satisfy the relation
$H_n(x)=2^{n/2}\mathrm{He}_n(\sqrt2\,x)$; see {National Institute of Standards and Technology}~\cite[Section~18.3]{NIST:DLMF} for a systematic discussion.
We use the probabilists' normalization because it is adapted to Gaussian derivatives and to $L^2(f_0)$.

\medskip
Setting $h_n(x):=\mathrm{He}_n(x)/\sqrt{n!}$ yields an orthonormal basis $\{h_n\}_{n\ge 0}$ of $L^2(f_0)$;
see, for example, {National Institute of Standards and Technology}~\cite[Section~18.18(i)]{NIST:DLMF} or Koekoek et~al.~\cite{koekoek2010hypergeometric}.

\medskip
\noindent\textbf{Mehler's identity and the reproducing property.}
For $\rho\in(-1,1)$, Mehler's identity states that
\begin{equation}
\label{main-eq:mehler_identity}
\sum_{n=0}^{\infty} \rho^n h_n(x)\,h_n(y)
=
\frac{1}{\sqrt{1-\rho^2}}\,
\exp\!\left\{\frac{\rho xy-\tfrac12\rho^2(x^2+y^2)}{1-\rho^2}\right\},
\qquad x,y\in\mathbb R,
\end{equation}
with absolute convergence for each $(x,y)$.
This formula goes back to Mehler~\cite{mehler1866ueber}; modern references and further identities are collected in
{National Institute of Standards and Technology}~\cite[Eq.~(18.18.28)]{NIST:DLMF}.
It implies that the integral operator on $L^2(f_0)$ with kernel given by the right-hand side of
\eqref{main-eq:mehler_identity} acts diagonally on $\{h_n\}$, multiplying the $n$th Hermite component by $\rho^n$.
This identity is the form used below to diagonalize $A_0$.

\subsection{Diagonalization of the unperturbed operator}
\label{main-subsec:Hermite_diagonalization}

The following form of the Gaussian Hermite diagonalization uses the parametrization \eqref{main-eq:beta_rho}; compare Ebner and Henze~\cite[Section~2]{EbnerHenze2023}.

\begin{proposition}
\label{main-prop:Hermite_basis}
Fix $\rho\in(0,1)$ and set $\beta=\sqrt{\rho}/(1-\rho)$.
The operator $A_0$ in \eqref{main-eq:oneparam_A0} is compact, self-adjoint and positive on $L^2(\varphi_\beta)$, with
eigenvalues
\begin{equation}
\label{main-eq:oneparam_unpert_eigs}
\lambda_n^{(0)}(\rho)=(1-\rho)\,\rho^{n},\qquad n=0,1,2,\dots,
\end{equation}
and an orthonormal eigenbasis $\{e_n(\cdot;\rho)\}_{n\ge 0}$ given by
\begin{equation*}
e_n(t;\rho)
=
c_\rho\,\exp\!\left\{-\frac{1-\rho}{2}\,t^2\right\}\,
h_n\!\left(\alpha_\rho t\right),
\qquad
c_\rho:=\Bigl(\frac{1+\rho}{1-\rho}\Bigr)^{1/4},
\qquad
\alpha_\rho:=\frac{\sqrt{1-\rho^2}}{\sqrt\rho}.
\end{equation*}
In particular, $e_n$ is even for even $n$ and odd for odd $n$.

\smallskip
Moreover, under the isometry $t=\beta x$ between $L^2(\varphi_\beta)$ and $L^2(f_0)$, the functions
\begin{equation}
\label{main-eq:gn_en_link}
g_n(x;\rho):=e_n(\beta x;\rho)
=
c_\rho\,\exp\!\left\{-\frac{\rho}{2(1-\rho)}\,x^2\right\}\,
h_n\!\left(\sqrt{\frac{1+\rho}{1-\rho}}\,x\right),
\qquad n\ge 0,
\end{equation}
form an orthonormal basis of $L^2(f_0)$.
\end{proposition}

The proof is given in Section~\ref{supp-subsec:proof_Hermite_basis} of the Supplementary Material.

The unnormalized weight $\theta_\rho$ only changes scale.

\begin{myrem}\label{main-rem:scaling_eigs}
By \eqref{main-eq:theta_vs_varphi}, the inner product for $\theta_\rho$ is $(1-\rho)^{-1}\langle\cdot,\cdot\rangle_\beta$. The operator $A_{0,\rho}$ obtained by integrating $K_0$ against $\theta_\rho$ is $(1-\rho)^{-1}A_0$. It has the same eigenfunctions and eigenvalues $\lambda_n^{(0),\rho}=\rho^n$. The same factor relates the functionals and kernels in \eqref{main-eq:scaling_mu_EP}--\eqref{main-eq:kernel_scaling_rho_beta}. Interlacing is unchanged by this positive scale.
\end{myrem}
The roles of the two kernels can be stated directly.
\begin{myrem}\label{main-rem:leading_and_centered_kernel}
Let $B_\rho$ be the integral operator with kernel $K_\rho$ on $L^2(\Phi)=L^2(f_0)$, and let $D_\rho$ have kernel $(1-\rho)^{-1}K_0(\beta x,\beta y)$. The rescaling in \eqref{main-eq:gn_en_link} gives $D_\rho g_n=\rho^n g_n$. If $P_0f=f-\int f\,d\Phi$, then $B_\rho=P_0D_\rho P_0$: the two single-variable terms and the constant in $K_\rho$ are precisely its centering terms. Hence $B_\rho g_{2j+1}=\rho^{2j+1}g_{2j+1}$, but the even part of $B_\rho$ is not diagonal in the even $g_n$.

Even with known parameters, the centered empirical characteristic function has covariance kernel $K_0(s,t)-\phi_0(s)\phi_0(t)$, not $K_0(s,t)$ alone. The term $\phi_0(s)\phi_0(t)$ accounts for centering; estimation of location and scale gives two further covariance corrections.
\end{myrem}

\subsection{Rank-one directions and overlap coefficients}
\label{main-subsec:overlaps}

The parameter-estimation corrections that appear in the EP/BHEP limits can be written as finite-rank perturbations of
$A_0$ along a small number of ``score-type'' directions.
To avoid notational conflict with characteristic functions, we use starred symbols.

\medskip
Recall the standard normal characteristic function $\phi_0(t)=\exp\!\left\{-t^2/2\right\}$.
Define
\begin{equation}
\label{main-eq:oneparam_s012}
 \phi_0^*(t):=\phi_0(t),
 \qquad
 \phi_1^*(t):=t\,\phi_0(t),
 \qquad
 \phi_2^*(t):=\frac{t^2}{\sqrt{2}}\,\phi_0(t),
 \qquad t\in\mathbb R.
\end{equation}
The direction $\phi_0^*$ accounts for centering. The directions $\phi_1^*$ and $\phi_2^*$ arise from the location and scale derivatives. They coincide, up to order, with the functions used in
Henze~\cite{Henze1990} and Ebner and Henze~\cite{EbnerHenze2023}.

\medskip
For $k\in\{0,1,2\}$ and $n\ge 0$, define the overlap coefficients
\begin{equation}
\label{main-eq:oneparam_aik}
 a_{n,k}(\rho):=\langle e_n(\cdot;\rho), \phi_k^*\rangle_\beta.
\end{equation}
By parity of $e_n$ and $\phi_k^*$ we have
\begin{equation*}
a_{2m,1}(\rho)=0,
\qquad
a_{2m+1,0}(\rho)=a_{2m+1,2}(\rho)=0,
\qquad m\ge 0.
\end{equation*}
\noindent
\emph{Notational remark.} The cosine-transform functions $c_1,c_2$ introduced in Section~\ref{main-subsec:general_weighted_norm_identity} are unrelated to the score-type directions $\phi_0^*,\phi_1^*,\phi_2^*$ defined above.

The required overlap coefficients are as follows.

\begin{proposition}
\label{main-prop:overlaps_closed_form}
Fix $\rho\in(0,1)$ and set $\beta=\sqrt{\rho}/(1-\rho)$.
Let $\{(\lambda_n^{(0)}(\rho),e_n(\cdot;\rho))\}_{n\ge 0}$ be the eigenpairs of $A_0$ from Proposition~\ref{main-prop:Hermite_basis},
and let $\phi_0^*,\phi_1^*,\phi_2^*$ be as in \eqref{main-eq:oneparam_s012}.
Then the possibly nonzero overlap coefficients in \eqref{main-eq:oneparam_aik} are, for $m=0,1,2,\dots,$
\begin{align*}
a_{2m,0}(\rho)
&=
(-1)^m\,(1-\rho)^{3/4}(1+\rho)^{1/4}\,
\frac{\sqrt{(2m)!}}{2^m\,m!}\,\rho^{2m},
\\
a_{2m+1,1}(\rho)
&=
(-1)^m\,(1+\rho)^{3/4}(1-\rho)^{5/4}\,
\frac{\sqrt{(2m+1)!}}{2^m\,m!}\,\rho^{2m+\tfrac12},
\\
a_{2m,2}(\rho)
&=
\frac{(-1)^m}{\sqrt{2}}\,(1-\rho)^{3/4}(1+\rho)^{1/4}\,
\frac{\sqrt{(2m)!}}{2^m\,m!}\,\rho^{2m-1}\,\Bigl(\rho^2-2m(1-\rho^2)\Bigr).
\end{align*}
Moreover, on the even sector the coefficients satisfy the affine relation
\begin{equation}
\label{main-eq:a22_affine_relation_global}
a_{2m,2}(\rho)
=
\left(\frac{\rho}{\sqrt{2}}-\frac{\sqrt{2}(1-\rho^2)}{\rho}\,m\right)a_{2m,0}(\rho),
\qquad m\ge 0.
\end{equation}
All remaining overlaps vanish by parity.
\end{proposition}

\begin{proof}
Write $C_\rho=(1-\rho)^{3/4}(1+\rho)^{1/4}$ and
\begin{equation*}
I_k(z):=c_\rho\int_{\mathbb R}t^k
\exp\!\left\{-(2-\rho)t^2/2\right\}
\exp\!\left\{z\alpha_\rho t-z^2/2\right\}\varphi_\beta(t)\,dt.
\end{equation*}
The combined Gaussian factor has variance $\rho$. Its moment generating function gives
\begin{align*}
I_0(z)&=C_\rho\exp\!\left\{-\rho^2z^2/2\right\},\\
I_1(z)&=C_\rho\rho\alpha_\rho z\exp\!\left\{-\rho^2z^2/2\right\},\\
I_2(z)&=C_\rho\{\rho+\rho^2\alpha_\rho^2z^2\}
\exp\!\left\{-\rho^2z^2/2\right\}.
\end{align*}
The coefficient of $z^n$ in $I_k$ is $\langle e_n,t^k\phi_0\rangle_\beta/\sqrt{n!}$. Coefficient comparison, with the factor $1/\sqrt2$ for $k=2$, proves the formulas and their affine relation. Differentiation and coefficient comparison under the integral are justified by Gaussian domination for $z$ in compact sets. The coefficients with $k=0$ and $k=1$ displayed above never vanish. Those with $k=2$ may vanish when $\rho^2=2m/(2m+1)$, $m\ge1$.
\end{proof}

\subsection{Finite-rank perturbations and secular equations}
\label{main-subsec:finite_rank_lemma}

The covariance operators arising from location and scale estimation can be written as finite-rank perturbations of the unperturbed Gaussian operator $A_0$.  The following lemma gives the finite-rank secular equation and the rank-one interlacing criterion.

\begin{lemma}
\label{main-lem:finite_rank_secular}
Let $A_0$ be the compact self-adjoint operator in \eqref{main-eq:oneparam_A0}.  Its eigenpairs are
\begin{equation*}
\{(\lambda_n^{(0)}(\rho),e_n)\}_{n\ge0},
\end{equation*}
as in \eqref{main-eq:oneparam_unpert_eigs}.  For real $u_1,\dots,u_r\in L^2(\varphi_\beta;\mathbb R)$ define the finite-rank perturbation
\begin{equation*}
A:=A_0-\sum_{\ell=1}^r u_\ell\otimes u_\ell,
\qquad
(u\otimes u)f:=\langle f,u\rangle_\beta\,u.
\end{equation*}
For real $\lambda\notin\{0\}\cup\{\lambda_n^{(0)}(\rho)\}_{n\ge0}$ set $a_{n,\ell}:=\langle e_n,u_\ell\rangle_\beta$ and
\begin{equation*}
G_{\ell m}(\lambda):=\sum_{n=0}^\infty \frac{a_{n,\ell}\,a_{n,m}}{\lambda-\lambda_n^{(0)}(\rho)},
\qquad 1\le \ell,m\le r.
\end{equation*}
Then $\lambda$ is an eigenvalue of $A$ if and only if
\begin{equation}
\label{main-eq:finite_rank_det}
\det\bigl(I_r+G(\lambda)\bigr)=0.
\end{equation}
Its multiplicity equals $\dim\ker(I_r+G(\lambda))$, which is also the order of the zero of this determinant.

\smallskip
\noindent
In particular, if $r=1$ then the nonzero eigenvalues of $A$ away from $\{\lambda_n^{(0)}(\rho)\}$ are the roots of
\begin{equation}
\label{main-eq:finite_rank_rank1}
F(\lambda):=
1+\sum_{n=0}^\infty \frac{a_{n,1}^2}{\lambda-\lambda_n^{(0)}(\rho)}=0.
\end{equation}
Assume furthermore that we restrict to an invariant subspace on which the unperturbed eigenvalues form a strictly
decreasing sequence of poles $\kappa_0>\kappa_1>\kappa_2>\cdots\downarrow 0$ (e.g.\ the even or odd subspace),
and that the corresponding overlap coefficients are all nonzero:
$a_{n,1}\neq 0$ for every eigenfunction in that subspace.
Then $F$ is strictly \emph{decreasing} on each interval $(\kappa_m,\kappa_{m-1})$ and satisfies
\begin{equation*}
\lim_{\lambda\downarrow \kappa_m}F(\lambda)=+\infty,
\qquad
\lim_{\lambda\uparrow \kappa_{m-1}}F(\lambda)=-\infty,
\qquad m\ge 1.
\end{equation*}
Consequently, $F(\lambda)=0$ has exactly one simple root in every interval $(\kappa_m,\kappa_{m-1})$, and these roots
strictly interlace with the poles. They exhaust the positive spectrum in that subspace.
\end{lemma}

The proof is given in Section~\ref{supp-subsec:proof_finite_rank_secular} of the Supplementary Material.

We shall also use the following rank-one interlacing consequence of the min--max principle.

\begin{lemma}
\label{main-lem:minmax_rank1}
Let $A:\mathcal H\to\mathcal H$ be a compact positive self-adjoint operator on an infinite-dimensional real Hilbert space $\mathcal H$,
with infinitely many positive eigenvalues $\lambda_0(A)\ge \lambda_1(A)\ge \lambda_2(A)\ge\cdots\downarrow 0$ repeated by multiplicity.
For $u\in\mathcal H$ define $\widetilde A:=A-u\otimes u$. This operator has infinitely many positive eigenvalues. Enumerating them in decreasing order with multiplicity, we have, for every $m\ge 0$,
\begin{equation*}
\lambda_m(A)\ \ge\ \lambda_m(\widetilde A)\ \ge\ \lambda_{m+1}(A).
\end{equation*}
\end{lemma}

The proof is given in Section~\ref{supp-subsec:proof_minmax_rank1} of the Supplementary Material.

At an unperturbed eigenvalue the resolvent formula is not defined. The following criterion covers that case.
\begin{proposition}\label{main-prop:pole_criterion}
In Lemma~\ref{main-lem:finite_rank_secular}, put $d_j=\lambda_j^{(0)}(\rho)$, $v_j=(a_{j,1},\ldots,a_{j,r})^\top$ and
\begin{equation*}
H_j:=I_r+\sum_{k\ne j}\frac{v_kv_k^\top}{d_j-d_k}.
\end{equation*}
Then $d_j$ is an eigenvalue of $A$ if and only if
\begin{equation}\label{main-eq:pole_bordered}
\det\begin{pmatrix}H_j&-v_j\\-v_j^\top&0\end{pmatrix}=0.
\end{equation}
Its multiplicity equals the nullity of this matrix. In rank one, $a_{j,1}\ne0$ excludes $d_j$ from the spectrum of $A$.
\end{proposition}
\begin{proof}
Write $f=\sum_kc_ke_k$ and $z_\ell=\langle f,u_\ell\rangle_\beta$. At $\lambda=d_j$, the coordinate equations are
\begin{equation*}
v_j^\top z=0,\qquad
c_k=-\frac{v_k^\top z}{d_j-d_k}\quad(k\ne j).
\end{equation*}
Substitution into $z=\sum_kc_kv_k$ gives $H_jz=c_jv_j$. Thus every eigenfunction determines a null vector $(z,c_j)^\top$ of the bordered matrix in \eqref{main-eq:pole_bordered}.

Conversely, let $(z,c)^\top$ belong to that nullspace and define
\begin{equation*}
f=ce_j-\sum_{k\ne j}\frac{v_k^\top z}{d_j-d_k}e_k.
\end{equation*}
Since $\delta_j:=\inf_{k\ne j}|d_j-d_k|>0$ and $\sum_k\|v_k\|^2<\infty$,
\begin{equation*}
\sum_{k\ne j}\frac{|v_k^\top z|^2}{|d_j-d_k|^2}
\le\delta_j^{-2}\|z\|^2\sum_{k\ne j}\|v_k\|^2<\infty.
\end{equation*}
The same bound justifies the defining series for $H_j$. The overlaps of the reconstructed function satisfy
\begin{equation*}
(\langle f,u_\ell\rangle_\beta)_{\ell=1}^r
=cv_j-(H_j-I_r)z=z.
\end{equation*}
Together with $v_j^\top z=0$, this proves every coordinate of $Af=d_jf$. If $f=0$, then the overlap identity gives $z=0$ and its $j$th coordinate gives $c=0$. Hence the two constructions are inverse linear maps, proving the equivalence and equality of multiplicities. In rank one the bordered determinant is $-a_{j,1}^2$.
\end{proof}

\subsection{Real covariance and quadratic-form limits}
\label{main-subsec:HilbertCLT_KL}
The Hilbert-space central limit theorem and the Gaussian squared-norm expansion are recalled as Theorems~\ref{supp-thm:HilbertCLT} and \ref{supp-thm:KLquadratic}. For characteristic functions they apply on the real space \eqref{main-eq:Hermitian_space}.
\begin{myrem}\label{main-rem:complex_H}
Let $\zeta$ be a centered square-integrable random element of $\mathcal H_\beta$ and put $C(s,t)=\EE[\zeta(s)\overline{\zeta(t)}]$. For $a\in\mathcal H_\beta$, $\langle a,\zeta\rangle_\beta$ is real, and the real covariance operator is
\begin{equation*}
(\mathcal C a)(s)=\EE[\langle a,\zeta\rangle_\beta\zeta(s)]
=\int_{\mathbb R}C(s,t)a(t)\varphi_\beta(t)\,dt.
\end{equation*}
In the normal cases below $C$ is real and preserves parity. By \eqref{main-eq:parity_isometry}, $\mathcal C$ is unitarily equivalent to the operator with kernel $C$ on $L^2(\varphi_\beta;\mathbb R)$. The latter eigenvalues are the weights of the chi-square series. Viewing the whole complex space as real and retaining only its complex covariance kernel would not, by itself, justify this conclusion.
\end{myrem}

\section{Projection of degenerate U-statistics under estimated parameters}
\label{main-sec:data_domain_projection}

The notation $\Phi$ and $f_0$ is as in Section~\ref{main-subsec:Gaussian_kernel_Krho}. Here $L^2(\Phi)$ is a real Hilbert space.  A measurable symmetric kernel $K$ is called canonical under $\Phi$ if $\mathbb E[K(x,X)]=0$ for $\Phi$-almost every $x$, where $X\sim\Phi$.  We first treat finite spectral expansions of canonical kernels.  For fully standardized residuals, the empirical sums of the eigenfunctions have the same first-order limits as the empirical sums of their projections on the orthogonal complement of the two-dimensional normal score space.  The infinite-rank passage is then stated separately, through a tail condition and a sufficient summability condition.  The one-parameter analogues and additional kernel examples are given in Sections~\ref{supp-sec:one_parameter_projection} and \ref{supp-sec:infinite_rank_and_examples} of the Supplementary Material.

Let $X_1,X_2,\ldots$ be i.i.d. standard normal variables.  We write
\begin{equation*}
\langle f,g\rangle_\Phi
:=
\int_{-\infty}^{\infty}f(x)g(x)f_0(x)\,dx,
\qquad
\|f\|_\Phi^2:=\langle f,f\rangle_\Phi .
\end{equation*}
The normal score directions used below are
\begin{equation*}
q_1(x):=x,
\qquad
q_2(x):=\frac{x^2-1}{\sqrt2}.
\end{equation*}
They are orthonormal in $L^2(\Phi)$.  Let
\begin{align}
\Pi_\mu f
&:=f-\langle f,q_1\rangle_\Phi q_1, \label{main-eq:data_Pi_mu}\\
\Pi_\sigma f
&:=f-\langle f,q_2\rangle_\Phi q_2, \label{main-eq:data_Pi_sigma}\\
\Pi_{\mu,\sigma}f
&:=f-\langle f,q_1\rangle_\Phi q_1-\langle f,q_2\rangle_\Phi q_2.
\nonumber\end{align}
For the main statements in this section we take $n\ge4$ and use the fully standardized residuals
\begin{equation*}
R_{n,i}:=\frac{X_i-\overline X_n}{S_n},
\qquad
\overline X_n:=\frac1n\sum_{i=1}^n X_i,
\qquad
S_n^2:=\frac1n\sum_{i=1}^n(X_i-\overline X_n)^2.
\end{equation*}

\subsection{Projection of empirical sums}

The functions in the next lemma are not assumed to be polynomials.  The differentiability and growth
conditions are used only for the Taylor expansion after standardization.

\begin{lemma}\label{main-lem:data_projection_empirical_sums}
Let $r_0,\ldots,r_M$ be real-valued functions on $\mathbb R$.  Assume that each $r_k$ is twice
continuously differentiable and that $r_k$, $r_k'$ and $r_k''$ have at most polynomial growth.  Then,
for $k=0,\ldots,M$,
\begin{equation}
\label{main-eq:data_projection_musigma}
\frac1{\sqrt n}\sum_{i=1}^n r_k(R_{n,i})
=
\frac1{\sqrt n}\sum_{i=1}^n(\Pi_{\mu,\sigma}r_k)(X_i)+o_{\mathbb P}(1).
\end{equation}
The remainders are uniform over $k=0,\ldots,M$.  Moreover,
\begin{equation*}
\frac1n\sum_{i=1}^n r_k^2(R_{n,i})
\xrightarrow[n\to\infty]{\mathbb P}
\mathbb E[r_k^2(X)],
\qquad k=0,\ldots,M.
\end{equation*}
\end{lemma}

The proof is given in Section~\ref{supp-subsec:proof_projection_empirical} of the Supplementary Material.

\subsection{Finite-rank degenerate kernels}

Let
\begin{equation*}
K_M(x,y)=\sum_{k=0}^{M}\gamma_k r_k(x)r_k(y),
\end{equation*}
where $r_0,\ldots,r_M$ are centered and orthonormal in $L^2(\Phi)$ and satisfy the assumptions
of Lemma~\ref{main-lem:data_projection_empirical_sums}.  Define
\begin{equation*}
U_n^{(\mu,\sigma)}(K_M)
:=
\frac{2}{n(n-1)}
\sum_{1\le i<j\le n}K_M(R_{n,i},R_{n,j}).
\end{equation*}

\begin{theorem}\label{main-thm:data_finite_rank_projection_limit}
Let
\begin{equation*}
\Gamma_M^{(\mu,\sigma)}
:=
\left(
\langle \Pi_{\mu,\sigma}r_i,\Pi_{\mu,\sigma}r_j\rangle_\Phi
\right)_{0\le i,j\le M},
\qquad
\Lambda_M:=\operatorname{diag}(\gamma_0,\ldots,\gamma_M).
\end{equation*}
Then
\begin{equation}
\label{main-eq:data_finite_rank_limit}
nU_n^{(\mu,\sigma)}(K_M)
\xrightarrow[n\to\infty]{d}
G_M^{\top}\Lambda_MG_M-\operatorname{tr}(\Lambda_M),
\end{equation}
where $G_M$ is centered Gaussian in $\mathbb R^{M+1}$ with covariance matrix
$\Gamma_M^{(\mu,\sigma)}$.
If $\omega_{M,0},\ldots,\omega_{M,M}$ are the eigenvalues of
\begin{equation}
\label{main-eq:data_omega_matrix}
\Gamma_M^{(\mu,\sigma)\,1/2}\Lambda_M\Gamma_M^{(\mu,\sigma)\,1/2},
\end{equation}
then
\begin{equation}
\label{main-eq:data_finite_rank_limit_diagonal}
nU_n^{(\mu,\sigma)}(K_M)
\xrightarrow[n\to\infty]{d}
\sum_{j=0}^{M}\omega_{M,j}N_j^2-
\sum_{k=0}^{M}\gamma_k,
\end{equation}
where $N_0,\ldots,N_M$ are independent standard normal variables.
\end{theorem}

\begin{proof}
Set
\begin{equation*}
S_{n,k}:=\frac1{\sqrt n}\sum_{i=1}^n r_k(R_{n,i}),
\qquad
Q_{n,k}:=\frac1n\sum_{i=1}^n r_k^2(R_{n,i}).
\end{equation*}
The identity
\begin{equation}
\label{main-eq:data_exact_U_identity}
nU_n^{(\mu,\sigma)}(K_M)
=
\frac{n}{n-1}
\sum_{k=0}^{M}\gamma_k\{S_{n,k}^2-Q_{n,k}\}
\end{equation}
follows by expanding the square of $\sum_i r_k(R_{n,i})$.  By
Lemma~\ref{main-lem:data_projection_empirical_sums} and the multivariate central limit theorem,
\begin{equation*}
(S_{n,0},\ldots,S_{n,M})^{\top}
\xrightarrow[n\to\infty]{d}
G_M,
\end{equation*}
with covariance matrix $\Gamma_M^{(\mu,\sigma)}$.  Also $Q_{n,k}\xrightarrow[n\to\infty]{\mathbb P}1$.
Substitution in \eqref{main-eq:data_exact_U_identity} gives \eqref{main-eq:data_finite_rank_limit}.  Diagonalizing
\eqref{main-eq:data_omega_matrix} gives \eqref{main-eq:data_finite_rank_limit_diagonal}.
\end{proof}

\subsection{Passage from finite to infinite rank}
\label{main-subsec:data_infinite_rank}
Let $K$ be a measurable symmetric canonical kernel with expansion
\begin{equation*}
K(x,y)=\sum_{k=0}^{\infty}\gamma_k r_k(x)r_k(y)
\end{equation*}
in $L^2(\Phi\otimes\Phi)$. The coefficients $\gamma_k$ are real. The functions $r_k$ are real, centered and orthonormal in $L^2(\Phi)$, and each $r_k,r_k',r_k''$ satisfies the smoothness and polynomial-growth assumptions of Lemma~\ref{main-lem:data_projection_empirical_sums}. Put
\begin{equation*}
K_M(x,y):=\sum_{k=0}^{M}\gamma_k r_k(x)r_k(y),\qquad
K_{M,L}:=K_L-K_M\quad(L>M).
\end{equation*}
We evaluate a measurable representative of $K$ at the residual pairs, and define $U_n^{(\mu,\sigma)}(K)$ by the same off-diagonal sum as for $K_M$. For $n\ge4$ the law of each pair of distinct fully standardized residuals is absolutely continuous with respect to $\Phi\otimes\Phi$. Indeed, the residual vector is uniform on the sphere of radius $\sqrt n$ in the $(n-1)$-dimensional subspace orthogonal to the constant vector. Two distinct coordinate functionals are linearly independent on this subspace. Their joint density is given in \eqref{supp-eq:residual_pair_density}. Thus the choice of representative has no effect almost surely. Section~\ref{supp-subsec:residual_pairs} gives the fixed-sample-size measure argument used below; no uniformly bounded pair-density ratio is assumed.

The approximation condition is, for every $\varepsilon>0$,
\begin{equation}
\label{main-eq:data_tail_condition_probability}
\lim_{M\to\infty}\limsup_{n\to\infty}
\PP\!\left(\left|nU_n^{(\mu,\sigma)}(K-K_M)\right|>\varepsilon\right)=0.
\end{equation}
This condition concerns the whole tail for each sample size, before taking the large-sample limit. Convergence for each fixed finite tail alone does not imply it.
For $k\ge0$ define
\begin{equation*}
S_{n,k}:=n^{-1/2}\sum_{i=1}^n r_k(R_{n,i}),\qquad
Q_{n,k}:=n^{-1}\sum_{i=1}^n r_k^2(R_{n,i}),
\end{equation*}
and
\begin{equation*}
A_k(n_0):=\sup_{n\ge n_0}\{\EE[S_{n,k}^2]+\EE[Q_{n,k}]\},\qquad n_0\ge4.
\end{equation*}
\begin{proposition}\label{main-prop:data_tail_sufficient_main}
If, for some $n_0\ge4$,
\begin{equation}
\label{main-eq:data_tail_sufficient_condition}
\sum_{k=0}^{\infty}|\gamma_k|A_k(n_0)<\infty,
\end{equation}
then \eqref{main-eq:data_tail_condition_probability} holds and
\begin{equation}
\label{main-eq:data_uniform_L1_tail}
\sup_{n\ge n_0}\EE\left|nU_n^{(\mu,\sigma)}(K-K_M)\right|
\le2\sum_{k>M}|\gamma_k|A_k(n_0).
\end{equation}
For each fixed $n\ge n_0$, the full tail is the $L^1$ limit of the finite tails as $L\to\infty$.
\end{proposition}
\begin{proof}
The exact finite-rank identity \eqref{main-eq:data_exact_U_identity} gives
\begin{equation*}
\sup_{n\ge n_0}\EE\left|nU_n^{(\mu,\sigma)}(K_{M,L})\right|
\le2\sum_{k=M+1}^{L}|\gamma_k|A_k(n_0).
\end{equation*}
Thus the finite tails are Cauchy in $L^1$, uniformly over $n\ge n_0$. Their limits equal the statistics evaluated at $K-K_M$. To verify this identification, $L^2(\Phi\otimes\Phi)$ convergence implies convergence in probability under each fixed residual-pair law by absolute continuity. The sum contains finitely many pairs for fixed $n$, so $nU_n^{(\mu,\sigma)}(K_{M,L})$ converges in probability to $nU_n^{(\mu,\sigma)}(K-K_M)$. Uniqueness of the limit proves the identification. Passing to $L\to\infty$ yields \eqref{main-eq:data_uniform_L1_tail}; Markov's inequality then gives \eqref{main-eq:data_tail_condition_probability}.
\end{proof}

For fully standardized normal observations, the expectation condition can be checked by a derivative estimate. No parity assumption is required.
\begin{proposition}\label{main-prop:derivative_tail_criterion}
There is a numerical constant $C$ such that every centered, continuously differentiable $r$ with $r,r'\in L^2(\Phi)$ satisfies
\begin{equation}\label{main-eq:Gaussian_derivative_bound}
\sup_{n\ge4}\left\{\EE\left[\left(n^{-1/2}\sum_{i=1}^n r(R_{n,i})\right)^2\right]
+\EE\left[n^{-1}\sum_{i=1}^n r^2(R_{n,i})\right]\right\}
\le C\{\|r\|_\Phi^2+\|r'\|_\Phi^2\}.
\end{equation}
Consequently, under the smoothness assumptions of this subsection, the condition
\begin{equation}\label{main-eq:derivative_summability}
\sum_{k\ge0}|\gamma_k|\{1+\|r_k'\|_\Phi^2\}<\infty
\end{equation}
implies \eqref{main-eq:data_tail_sufficient_condition} and \eqref{main-eq:data_tail_condition_probability}.
\end{proposition}
\begin{proof}
Put $F_n=n^{-1/2}\sum_i r(R_{n,i})$ and $Q_n=n^{-1}\sum_i r^2(R_{n,i})$. Lemma~\ref{supp-lem:residual_density_rate} gives constants $C_0,C_1$, independent of $n\ge4$, such that $f_n\le C_0f_0$ and $\|f_n/f_0-1\|_\Phi\le C_1/n$. Since $\int r\,d\Phi=0$,
\begin{equation*}
|\EE F_n|\le\sqrt n\,\|r\|_\Phi\|f_n/f_0-1\|_\Phi
\le C_1n^{-1/2}\|r\|_\Phi,
\qquad \EE Q_n\le C_0\|r\|_\Phi^2.
\end{equation*}
The residual Jacobian and independence of the normal radius and direction give
\begin{equation*}
\operatorname{Var}(F_n)\le\EE\|\nabla F_n\|^2
\le\EE[S_n^{-2}]\int |r'(x)|^2f_n(x)\,dx
\le4C_0\|r'\|_\Phi^2,
\end{equation*}
because $\EE[S_n^{-2}]=n/(n-3)\le4$. Section~\ref{supp-sec:full_derivative_tail} derives the Jacobian and proves the Sobolev approximation needed to apply Poincar\'e at zero sample variance. Therefore
\begin{equation*}
\EE F_n^2+\EE Q_n
\le4C_0\|r'\|_\Phi^2+(C_0+C_1^2/n)\|r\|_\Phi^2,
\end{equation*}
which proves \eqref{main-eq:Gaussian_derivative_bound}. Since $\|r_k\|_\Phi=1$, condition \eqref{main-eq:derivative_summability} implies \eqref{main-eq:data_tail_sufficient_condition}; Proposition~\ref{main-prop:data_tail_sufficient_main} gives \eqref{main-eq:data_tail_condition_probability}.
\end{proof}

Write $(T_Kf)(x)=\int K(x,y)f(y)\,d\Phi(y)$.
\begin{proposition}\label{main-prop:data_infinite_rank_limit_main}
Assume \eqref{main-eq:data_tail_condition_probability} and $\sum_k|\gamma_k|<\infty$. Let $\{\omega_j\}_{j\ge0}$ be the nonzero eigenvalues of $\Pi_{\mu,\sigma}T_K\Pi_{\mu,\sigma}$, counted with multiplicity and completed by zeros when necessary. Then
\begin{equation}
\label{main-eq:data_infinite_rank_limit_main}
nU_n^{(\mu,\sigma)}(K)\xrightarrow[n\to\infty]{d}
\sum_{j=0}^{\infty}\omega_jN_j^2-\sum_{k=0}^{\infty}\gamma_k,
\end{equation}
where the $N_j$ are independent standard normal variables. The series converges absolutely almost surely and in $L^1$.
\end{proposition}
The proof is given in Section~\ref{supp-subsec:proof_infinite_rank_limit} of the Supplementary Material.

\begin{proposition}\label{main-prop:aligned_score_deletion}
Let $T_K$ have the spectral expansion
\begin{equation*}
T_K=
\sum_{k=0}^{\infty}\gamma_k r_k\otimes r_k
\end{equation*}
in $L^2(\Phi)$, with $\sum_k|\gamma_k|<\infty$.  Let $J$ be a finite index set and let $\Pi_J$ be the orthogonal projection onto the orthogonal complement of $\mathrm{span}\{r_j:j\in J\}$.  Then
\begin{equation*}
\Pi_JT_K\Pi_J
=
\sum_{k\notin J}\gamma_k r_k\otimes r_k.
\end{equation*}
The nonzero spectrum is therefore obtained by removing the indexed eigenvalues, with multiplicity.
\end{proposition}

\begin{proof}
The assertion follows from $\Pi_J r_j=0$ for $j\in J$ and $\Pi_J r_k=r_k$ for $k\notin J$.  The trace-norm convergence of the series permits termwise multiplication by the bounded projection $\Pi_J$.
\end{proof}

Section~\ref{supp-sec:tail_full_Krho} gives explicit estimates for the odd Gaussian eigenfunctions. Section~\ref{supp-sec:one_parameter_Krho_tail} treats the one-parameter cases. Corollary~\ref{main-cor:weighted_kernel_full_tail} below verifies the condition for the full kernel without using explicit eigenfunctions of its centered even part.

\subsection{Data-domain and frequency-domain factorization}
\label{main-subsec:data_frequency_factorization}

We keep $L^2(\Phi)$ real and define $\mathcal H_\theta$ as in \eqref{main-eq:Hermitian_space}, replacing $\varphi_\beta$ by $\theta$. The target is the standard normal law throughout this subsection. The inner product on $\mathcal H_\theta$ is real.

Let $\theta$ be positive almost everywhere, even and integrable, and equip $\mathcal H_\theta$ with
\begin{equation*}
\langle f,g\rangle_\theta:=\int_{-\infty}^{\infty}f(t)\overline{g(t)}\theta(t)\,dt.
\end{equation*}
Set
\begin{equation*}
\xi_x(t):=\exp\!\left\{itx\right\}-\phi_0(t).
\end{equation*}
Since $\theta$ is even and $\phi_0$ is real and even, the kernel of Proposition~\ref{main-prop:weighted_L2_kernel_rep} satisfies
\begin{equation*}
K_\theta(x,y)=\langle \xi_x,\xi_y\rangle_\theta.
\end{equation*}
Let $T_{K_\theta}$ be the integral operator
\begin{equation*}
(T_{K_\theta}f)(x):=\int_{-\infty}^{\infty}K_\theta(x,y)f(y)\,d\Phi(y).
\end{equation*}
Define $L_\theta:L^2(\Phi;\mathbb R)\to\mathcal H_\theta$ by
\begin{equation*}
(L_\theta f)(t):=\mathbb E_\Phi[f(X)\xi_X(t)].
\end{equation*}
The operator is Hilbert--Schmidt.  Indeed,
\begin{equation*}
\int_{-\infty}^{\infty}\int_{-\infty}^{\infty}
|\xi_x(t)|^2\,d\Phi(x)\,\theta(t)\,dt
\le 4\|\theta\|_{L^1}<\infty.
\end{equation*}
In particular, $L_\theta$ is compact and bounded.
For full standardization put
\begin{equation*}
L_{\theta,\mu,\sigma}:=L_\theta\Pi_{\mu,\sigma}.
\end{equation*}
The projected summand is
\begin{equation*}
\xi_x^{(\mu,\sigma)}:=\xi_x-q_1(x)L_\theta q_1-q_2(x)L_\theta q_2.
\end{equation*}

\begin{proposition}
\label{main-prop:data_frequency_factorization}
The data-domain and Fourier-domain operators satisfy
\begin{equation}
\label{main-eq:factorization_data_side}
L_{\theta,\mu,\sigma}^*L_{\theta,\mu,\sigma}
=
\Pi_{\mu,\sigma}T_{K_\theta}\Pi_{\mu,\sigma}
\end{equation}
and
\begin{equation}
\label{main-eq:factorization_frequency_side}
L_{\theta,\mu,\sigma}L_{\theta,\mu,\sigma}^*
=
\mathcal C_{\theta}^{(\mu,\sigma)},
\end{equation}
where $\mathcal C_{\theta}^{(\mu,\sigma)}$ is the real covariance operator in $\mathcal H_\theta$ generated by the projected Fourier summand.  The two compact positive operators in \eqref{main-eq:factorization_data_side} and \eqref{main-eq:factorization_frequency_side} have the same positive eigenvalues, counted with multiplicity.
\end{proposition}

The proof is given in Section~\ref{supp-subsec:proof_factorization} of the Supplementary Material.

For the normal target,
\begin{equation*}
L_\theta q_1(t)=it\phi_0(t),
\qquad
L_\theta q_2(t)=-\frac{t^2}{\sqrt2}\phi_0(t).
\end{equation*}
Thus
\begin{equation*}
\xi_x^{(\mu,\sigma)}(t)
=
\exp\!\left\{itx\right\}-\phi_0(t)-it\phi_0(t)x+\frac{t^2}{2}\phi_0(t)(x^2-1),
\end{equation*}
which is the summand in the full empirical characteristic-function expansion in Section~\ref{main-sec:full_studentization}.  The one-parameter identities are obtained by retaining the corresponding score direction.

The factorization also verifies the derivative condition for the full kernel.
\begin{cor}\label{main-cor:weighted_kernel_full_tail}
Suppose that $\theta$ is positive almost everywhere, even and satisfies
\begin{equation}\label{main-eq:weight_fourth_moment}
\int_{\mathbb R}(1+t^4)\theta(t)\,dt<\infty.
\end{equation}
Then $K_\theta$ satisfies the hypotheses of Proposition~\ref{main-prop:data_infinite_rank_limit_main}, including its tail condition. Its eigenfunctions corresponding to positive eigenvalues have bounded, twice continuously differentiable representatives, and
\begin{equation}\label{main-eq:full_kernel_derivative_trace}
\sum_k\gamma_k\|r_k'\|_\Phi^2\le\int_{\mathbb R}t^2\theta(t)\,dt,
\qquad
\sum_k\gamma_k=\int_{\mathbb R}\{1-\phi_0^2(t)\}\theta(t)\,dt.
\end{equation}
In particular, the direct limit theorem applies to the entire Gaussian kernel $K_\rho$ for every $0<\rho<1$.
\end{cor}
\begin{proof}
The factorization $T_{K_\theta}=L_\theta^*L_\theta$ makes this operator positive and trace class. Parseval's identity and Tonelli's theorem give
\begin{equation*}
\operatorname{tr}(T_{K_\theta})=\|L_\theta\|_{\mathrm{HS}}^2
=\int_{\mathbb R}\EE_\Phi|\xi_X(t)|^2\theta(t)\,dt
=\int_{\mathbb R}\{1-\phi_0^2(t)\}\theta(t)\,dt.
\end{equation*}
For an orthonormal eigenfamily $r_k$ with $\gamma_k>0$, set $u_k=\gamma_k^{-1/2}L_\theta r_k$. Then
\begin{equation*}
\langle u_k,u_\ell\rangle_\theta
=(\gamma_k\gamma_\ell)^{-1/2}\langle r_k,T_{K_\theta}r_\ell\rangle_\Phi
=\boldsymbol1_{\{k=\ell\}},
\qquad
r_k(x)=\gamma_k^{-1/2}\langle u_k,\xi_x\rangle_\theta.
\end{equation*}
The second identity first holds $\Phi$-almost everywhere and defines a representative for all $x$. The maps
\begin{equation*}
\partial_x\xi_x(t)=it\exp\!\left\{itx\right\},\qquad
\partial_x^2\xi_x(t)=-t^2\exp\!\left\{itx\right\}
\end{equation*}
are continuous in $\mathcal H_\theta$ by dominated convergence under \eqref{main-eq:weight_fourth_moment}. Their squared norms are $\int t^2\theta(t)\,dt$ and $\int t^4\theta(t)\,dt$, respectively, independently of $x$. Since also $\|\xi_x\|_\theta\le2\|\theta\|_{L^1}^{1/2}$, the representative of $r_k$ and its first two derivatives are bounded. Bessel's inequality now gives
\begin{equation*}
\sum_k\gamma_k|r_k'(x)|^2
=\sum_k|\langle u_k,\partial_x\xi_x\rangle_\theta|^2
\le\|\partial_x\xi_x\|_\theta^2
=\int_{\mathbb R}t^2\theta(t)\,dt.
\end{equation*}
Integration with respect to $\Phi$ proves the derivative bound in \eqref{main-eq:full_kernel_derivative_trace}. Canonicality implies $T_{K_\theta}1=0$, so
$\gamma_k\langle r_k,1\rangle_\Phi=\langle r_k,T_{K_\theta}1\rangle_\Phi=0$.
Thus the eigenfunctions are centered and satisfy the smoothness assumptions of Lemma~\ref{main-lem:data_projection_empirical_sums}. The trace and derivative bounds imply \eqref{main-eq:derivative_summability}. Propositions~\ref{main-prop:derivative_tail_criterion}, \ref{main-prop:data_tail_sufficient_main} and \ref{main-prop:data_infinite_rank_limit_main} give the full-kernel limit. Gaussian weights satisfy \eqref{main-eq:weight_fourth_moment}.
\end{proof}

\begin{myrem}
\label{main-rem:factorization_normalization}
If the weight $\theta_\rho$ is used instead of $\varphi_\beta$, the eigenvalues in
Proposition~\ref{main-prop:data_frequency_factorization} are multiplied by $(1-\rho)^{-1}$; see
Remark~\ref{main-rem:scaling_eigs}.  Thus a frequency-domain eigenvalue $\lambda$ and the corresponding
data-domain eigenvalue $\omega$ satisfy $\lambda=(1-\rho)\omega$ when the former is written in the
normalized EP scale.
\end{myrem}

The odd data-domain restriction of $K_\rho$, including the tail verification and the equivalence with the odd frequency-domain secular equation, is worked out in Section~\ref{supp-sec:K_rho_data_domain} of the Supplementary Material.

The one-parameter normalizations, where only the mean or only the variance is estimated, are recorded in Section~\ref{supp-sec:one_param_studentization} of the Supplementary Material.  The main text henceforth concentrates on the fully standardized statistic.

\section{Full standardization: mean and variance estimated}
\label{main-sec:full_studentization}

We finally consider the standard EP/BHEP statistic based on standardized residuals, i.e.\ the case where both $\mu$ and
$\sigma^2$ are estimated.

\subsection{Setup}

Assume $X_1,\dots,X_n$ are i.i.d.\ $\mathcal N(\mu,\sigma^2)$ with $\sigma>0$.
Since the statistic is affine invariant, we may work under $H_0$ with $\mu=0$ and $\sigma^2=1$.
Define the usual estimators
\begin{equation*}
\overline{X}_n:=\frac{1}{n}\sum_{j=1}^n X_j,
\qquad
S_n^2:=\frac{1}{n}\sum_{j=1}^n (X_j-\overline{X}_n)^2
=\frac{1}{n}\sum_{j=1}^n X_j^2-\overline{X}_n^{\,2},
\end{equation*}
and the standardized residuals
\begin{equation*}
V_{n,j}:=\frac{X_j-\overline{X}_n}{S_n},
\qquad j=1,\dots,n.
\end{equation*}
Let $\psi_n$ be the empirical characteristic function of $\{V_{n,j}\}$,
\begin{equation*}
\psi_n(t):=\frac{1}{n}\sum_{j=1}^n \exp\!\left\{itV_{n,j}\right\},\qquad t\in\mathbb R,
\end{equation*}
and define the fully studentized EP/BHEP statistic
\begin{equation*}
T_{n,\beta}^{(\mu,\sigma)}
:=
n\int_{-\infty}^{\infty}\bigl|\psi_n(t)-\phi_0(t)\bigr|^2\,\varphi_\beta(t)\,dt.
\end{equation*}

\subsection{Limit law and eigenvalue equations}

\begin{theorem}\label{main-thm:mu_sigma_main}
Fix $\rho\in(0,1)$ and set $\beta=\sqrt{\rho}/(1-\rho)$.  Under
$H_0:X_i\stackrel{\mathrm{i.i.d.}}{\sim}\mathcal N(0,1)$,
\begin{equation}
\label{main-eq:mu_sigma_limit_series}
T_{n,\beta}^{(\mu,\sigma)}
\xrightarrow[n\to\infty]{d}
T_{\infty}^{(\mu,\sigma)}
:=
\int_{-\infty}^{\infty}|Z(t)|^2\varphi_\beta(t)\,dt
\stackrel{d}{=}
\sum_{n=0}^{\infty}\lambda_n^{(\mu,\sigma)}(\rho)N_n^2,
\end{equation}
where $Z$ is a centered real Gaussian element of $\mathcal H_\beta$, the weights are the positive eigenvalues of the real operator $A^{(\mu,\sigma)}$ defined below, and $\{N_n\}_{n\ge0}$ are i.i.d.
standard normal variables. The kernel $\EE[Z(s)\overline{Z(t)}]$ is
\begin{equation}
\label{main-eq:mu_sigma_kernel}
K^{(\mu,\sigma)}(s,t)
=
\exp\!\left\{-\frac{(s-t)^2}{2}\right\}
-
\left(1+st+\frac{s^2t^2}{2}\right)
\exp\!\left\{-\frac{s^2+t^2}{2}\right\}.
\end{equation}
Its covariance operator on $\mathcal H_\beta$ is unitarily equivalent, through $J_\beta$, to the operator on $L^2(\varphi_\beta;\mathbb R)$
\begin{equation}
\label{main-eq:mu_sigma_operator_form}
A^{(\mu,\sigma)}
=
A_0-\phi_0^*\otimes\phi_0^*-\phi_1^*\otimes\phi_1^*-\phi_2^*\otimes\phi_2^*.
\end{equation}
The parity subspaces are invariant.  The odd eigenvalues $\{\lambda_{m,\mathrm{odd}}^{(\mu,\sigma)}(\rho)\}_{m\ge1}$ are the unique roots in
\begin{equation*}
\lambda\in\bigl((1-\rho)\rho^{2m+1},(1-\rho)\rho^{2m-1}\bigr),
\qquad m\ge1,
\end{equation*}
of
\begin{equation*}
F_{\mathrm{odd}}(\lambda)
:=
1+\sum_{j=0}^{\infty}\frac{a_{2j+1,1}^2(\rho)}{\lambda-(1-\rho)\rho^{2j+1}}=0.
\end{equation*}
The even eigenvalues away from the poles $\kappa_j:=(1-\rho)\rho^{2j}$ are the positive roots of
\begin{equation}
\label{main-eq:mu_sigma_even_det}
D_{\mathrm{even}}(\lambda)
:=
\det
\begin{pmatrix}
1+G_{22}(\lambda) & G_{20}(\lambda)\\
G_{20}(\lambda) & 1+G_{00}(\lambda)
\end{pmatrix}=0,
\end{equation}
where $\lambda\notin\{0\}\cup\{\kappa_j:j\ge0\}$ and, for $p,q\in\{0,2\}$,
\begin{equation*}
G_{pq}(\lambda)
:=
\sum_{j=0}^{\infty}
\frac{a_{2j,p}(\rho)a_{2j,q}(\rho)}{\lambda-(1-\rho)\rho^{2j}}.
\end{equation*}
For the values at the poles, put
\begin{equation*}
b_j:=\frac{\rho}{\sqrt2}-\frac{\sqrt2(1-\rho^2)}{\rho}j,\qquad
E_j(\rho):=1+b_j^2+
\sum_{k\ne j}\frac{a_{2k,0}^2(\rho)(b_k-b_j)^2}{\kappa_j-\kappa_k}.
\end{equation*}
Then $\kappa_j$ is an even eigenvalue if and only if $E_j(\rho)=0$. Together with \eqref{main-eq:mu_sigma_even_det}, this criterion describes the entire positive even spectrum. Multiplicities are those of Lemma~\ref{main-lem:finite_rank_secular} off the poles and Proposition~\ref{main-prop:pole_criterion} at the poles.

Let $w_j(\rho):=a_{2j,0}^2(\rho)$ and, for $r=0,1,2$,
\begin{equation*}
S_r(\lambda):=\sum_{j=0}^{\infty}\frac{j^r w_j(\rho)}{\lambda-\kappa_j},\qquad 0^0:=1.
\end{equation*}
Away from the poles, \eqref{main-eq:mu_sigma_even_det} is equivalent to
\begin{align}
\label{main-eq:mu_sigma_even_scalar}
D_{\mathrm{even}}(\lambda)
&=1+\left(1+\frac{\rho^2}{2}\right)S_0(\lambda)
-2(1-\rho^2)S_1(\lambda)\nonumber\\
&\quad+2\frac{(1-\rho^2)^2}{\rho^2}
\left\{S_2(\lambda)(1+S_0(\lambda))-S_1^2(\lambda)\right\}=0.
\end{align}
If $\{\lambda_{m,\mathrm{even}}^{(\mu,\sigma)}(\rho)\}_{m\ge1}$ are ordered decreasingly, then
\begin{equation}
\label{main-eq:mu_sigma_even_bracket}
(1-\rho)\rho^{2m-2}>
\lambda_{m,\mathrm{even}}^{(\mu,\sigma)}(\rho)>
(1-\rho)\rho^{2m+2},
\qquad m\ge1.
\end{equation}
\end{theorem}

\begin{proof}
Proposition~\ref{supp-prop:supp_full_ecf_limit} gives the expansion
\begin{equation*}
\sqrt n(\psi_n-\phi_0)
=n^{-1/2}\sum_{j=1}^n\zeta_{\mu,\sigma}(\cdot;X_j)+o_{\PP}(1)
\quad\text{in }\mathcal H_\beta,
\end{equation*}
where
\begin{equation*}
\zeta_{\mu,\sigma}(t;x)
=\exp\!\left\{itx\right\}-\phi_0(t)-it\phi_0(t)x
+\tfrac12 t^2\phi_0(t)(x^2-1).
\end{equation*}
The summands are centered and square integrable. Their covariance is \eqref{main-eq:mu_sigma_kernel}, by the normal characteristic-function identities. The real Hilbert-space central limit theorem, Remark~\ref{main-rem:complex_H} and the Gaussian squared-norm expansion give \eqref{main-eq:mu_sigma_limit_series}.

The operator identity \eqref{main-eq:mu_sigma_operator_form} follows from \eqref{main-eq:mu_sigma_kernel}.  Since $\phi_1^*$ is odd and $\phi_0^*,\phi_2^*$ are even, the odd and even subspaces are invariant.  Lemma~\ref{main-lem:finite_rank_secular} gives the off-pole equations. In the odd subspace the nonzero overlaps exclude the poles by Proposition~\ref{main-prop:pole_criterion}. In the even subspace, use the coordinate order $(0,2)$ and put $v_j=a_{2j,0}(1,b_j)^\top$ and $q_j=(-b_j,1)^\top$. Then $a_{2j,0}\ne0$ and $v_j^\perp=\operatorname{span}\{q_j\}$. With $H_j$ from Proposition~\ref{main-prop:pole_criterion}, the bordered equations are
\begin{equation*}
v_j^\top z=0,\qquad H_jz=cv_j.
\end{equation*}
A nonzero solution must have $z=tq_j$ with $t\ne0$. Such a solution exists exactly when $q_j^\top H_jq_j=0$, because in dimension two this is equivalent to $H_jq_j\in\operatorname{span}\{v_j\}$. Directly,
\begin{equation*}
q_j^\top H_jq_j
=1+b_j^2+\sum_{k\ne j}\frac{a_{2k,0}^2(b_k-b_j)^2}{\kappa_j-\kappa_k}
=E_j(\rho).
\end{equation*}
This proves the criterion at the poles. To obtain the scalar equation away from them, put $a=\rho/\sqrt2$ and $b=\sqrt2(1-\rho^2)/\rho$, so that $b_j=a-b\,j$. Relation \eqref{main-eq:a22_affine_relation_global} gives
\begin{equation*}
G_{00}=S_0,\qquad G_{20}=aS_0-bS_1,\qquad
G_{22}=a^2S_0-2abS_1+b^2S_2.
\end{equation*}
Substitution into $(1+G_{00})(1+G_{22})-G_{20}^2$ gives
$1+(1+a^2)S_0-2abS_1+b^2\{S_2(1+S_0)-S_1^2\}$, which is \eqref{main-eq:mu_sigma_even_scalar}.

It remains to prove \eqref{main-eq:mu_sigma_even_bracket}.  Put
\begin{equation*}
\widetilde A_{\mathrm{even}}:=(A_0)_{\mathrm{even}}-\phi_0^*\otimes\phi_0^*.
\end{equation*}
This operator is positive: it is the even restriction of the covariance operator with kernel $K_0(s,t)-\phi_0(s)\phi_0(t)$. Since $a_{2j,0}(\rho)\ne0$ for every $j$, Lemma~\ref{main-lem:finite_rank_secular} gives
\begin{equation*}
\lambda_{2m-2}^{(0)}(\rho)>\widetilde\lambda_{m,\mathrm{even}}(\rho)>\lambda_{2m}^{(0)}(\rho),
\qquad m\ge1,
\end{equation*}
where $\widetilde\lambda_{m,\mathrm{even}}(\rho)$ are the eigenvalues of $\widetilde A_{\mathrm{even}}$.  Moreover
\begin{equation*}
A_{\mathrm{even}}^{(\mu,\sigma)}=\widetilde A_{\mathrm{even}}-\phi_2^*\otimes\phi_2^*.
\end{equation*}
The rank-one interlacing Lemma~\ref{main-lem:minmax_rank1} yields
\begin{equation*}
\widetilde\lambda_{m,\mathrm{even}}(\rho)
\ge
\lambda_{m,\mathrm{even}}^{(\mu,\sigma)}(\rho)
\ge
\widetilde\lambda_{m+1,\mathrm{even}}(\rho),
\qquad m\ge1.
\end{equation*}
Combining the last two displays gives \eqref{main-eq:mu_sigma_even_bracket}.
\end{proof}

\subsection{Simplicity and ordering of the full spectrum}
\label{main-subsec:simple_full_spectrum}
The scalar odd equation gives simplicity only in the odd subspace. A different argument proves simplicity in the full space and excludes coincidences between the two parity spectra.
\begin{theorem}\label{main-thm:simple_full_spectrum}
For every $0<\rho<1$, all positive eigenvalues of $A^{(\mu,\sigma)}$ are simple. The parity spectra are disjoint and satisfy
\begin{equation}\label{main-eq:full_parity_order}
\lambda_{1,\mathrm{odd}}^{(\mu,\sigma)}>
\lambda_{1,\mathrm{even}}^{(\mu,\sigma)}>
\lambda_{2,\mathrm{odd}}^{(\mu,\sigma)}>
\lambda_{2,\mathrm{even}}^{(\mu,\sigma)}>\cdots>0.
\end{equation}
In the decreasing enumeration used in \eqref{main-eq:mu_sigma_limit_series}, this means
\begin{equation}\label{main-eq:full_parity_indexing}
\lambda_{2m-2}^{(\mu,\sigma)}=\lambda_{m,\mathrm{odd}}^{(\mu,\sigma)},\qquad
\lambda_{2m-1}^{(\mu,\sigma)}=\lambda_{m,\mathrm{even}}^{(\mu,\sigma)},\qquad m\ge1.
\end{equation}
The same simplicity and ordering hold for the positive spectrum of the projected observation-space operator, with the deterministic scaling specified in Remark~\ref{main-rem:factorization_normalization}.
\end{theorem}
\begin{proof}
Apply Theorem~\ref{supp-thm:consecutive_corrections} with $r=3$ and $\theta=\varphi_\beta$. This weight is positive, even and integrable, and its kernel is exactly
\begin{equation*}
\mathcal K_3(s,t)=\exp\!\left\{-\frac{s^2+t^2}{2}\right\}
\left(\exp\!\left\{st\right\}-1-st-\frac{s^2t^2}{2}\right)
=K^{(\mu,\sigma)}(s,t).
\end{equation*}
Thus $A^{[3]}_{\varphi_\beta}=A^{(\mu,\sigma)}$ on $L^2(\varphi_\beta;\mathbb R)$. Theorem~\ref{supp-thm:consecutive_corrections} gives simple positive eigenvalues and parity $(-1)^{3+n}$ for the eigenfunction of the $n$th eigenvalue, indexed from zero. The successive parities are therefore odd, even, odd, even, which proves \eqref{main-eq:full_parity_order} and \eqref{main-eq:full_parity_indexing}. The determinant and exterior-power arguments establishing that theorem are given in Section~\ref{supp-sec:total_positivity}.

Proposition~\ref{main-prop:data_frequency_factorization} identifies the positive eigenspaces in observation and Fourier coordinates by mutually inverse maps, preserving both multiplicity and parity. Remark~\ref{main-rem:factorization_normalization} multiplies all the eigenvalues by the same positive constant when the unnormalized weight is used. Hence neither simplicity nor their order changes.
\end{proof}

The result concerns positive eigenvalues; it does not assert simplicity of the nullspace of the projected observation operator. It also does not imply simplicity of the combined variance-only spectrum. Its even restriction is the same as here, but its unperturbed odd eigenvalues can coincide with even eigenvalues; Proposition~\ref{supp-prop:variance_only_double} gives an example by analytic bounds.

\subsection{Unique parameters for coincidences at even poles}
\label{main-subsec:even_pole_coincidence}
\label{main-subsec:even_pole_existence}
The pole criterion admits a complete parameter classification within the even subspace. In particular, the exception is unique for each fixed pole, but is not unique over the entire family of poles.
\begin{theorem}\label{main-thm:unique_pole_sequence}
For each integer $j\ge1$, there is exactly one $\rho_j\in(0,1)$ such that $\kappa_j=(1-\rho_j)\rho_j^{2j}$ is an eigenvalue of $A_{\mathrm{even}}^{(\mu,\sigma)}$. It equals $\lambda_{j,\mathrm{even}}^{(\mu,\sigma)}(\rho_j)$ and is simple in the full standardized space. The zero of $E_j$ is simple, with $E_j'(\rho_j)>0$, and
\begin{equation}\label{main-eq:ordered_exceptional_parameters}
0<\rho_1<\rho_2<\cdots<1,\qquad \rho_j\longrightarrow1.
\end{equation}
The pole $\kappa_0$ is never an even eigenvalue. Thus, at each fixed $\rho$, at most one even eigenvalue coincides with an unperturbed even pole. The pole classification is identical for the variance-only even restriction.
\end{theorem}
\begin{proof}
Let $j\ge1$, put $x=\rho^2$ and $c_k=\binom{2k}{k}4^{-k}$, and set $p_j(x)=x^j/(1-x^j)$. Section~\ref{supp-sec:unique_poles} proves the identities
\begin{align*}
E_j(\rho)&=2(1-x)^{5/2}x^{j-1}\mathcal H_j(x),\\
\mathcal H_j(x)&=\sum_{h\ge1}\frac{c_{j+h}h^2x^h}{1-x^h}
-\sum_{h=1}^j\frac{c_{j-h}h^2}{1-x^h},
\end{align*}
and establishes that $\mathcal H_j/p_j$ has positive derivative throughout $(0,1)$, with limits $-\infty$ and $+\infty$ at the two endpoints. There is therefore exactly one zero $x_j\in(0,1)$. Since $p_j(x_j)>0$,
\begin{equation*}
\mathcal H_j'(x_j)
=p_j(x_j)\left.\frac{d}{dx}\left(\frac{\mathcal H_j(x)}{p_j(x)}\right)\right|_{x=x_j}>0.
\end{equation*}
At $\rho_j=\sqrt{x_j}$, differentiation of the first identity gives
\begin{equation*}
E_j'(\rho_j)=4\rho_j(1-x_j)^{5/2}x_j^{j-1}\mathcal H_j'(x_j)>0.
\end{equation*}
Section~\ref{supp-sec:unique_poles} also proves, for $0<x<1$,
\begin{equation*}
\mathcal H_{j+1}(x)<\mathcal H_j(x),\qquad
\mathcal H_j(x)\le\frac{x(1+x)}{(1-x)^4}-j^2.
\end{equation*}
The first inequality gives $\mathcal H_{j+1}(x_j)<0$, hence $x_{j+1}>x_j$ by the unique sign change. For every fixed $x<1$, the second inequality is negative for all sufficiently large $j$, so $x_j>x$ eventually. Thus $x_j\uparrow1$, proving \eqref{main-eq:ordered_exceptional_parameters}.

The pole criterion in Theorem~\ref{main-thm:mu_sigma_main} now gives existence and uniqueness of the corresponding parameter. If an even eigenvalue with index $m$ equals $\kappa_j$, the strict bounds \eqref{main-eq:mu_sigma_even_bracket} require
$\kappa_{m-1}>\kappa_j>\kappa_{m+1}$, which, since the poles decrease strictly, forces $m-1<j<m+1$ and therefore $m=j$. Theorem~\ref{main-thm:simple_full_spectrum} proves simplicity in the full space. Every term in the expression for $E_0$ is positive, so $\kappa_0$ is excluded. Strict ordering of the $\rho_j$ excludes two even-pole coincidences at a fixed parameter. Finally, the variance-only and fully standardized even restrictions are the same operator.
\end{proof}

\begin{proposition}\label{main-prop:even_pole_exists}
The first exceptional value $\rho_*:=\rho_1$ belongs to $(2/3,3/4)$. In particular, $(1-\rho_*)\rho_*^2$ is a simple eigenvalue of $A^{(\mu,\sigma)}$ and its largest even eigenvalue.
\end{proposition}
\begin{proof}
With $x=\rho^2$, separating $k=0$ in the original pole criterion gives
\begin{align}\label{main-eq:even_pole_E1}
E_1(\rho)&=1+\frac{(3x-2)^2}{2x}-\frac{2(1-x)^{3/2}}{x}\nonumber\\
&\quad+\frac{2(1-x)^{5/2}}{x^2}
\sum_{k=2}^\infty\frac{c_k(k-1)^2x^{2k}}{1-x^{k-1}}.
\end{align}
Using $c_k\le3/8$ and $1-x^{k-1}\ge1-x$ in the positive series yields
\begin{equation*}
E_1(2/3)\le\frac32-\frac{253673\sqrt5}{329550}<0,
\qquad E_1(3/4)\ge\frac{313}{288}-\frac{7\sqrt7}{18}>0.
\end{equation*}
Theorem~\ref{main-thm:unique_pole_sequence} then locates the unique zero and identifies its eigenvalue.
\end{proof}

Rational interval bounds give $0.71351870517303<\rho_*<0.71351870517305$, with $(1-\rho_*)\rho_*^2\simeq0.145850189093146$. Sections~\ref{supp-sec:pole_validation} and \ref{supp-sec:unique_poles} give the bounds and construct an eigenfunction at the coincident value. Theorem~\ref{main-thm:unique_pole_sequence} establishes uniqueness for each fixed pole; the exceptional parameters themselves form an infinite, strictly increasing sequence.

\subsection{Relation with the Fredholm determinant formulation}
\label{main-subsec:EH_equivalence}
The determinant representation in Ebner and Henze~\cite[Section~2]{EbnerHenze2023} uses the same Gaussian operator and the same three correction directions. The following identity gives the correspondence with its reciprocal spectral parameter, with the Fredholm determinant normalized to equal one at the origin.
\begin{proposition}\label{main-prop:EH_equivalence}
Put $d_n=(1-\rho)\rho^n$ and
\begin{equation*}
d(z):=\prod_{n=0}^{\infty}(1-zd_n),\qquad
\Delta(z):=\det(I-zA^{(\mu,\sigma)}).
\end{equation*}
For $z\ne0$ with $z^{-1}\notin\{d_n:n\ge0\}$,
\begin{equation}\label{main-eq:Fredholm_factorization}
\Delta(z)=d(z)F_{\mathrm{odd}}(z^{-1})D_{\mathrm{even}}(z^{-1}).
\end{equation}
The right-hand side has an entire continuation. Its positive zeros are the reciprocals of the positive eigenvalues of $A^{(\mu,\sigma)}$, with their multiplicities.
\end{proposition}
\begin{proof}
The operators are complexified when $z$ is complex. The operator $A_0$ is trace class, since $\sum_n d_n=1$. The correction has finite rank, so $A^{(\mu,\sigma)}$ is trace class as well. Let $U:\mathbb R^3\to L^2(\varphi_\beta;\mathbb R)$ have columns $\phi_0^*,\phi_1^*,\phi_2^*$. For $I-zA_0$ invertible, the finite-rank determinant identity gives
\begin{equation*}
\Delta(z)=d(z)\det\{I_3+zU^*(I-zA_0)^{-1}U\}.
\end{equation*}
It follows first in finite dimension by factoring $I-zA_0$; trace-norm approximation then proves the above identity. The entries of the second determinant are the series of Lemma~\ref{main-lem:finite_rank_secular} evaluated at $\lambda=z^{-1}$. Parity separates its odd scalar factor from its even two-dimensional factor and gives \eqref{main-eq:Fredholm_factorization}. Entire continuation and the statement about zeros follow from the trace-class Fredholm determinant. Values at $z=d_n^{-1}$ are values of this continuation, not products of separately evaluated singular factors.

The Gaussian kernel, weight and three correction functions in Ebner and Henze~\cite{EbnerHenze2023} define the same operators $A_0$ and $A^{(\mu,\sigma)}$. Their correction functions are $\phi_2^*,\phi_1^*,\phi_0^*$ in that order, and their reciprocal parameter is $z$. Proposition~\ref{main-prop:Hermite_basis} computes the spectrum $d_n$ directly from this common operator. Thus the correspondence is an identity of integral operators and of their normalized Fredholm determinants, with the overlaps evaluated in Proposition~\ref{main-prop:overlaps_closed_form}. Proposition~\ref{main-prop:pole_criterion} gives the criterion at the poles. The last assertion of Ebner and Henze~\cite[Theorem~2.2]{EbnerHenze2023} excludes all unperturbed eigenvalues. Theorem~\ref{main-thm:unique_pole_sequence} gives an infinite sequence of distinct counterexamples to this exclusion, one at each even pole except the largest. The Fredholm product identity remains valid when the singularities are treated by entire continuation.
\end{proof}

The reason that parity does not imply exclusion of the even poles can be seen directly. Put $p_k=(1,b_k)^\top$ and, for $\lambda$ near $\kappa_j$, define
\begin{equation*}
H_j(\lambda)=I_2+\sum_{k\ne j}\frac{w_kp_kp_k^\top}{\lambda-\kappa_k}.
\end{equation*}
The series is analytic there. Since permutation of the two coordinates does not change the determinant,
\begin{equation*}
D_{\mathrm{even}}(\lambda)=\det H_j(\lambda)
+\frac{w_j}{\lambda-\kappa_j}\,p_j^\top\operatorname{adj}\{H_j(\lambda)\}p_j.
\end{equation*}
For a symmetric two-dimensional matrix, the last quadratic form equals that along $(-b_j,1)^\top$ in the original matrix. Therefore
\begin{equation}\label{main-eq:even_residue_criterion}
\lim_{\lambda\to\kappa_j}(\lambda-\kappa_j)D_{\mathrm{even}}(\lambda)
=w_jE_j(\rho).
\end{equation}
Equation~\eqref{main-eq:even_residue_criterion} identifies the possible cancellation. Parity separates the odd and even determinants but gives no reason for $E_j(\rho)$ to be nonzero. At a zero of $E_j$, the singularity of the even determinant is removable; the zero of the unperturbed Fredholm factor cannot then be discarded. The criterion and construction concern equality of eigenvalues, not preservation of the corresponding unperturbed eigenfunction; see Section~\ref{supp-sec:pole_validation}.

The preceding results distinguish rank-one interlacing from rank-two bounds.
\begin{myrem}\label{main-rem:global_spectrum}
In the mean-only and fully standardized cases, the positive odd eigenvalues lie strictly between successive odd eigenvalues of $A_0$. In the variance-only case, the odd spectrum is unchanged. The even spectrum with estimated variance satisfies the two-step bounds \eqref{main-eq:mu_sigma_even_bracket}; these bounds do not by themselves exclude equality with an interior pole. Proposition~\ref{main-prop:pole_criterion} treats that question, and Theorem~\ref{main-thm:unique_pole_sequence} classifies the corresponding parameters. Theorem~\ref{main-thm:simple_full_spectrum} gives simplicity of the positive full-standardization spectrum, including at these parameters. The exact deletion in Proposition~\ref{main-prop:aligned_score_deletion} is a sufficient conclusion for a score space generated by eigenfunctions, not a necessary criterion for every possible coincidence of eigenvalues.
\end{myrem}

The statistic used for numerical calibration is the $V$-form. Its relation to the residual $U$-form includes a nonzero centering constant, as in the general parameter-substitution results of Cupari{\'c} et~al.~\cite{CuparicMilosevicObradovic2022}. The following corollary computes this constant in the normalization of the present paper.
\Needspace{13\baselineskip}
\begin{cor}\label{main-cor:full_U_V_limit}
Under the assumptions of Theorem~\ref{main-thm:mu_sigma_main}, let
\begin{equation*}
\tau_\beta:=1-(1+2\beta^2)^{-1/2}.
\end{equation*}
Then
\begin{equation*}
nU_n^{(\mu,\sigma)}(K_\rho)
\xrightarrow[n\to\infty]{d}
\frac{1}{1-\rho}\{T_\infty^{(\mu,\sigma)}-\tau_\beta\}.
\end{equation*}
The mean of the limiting variable is
\begin{equation*}
-\frac{1}{1-\rho}\left\{\frac{\beta^2}{(1+2\beta^2)^{3/2}}
+\frac{3\beta^4}{2(1+2\beta^2)^{5/2}}\right\}.
\end{equation*}
\end{cor}
The proof is given in Section~\ref{supp-subsec:proof_U_V_limit} of the Supplementary Material.

\section{Discussion}
\label{main-sec:discussion}

The spectral effect of parameter substitution depends on the correction directions and their relation to the kernel eigenfunctions. For the fully standardized Epps--Pulley statistic, Theorem~\ref{main-thm:simple_full_spectrum} gives a simple positive spectrum with strict alternation between odd and even eigenfunctions. Theorem~\ref{main-thm:unique_pole_sequence} shows that each unperturbed even pole, except the largest, nevertheless belongs to the corrected even spectrum at a unique weight parameter. Thus simplicity of the corrected spectrum does not imply disjointness from the unperturbed spectrum. The pole criterion makes these coincidences part of the spectral description rather than values excluded from the eigenvalue equations.

The observation-space analysis connects these spectral results to quadratic statistics evaluated at standardized observations. Proposition~\ref{main-prop:derivative_tail_criterion} controls the complete spectral expansion through Gaussian norms of the eigenfunctions and their first derivatives; its summability criterion also applies to signed kernels. For characteristic-function kernels, Corollary~\ref{main-cor:weighted_kernel_full_tail} verifies the criterion under a finite fourth moment of the weight, without requiring explicit eigenfunctions of the centered even kernel. The factorization then identifies the positive observation and Fourier spectra, including multiplicities and parity. The covariance projection agrees with the parameter-substitution theory of Arcones~\cite{Arcones2007} and Cupari{\'c} et~al.~\cite{CuparicMilosevicObradovic2022}, while Corollary~\ref{main-cor:full_U_V_limit} retains the original kernel trace in passing from the $V$-form to the $U$-form.

The simplicity result extends beyond Gaussian weights. Theorem~\ref{supp-thm:consecutive_corrections} in the Supplementary Material applies to every positive, even, integrable weight when the Gaussian covariance corrections remove a consecutive initial block. The variance-only example in Proposition~\ref{supp-prop:variance_only_double} shows that a nonconsecutive correction can instead produce an eigenvalue of multiplicity two. This suggests investigating which other sets of correction directions preserve simplicity. Within the variance-only family, a more specific question is to determine all intersections of the even and odd eigenvalue curves, and whether each prescribed pair can intersect at more than one parameter value.

The exceptional sequence also raises a quantitative question. Theorem~\ref{main-thm:unique_pole_sequence} gives $\rho_j\to1$, but not an asymptotic equivalent for $1-\rho_j$ or for the coincident eigenvalue $(1-\rho_j)\rho_j^{2j}$. The series representation of the pole criterion provides a starting point for investigating the regime in which $j$ increases and $\rho$ approaches one. A related problem concerns calibration when the weight parameter varies with sample size. Here the objective is to determine joint conditions on the approach of $\rho$ to one and on the growth of the Hermite truncation dimension under which the numerical critical values yield asymptotically correct rejection probabilities. This would extend the present fixed-weight calculations to a regime in which the number of retained Hermite functions must increase.

Local power and selection of the weight have already been studied for this class of tests. Henze and Wagner~\cite{HenzeWagner1997} establish the behavior under contiguous alternatives, and Tenreiro~\cite{Tenreiro2019} treats automatic selection and calibration over finite families of tuning parameters. The present spectral results suggest a more specific comparison: for a prescribed family of local alternatives, determine the coordinates of the projected mean shift in the corrected eigenbasis, and separate the contributions of the even and odd components to the resulting noncentral quadratic forms. Such a comparison requires the alternative coordinates as well as the null eigenvalues. Extending this analysis to selection over a growing set or a continuum of weights would require joint control of the projected process and of the spectral approximation; the fixed-weight null results alone do not justify that calibration.

For the multivariate BHEP statistic, an extension should address spectral multiplicities directly. With a radial Gaussian weight, a suitable starting point is to combine the tensor-product Hermite representation with a decomposition into subspaces determined by rotational symmetry. The questions are which subspaces are altered by estimating the mean and covariance matrix, which multiplicities are imposed by symmetry, and whether pole coincidences can be classified within the affected subspaces. The multivariate calculations of Ebner, Jim\'enez-Gamero and Milo\v{s}evi\'c~\cite{EbnerJimenezGameroMilosevic2025} provide numerical comparisons for such an analysis. The aim would be an analytic description of the invariant subspaces and their spectra, rather than an assertion of whole-space simplicity analogous to the univariate result.

\begin{acks}[Acknowledgments]
Samis Trevezas is the corresponding author. He is also affiliated with MICS Laboratory, CentraleSup\'elec, Universit\'e Paris-Saclay.
\end{acks}

\clearpage
\section*{Supplementary Material}
\setcounter{section}{0}
\setcounter{equation}{0}
\setcounter{theorem}{0}
\setcounter{prop}{0}
\setcounter{lem}{0}
\setcounter{myrem}{0}
\setcounter{ex}{0}
\setcounter{table}{0}
\renewcommand{\thesection}{S\arabic{section}}
\renewcommand{\theHsection}{supp.\arabic{section}}
\renewcommand{\theequation}{S\arabic{section}.\arabic{equation}}
\renewcommand{\theHequation}{supp.\arabic{section}.\arabic{equation}}
\renewcommand{\thetheorem}{S\arabic{theorem}}
\renewcommand{\theHtheorem}{supp.\arabic{theorem}}
\renewcommand{\theprop}{S\arabic{prop}}
\renewcommand{\theHprop}{supp.\arabic{prop}}
\renewcommand{\thelem}{S\arabic{lem}}
\renewcommand{\theHlem}{supp.\arabic{lem}}
\renewcommand{\themyrem}{S\arabic{myrem}}
\renewcommand{\theHmyrem}{supp.\arabic{myrem}}
\renewcommand{\theex}{S\arabic{ex}}
\renewcommand{\theHex}{supp.\arabic{ex}}

The notation is that of the main paper. Supplementary results, sections and equations carry an S prefix; references without that prefix are to the main paper. Sections~\ref{supp-sec:supp_one_parameter_ecf}--\ref{supp-sec:one_parameter_Krho_tail} give the one-parameter results and the explicit odd-kernel calculations; Section~\ref{supp-sec:numerical_use} gives the numerical study. The principal supplementary proofs for the main results are in Section~\ref{supp-sec:full_derivative_tail} (the derivative estimate and residual-pair laws), Section~\ref{supp-sec:supp_full_ecf} (the full characteristic-function expansion), Section~\ref{supp-sec:secular_multiplicity} (off-pole multiplicity), Section~\ref{supp-sec:total_positivity} (simplicity and parity), and Section~\ref{supp-sec:unique_poles} (unique exceptional parameters). Section~\ref{supp-sec:pole_validation} gives a rational enclosure and an eigenfunction construction, and Section~\ref{supp-sec:Arcones_comparison} gives the comparison with Arcones's substitution result. Section~\ref{supp-sec:preliminary_proofs} contains the supporting proofs for the Hermite, perturbation, empirical-sum, infinite-rank, factorization and original-trace results.

\section{One-parameter ECF limits}
\label{supp-sec:supp_one_parameter_ecf}
\setcounter{equation}{0}

This section proves the empirical characteristic-function limits used in
Theorems~\ref{supp-thm:mu_only_main_updated} and \ref{supp-thm:sigma_only_main_updated}.  The
notation is that of the main paper.  We fix $0<\rho<1$, set $\beta=\sqrt\rho/(1-\rho)$, and
write $\varphi_\beta$ for the $N(0,\beta^2)$ density.  Gaussian limits of characteristic-function processes are taken in the real space $\mathcal H_\beta$ of \eqref{main-eq:Hermitian_space}. Its covariance spectra are computed on $L^2(\varphi_\beta;\mathbb R)$ using \eqref{main-eq:parity_isometry}. All limits are under
$X_i\stackrel{\mathrm{i.i.d.}}{\sim}N(0,1)$.

\begin{proposition}\label{supp-prop:supp_mu_ecf_limit}
Let
\begin{equation*}
Y_{n,j}=X_j-\overline X_n,
\qquad
\psi_n(t)=\frac1n\sum_{j=1}^n \exp\!\left\{itY_{n,j}\right\}.
\end{equation*}
Then
\begin{equation}
\label{supp-eq:supp_mu_ecf_expansion}
\sqrt n(\psi_n-\phi_0)
=
\frac1{\sqrt n}\sum_{j=1}^n \zeta_\mu(\cdot;X_j)+o_{\mathbb P}(1)
\quad\hbox{in }L^2(\varphi_\beta),
\end{equation}
where
\begin{equation*}
\zeta_\mu(t;x):=\exp\!\left\{itx\right\}-\phi_0(t)-it\phi_0(t)x.
\end{equation*}
Consequently, $\sqrt n(\psi_n-\phi_0)$ converges in distribution in $L^2(\varphi_\beta)$ to a
centered real Gaussian element of $\mathcal H_\beta$ with kernel $\EE[Z(s)\overline{Z(t)}]$
\begin{equation}
\label{supp-eq:supp_mu_covariance}
K^{(\mu)}(s,t)
=
\exp\!\left\{-\frac{(s-t)^2}{2}\right\}
-
(1+st)\exp\!\left\{-\frac{s^2+t^2}{2}\right\}.
\end{equation}
\end{proposition}

\begin{proof}
Since $\psi_n(t)=\exp\!\left\{-it\overline X_n\right\}\phi_n(t)$ and
$|\exp\!\left\{-iu\right\}-(1-iu)|\le u^2/2$, we have
\begin{equation*}
\left\|\exp\!\left\{-it\overline X_n\right\}-(1-it\overline X_n)\right\|_\beta
\le C_\beta \overline X_n^2.
\end{equation*}
Thus
\begin{equation*}
\psi_n(t)-\phi_0(t)
=
\phi_n(t)-\phi_0(t)-it\phi_0(t)\overline X_n+r_n(t),
\qquad
\sqrt n\|r_n\|_\beta=o_{\mathbb P}(1).
\end{equation*}
The replacement of $\phi_n$ by $\phi_0$ in the linear term is legitimate because
$\|t(\phi_n-\phi_0)\|_\beta=O_{\mathbb P}(n^{-1/2})$ and
$\overline X_n=O_{\mathbb P}(n^{-1/2})$.  Multiplying by $\sqrt n$ and using
$\sqrt n\overline X_n=n^{-1/2}\sum_{j=1}^n X_j$ gives \eqref{supp-eq:supp_mu_ecf_expansion}.
The summands have finite second moment in $L^2(\varphi_\beta)$, hence the Hilbert-space
central limit theorem applies.

Using
\begin{align*}
\mathbb E[\exp\!\left\{iuX\right\}]
&=\exp\!\left\{-u^2/2\right\},\\
\mathbb E[X\exp\!\left\{iuX\right\}]
&=iu\exp\!\left\{-u^2/2\right\},
\end{align*}
in $\mathbb E[\zeta_\mu(s;X)\overline{\zeta_\mu(t;X)}]$ gives
\eqref{supp-eq:supp_mu_covariance}.
\end{proof}

\begin{proposition}\label{supp-prop:supp_sigma_ecf_limit}
Let
\begin{equation*}
S_{n,0}^2=\frac1n\sum_{j=1}^n X_j^2,
\qquad
W_{n,j}=\frac{X_j}{S_{n,0}},
\qquad
\psi_n(t)=\frac1n\sum_{j=1}^n \exp\!\left\{itW_{n,j}\right\}.
\end{equation*}
Then
\begin{equation}
\label{supp-eq:supp_sigma_ecf_expansion}
\sqrt n(\psi_n-\phi_0)
=
\frac1{\sqrt n}\sum_{j=1}^n \zeta_\sigma(\cdot;X_j)+o_{\mathbb P}(1)
\quad\hbox{in }L^2(\varphi_\beta),
\end{equation}
where
\begin{equation*}
\zeta_\sigma(t;x):=\exp\!\left\{itx\right\}-\phi_0(t)+\frac{t^2}{2}\phi_0(t)(x^2-1).
\end{equation*}
Consequently, $\sqrt n(\psi_n-\phi_0)$ converges in distribution in $L^2(\varphi_\beta)$ to a
centered real Gaussian element of $\mathcal H_\beta$ with kernel $\EE[Z(s)\overline{Z(t)}]$
\begin{equation}
\label{supp-eq:supp_sigma_covariance}
K^{(\sigma)}(s,t)
=
\exp\!\left\{-\frac{(s-t)^2}{2}\right\}
-
\left(1+\frac{s^2t^2}{2}\right)
\exp\!\left\{-\frac{s^2+t^2}{2}\right\}.
\end{equation}
\end{proposition}

\begin{proof}
Since $S_{n,0}^2-1=n^{-1}\sum_{j=1}^n(X_j^2-1)$,
\begin{equation*}
\sqrt n(S_{n,0}-1)
=
\frac1{2\sqrt n}\sum_{j=1}^n(X_j^2-1)+o_{\mathbb P}(1).
\end{equation*}
The expansion $S_{n,0}^{-1}=1-(S_{n,0}-1)+o_{\mathbb P}(n^{-1/2})$, followed by a first-order
Taylor expansion of $x\mapsto \exp\!\left\{itx\right\}$, gives
\begin{equation*}
\psi_n(t)-\phi_0(t)
=
\phi_n(t)-\phi_0(t)
-it(S_{n,0}-1)\frac1n\sum_{j=1}^n X_j\exp\!\left\{itX_j\right\}
+r_n(t),
\end{equation*}
Put $b_n=S_{n,0}-1$. On $|b_n|\le1/2$, the identity $(1+b_n)^{-1}=1-b_n+b_n^2/(1+b_n)$ and $|\exp\!\left\{iu\right\}-1-iu|\le u^2/2$ give
\begin{equation*}
|r_n(t)|\le Cb_n^2\left\{|t|\frac1n\sum_{j=1}^n|X_j|+t^2\frac1n\sum_{j=1}^nX_j^2\right\}.
\end{equation*}
Thus $\sqrt n\|r_n\|_\beta=o_{\mathbb P}(1)$, since $b_n=O_{\mathbb P}(n^{-1/2})$ and the Gaussian weight has a finite fourth moment. Moreover,
\begin{equation*}
\left\|t\left(\frac1n\sum_{j=1}^n X_j\exp\!\left\{itX_j\right\}-\mathbb E[X\exp\!\left\{itX\right\}]\right)\right\|_\beta
=
O_{\mathbb P}(n^{-1/2}),
\end{equation*}
and $\mathbb E[X\exp\!\left\{itX\right\}]=it\phi_0(t)$.  Substitution of the expansion of $S_{n,0}-1$ gives
\eqref{supp-eq:supp_sigma_ecf_expansion}.  The Hilbert-space central limit theorem applies because
$\mathbb E\|\zeta_\sigma(\cdot;X)\|_\beta^2<\infty$.

Using the identities
\begin{align*}
\mathbb E[\exp\!\left\{iuX\right\}]&=\exp\!\left\{-u^2/2\right\},\\
\mathbb E[X^2\exp\!\left\{iuX\right\}]&=(1-u^2)\exp\!\left\{-u^2/2\right\},
\end{align*}
in $\mathbb E[\zeta_\sigma(s;X)\overline{\zeta_\sigma(t;X)}]$ gives
\eqref{supp-eq:supp_sigma_covariance}.
\end{proof}
\section{One-parameter spectral results}
\label{supp-sec:one_param_studentization}

\subsection{Notation}
\label{supp-subsec:one_param_shared_notation}

Throughout this section, $\phi_0(t)=\exp\!\left\{-t^2/2\right\}$ and $\varphi_\beta$ is the Gaussian weight
\eqref{main-eq:EP_varphi_beta}.  We fix $\rho\in(0,1)$ and set $\beta=\sqrt{\rho}/(1-\rho)$.
All limits are under $H_0:X_i\stackrel{\mathrm{i.i.d.}}{\sim}\mathcal N(0,1)$.  We use the
operator $A_0$, its eigenpairs, and the overlap coefficients introduced in
Section~\ref{main-sec:spectral_preliminaries}.  The empirical characteristic-function limits needed below
are proved in Section~\ref{supp-sec:supp_one_parameter_ecf} of the Supplementary Material.

\subsection{Known variance, estimated mean}
\label{supp-subsec:mu_only}

Assume $X_1,\dots,X_n$ are i.i.d. $\mathcal N(\mu,1)$.  By translation invariance we take
$\mu=0$.  Let
\begin{equation*}
\overline{X}_n=\frac{1}{n}\sum_{j=1}^n X_j,
\qquad
Y_{n,j}:=X_j-\overline{X}_n.
\end{equation*}
The empirical characteristic function of the centered sample is
\begin{equation*}
\psi_n(t):=\frac{1}{n}\sum_{j=1}^n \exp\!\left\{itY_{n,j}\right\}.
\end{equation*}
The statistic is
\begin{equation*}
T_{n,\beta}^{(\mu)}
:=
 n\int_{-\infty}^{\infty}|\psi_n(t)-\phi_0(t)|^2\,\varphi_\beta(t)\,dt.
\end{equation*}

\begin{theorem}\label{supp-thm:mu_only_main_updated}
Under $H_0$,
\begin{equation}
\label{supp-eq:mu_only_limit_series_updated}
T_{n,\beta}^{(\mu)}
\xrightarrow[n\to\infty]{d}
T_{\infty}^{(\mu)}
:=
\int_{-\infty}^{\infty}|Z(t)|^2\varphi_\beta(t)\,dt
\stackrel{d}{=}
\sum_{n=0}^{\infty}\lambda_n^{(\mu)}(\rho)N_n^2,
\end{equation}
where $Z$ is a centered real Gaussian element in $\mathcal H_\beta$, $\{N_n\}_{n\ge0}$ are i.i.d.
standard normal variables, and $\{\lambda_n^{(\mu)}(\rho)\}_{n\ge0}$ are the nonzero eigenvalues of
the covariance operator $A^{(\mu)}$.  Its kernel is
\begin{equation}
\label{supp-eq:mu_only_kernel_updated}
K^{(\mu)}(s,t)
=
\exp\!\left\{-\frac{(s-t)^2}{2}\right\}
-
(1+st)\exp\!\left\{-\frac{s^2+t^2}{2}\right\},
\end{equation}
or, equivalently,
\begin{equation}
\label{supp-eq:mu_only_operator_form}
A^{(\mu)}=A_0-\phi_0^*\otimes\phi_0^*-\phi_1^*\otimes\phi_1^*.
\end{equation}
The even and odd subspaces are invariant.  The even eigenvalues $\{\lambda_{m,\mathrm{even}}^{(\mu)}(\rho)\}_{m\ge1}$ are the unique roots in
\begin{equation*}
\lambda\in\bigl((1-\rho)\rho^{2m},(1-\rho)\rho^{2m-2}\bigr),
\qquad m\ge1,
\end{equation*}
of
\begin{equation}
\label{supp-eq:mu_only_even_secular_updated}
F_{\mathrm{even}}(\lambda)
:=
1+\sum_{j=0}^{\infty}\frac{a_{2j,0}^2(\rho)}{\lambda-(1-\rho)\rho^{2j}}=0.
\end{equation}
The odd eigenvalues $\{\lambda_{m,\mathrm{odd}}^{(\mu)}(\rho)\}_{m\ge1}$ are the unique roots in
\begin{equation*}
\lambda\in\bigl((1-\rho)\rho^{2m+1},(1-\rho)\rho^{2m-1}\bigr),
\qquad m\ge1,
\end{equation*}
of
\begin{equation}
\label{supp-eq:mu_only_odd_secular_updated}
F_{\mathrm{odd}}(\lambda)
:=
1+\sum_{j=0}^{\infty}\frac{a_{2j+1,1}^2(\rho)}{\lambda-(1-\rho)\rho^{2j+1}}=0.
\end{equation}
In particular,
\begin{equation}
\label{supp-eq:mu_only_interlace_odd_updated}
(1-\rho)\rho>\lambda_{1,\mathrm{odd}}^{(\mu)}(\rho)>(1-\rho)\rho^3>
\lambda_{2,\mathrm{odd}}^{(\mu)}(\rho)>(1-\rho)\rho^5>\cdots.
\end{equation}
All positive eigenvalues of $A^{(\mu)}$ are simple, and the full ordering is
\begin{equation*}
\lambda_{1,\mathrm{even}}^{(\mu)}>
\lambda_{1,\mathrm{odd}}^{(\mu)}>
\lambda_{2,\mathrm{even}}^{(\mu)}>
\lambda_{2,\mathrm{odd}}^{(\mu)}>\cdots>0.
\end{equation*}
\end{theorem}

\begin{proof}
Proposition~\ref{supp-prop:supp_mu_ecf_limit} gives the weak convergence of
$\sqrt n(\psi_n-\phi_0)$ in $L^2(\varphi_\beta)$ and the covariance kernel
\eqref{supp-eq:mu_only_kernel_updated}.  The continuous mapping theorem and Theorem~\ref{supp-thm:KLquadratic}
give \eqref{supp-eq:mu_only_limit_series_updated}.

The operator identity \eqref{supp-eq:mu_only_operator_form} follows from the displayed covariance kernel.
Since $\phi_0^*$ is even and $\phi_1^*$ is odd, the two parity subspaces are invariant and
\begin{equation*}
A_{\mathrm{even}}^{(\mu)}=(A_0)_{\mathrm{even}}-\phi_0^*\otimes\phi_0^*,
\qquad
A_{\mathrm{odd}}^{(\mu)}=(A_0)_{\mathrm{odd}}-\phi_1^*\otimes\phi_1^*.
\end{equation*}
Lemma~\ref{main-lem:finite_rank_secular} gives \eqref{supp-eq:mu_only_even_secular_updated} and
\eqref{supp-eq:mu_only_odd_secular_updated}.  Proposition~\ref{main-prop:overlaps_closed_form} gives
$a_{2j+1,1}(\rho)\ne0$ for every $j$, and the strict interlacing in
\eqref{supp-eq:mu_only_interlace_odd_updated} follows from the rank-one part of
Lemma~\ref{main-lem:finite_rank_secular}. Whole-space simplicity and the ordering between the two parity spectra follow from Theorem~\ref{supp-thm:consecutive_corrections} with $r=2$. Its proof does not use the present theorem.
\end{proof}

\subsection{Known mean, estimated variance}
\label{supp-subsec:sigma_only}

Assume $X_1,\dots,X_n$ are i.i.d. $\mathcal N(0,\sigma^2)$, with known mean $0$ and $\sigma>0$. By scale invariance we take $\sigma^2=1$. Put
\begin{equation*}
S_{n,0}^2:=\frac1n\sum_{j=1}^n X_j^2,
\qquad
W_{n,j}:=\frac{X_j}{S_{n,0}}.
\end{equation*}
Let
\begin{equation*}
\psi_n(t):=\frac1n\sum_{j=1}^n \exp\!\left\{itW_{n,j}\right\}.
\end{equation*}
The statistic is
\begin{equation*}
T_{n,\beta}^{(\sigma)}
:=
n\int_{-\infty}^{\infty}|\psi_n(t)-\phi_0(t)|^2\,\varphi_\beta(t)\,dt.
\end{equation*}

\begin{theorem}\label{supp-thm:sigma_only_main_updated}
Under $H_0$,
\begin{equation}
\label{supp-eq:sigma_only_limit_series_updated}
T_{n,\beta}^{(\sigma)}
\xrightarrow[n\to\infty]{d}
T_{\infty}^{(\sigma)}
:=
\int_{-\infty}^{\infty}|Z(t)|^2\varphi_\beta(t)\,dt
\stackrel{d}{=}
\sum_{n=0}^{\infty}\lambda_n^{(\sigma)}(\rho)N_n^2.
\end{equation}
Here $Z$ is a centered real Gaussian element of $\mathcal H_\beta$, the $N_n$ are independent standard normal variables, and the $\lambda_n^{(\sigma)}$ are the positive eigenvalues of the covariance operator, with multiplicity. The kernel $\EE[Z(s)\overline{Z(t)}]$ is
\begin{equation}
\label{supp-eq:sigma_only_kernel_updated}
K^{(\sigma)}(s,t)
=
\exp\!\left\{-\frac{(s-t)^2}{2}\right\}
-
\left(1+\frac{s^2t^2}{2}\right)\exp\!\left\{-\frac{s^2+t^2}{2}\right\},
\end{equation}
or
\begin{equation}
\label{supp-eq:sigma_only_operator_form}
A^{(\sigma)}=A_0-\phi_0^*\otimes\phi_0^*-\phi_2^*\otimes\phi_2^*.
\end{equation}
The odd eigenvalues are unperturbed:
\begin{equation}
\label{supp-eq:sigma_only_odd_eigs_updated}
\lambda_{m,\mathrm{odd}}^{(\sigma)}(\rho)=\lambda_{2m-1}^{(0)}(\rho)=(1-\rho)\rho^{2m-1},
\qquad m\ge1.
\end{equation}
The even eigenvalues away from $\{(1-\rho)\rho^{2j}:j\ge0\}$ are the positive roots of
\begin{equation}
\label{supp-eq:sigma_only_det_updated}
D_{\mathrm{even}}^{(\sigma)}(\lambda)
:=
\det
\begin{pmatrix}
1+G_{22}(\lambda) & G_{20}(\lambda)\\
G_{20}(\lambda) & 1+G_{00}(\lambda)
\end{pmatrix}=0,
\end{equation}
where $\lambda>0$ is not a pole and, for $p,q\in\{0,2\}$,
\begin{equation}
\label{supp-eq:sigma_only_Gpq_updated}
G_{pq}(\lambda)
:=
\sum_{j=0}^{\infty}\frac{a_{2j,p}(\rho)a_{2j,q}(\rho)}{\lambda-(1-\rho)\rho^{2j}}.
\end{equation}
If $\{\lambda_{m,\mathrm{even}}^{(\sigma)}(\rho)\}_{m\ge1}$ are ordered decreasingly, then
\begin{equation}
\label{supp-eq:sigma_only_rank2_bracket_updated}
\lambda_{2m-2}^{(0)}(\rho)>
\lambda_{m,\mathrm{even}}^{(\sigma)}(\rho)>
\lambda_{2m+2}^{(0)}(\rho),
\qquad m\ge1.
\end{equation}
At a pole, the even-spectrum criterion of Theorem~\ref{main-thm:mu_sigma_main} applies, since the even operators are identical. Hence Theorem~\ref{main-thm:unique_pole_sequence} gives the unique, strictly ordered coincidence parameters for this even restriction as well. Every positive eigenvalue in either parity restriction is simple; the two spectra need not be disjoint.
\end{theorem}

\begin{proof}
Proposition~\ref{supp-prop:supp_sigma_ecf_limit} gives the weak convergence, the covariance
kernel \eqref{supp-eq:sigma_only_kernel_updated}, and hence \eqref{supp-eq:sigma_only_limit_series_updated}
by the continuous mapping theorem and Theorem~\ref{supp-thm:KLquadratic}.

The operator identity \eqref{supp-eq:sigma_only_operator_form} follows from the kernel
\eqref{supp-eq:sigma_only_kernel_updated}.  Since $\phi_0^*$ and $\phi_2^*$ are even, the odd
restriction of $A^{(\sigma)}$ is $(A_0)_{\mathrm{odd}}$, giving
\eqref{supp-eq:sigma_only_odd_eigs_updated}.  On the even subspace, Lemma~\ref{main-lem:finite_rank_secular}
gives \eqref{supp-eq:sigma_only_det_updated}--\eqref{supp-eq:sigma_only_Gpq_updated}.

Let
\begin{equation*}
\widetilde A_{\mathrm{even}}:=(A_0)_{\mathrm{even}}-\phi_0^*\otimes\phi_0^*.
\end{equation*}
The rank-one interlacing in Lemma~\ref{main-lem:finite_rank_secular}, together with
$a_{2j,0}(\rho)\ne0$, gives
\begin{equation*}
\lambda_{2m-2}^{(0)}(\rho)>
\lambda_{m,\mathrm{even}}^{(\mu)}(\rho)>
\lambda_{2m}^{(0)}(\rho),
\qquad m\ge1.
\end{equation*}
Since $A_{\mathrm{even}}^{(\sigma)}=\widetilde A_{\mathrm{even}}-\phi_2^*\otimes\phi_2^*$,
Lemma~\ref{main-lem:minmax_rank1} gives
\begin{equation*}
\lambda_{m,\mathrm{even}}^{(\mu)}(\rho)
\ge
\lambda_{m,\mathrm{even}}^{(\sigma)}(\rho)
\ge
\lambda_{m+1,\mathrm{even}}^{(\mu)}(\rho).
\end{equation*}
Combining the last two displays yields \eqref{supp-eq:sigma_only_rank2_bracket_updated}. The even restriction agrees with the fully standardized one, so Theorems~\ref{main-thm:simple_full_spectrum} and \ref{main-thm:unique_pole_sequence} give its simplicity and pole classification. The odd restriction has distinct geometric eigenvalues. Proposition~\ref{supp-prop:variance_only_double} proves that the combined spectrum can contain an eigenvalue of multiplicity two.
\end{proof}

\section{One-parameter projection expansions}
\label{supp-sec:one_parameter_projection}
\setcounter{equation}{0}

This section gives the one-parameter versions of Lemma~\ref{main-lem:data_projection_empirical_sums}.
The notation is that of Section~\ref{main-sec:data_domain_projection} of the main paper.  Thus
$X_1,X_2,\ldots$ are i.i.d. standard normal variables and
\begin{equation*}
\overline X_n=\frac1n\sum_{i=1}^n X_i,
\qquad
S_{n,0}^2=\frac1n\sum_{i=1}^n X_i^2.
\end{equation*}
We use
\begin{equation*}
R_{n,i}^{(\mu)}:=X_i-\overline X_n,
\qquad
R_{n,i}^{(\sigma)}:=\frac{X_i}{S_{n,0}}.
\end{equation*}
The projections $\Pi_\mu$ and $\Pi_\sigma$ are defined in
\eqref{main-eq:data_Pi_mu} and \eqref{main-eq:data_Pi_sigma}.

\begin{lemma}\label{supp-lem:supp_one_parameter_projection}
Let $r_0,\ldots,r_M$ be real-valued functions on $\mathbb R$.  Assume that each $r_k$ is
twice continuously differentiable and that $r_k$, $r_k'$ and $r_k''$ have at most polynomial
growth.  Then, for each $k=0,\ldots,M$,
\begin{align}
\frac1{\sqrt n}\sum_{i=1}^n r_k(R_{n,i}^{(\mu)})
&=
\frac1{\sqrt n}\sum_{i=1}^n(\Pi_\mu r_k)(X_i)+o_{\mathbb P}(1),
\label{supp-eq:supp_projection_mu}\\
\frac1{\sqrt n}\sum_{i=1}^n r_k(R_{n,i}^{(\sigma)})
&=
\frac1{\sqrt n}\sum_{i=1}^n(\Pi_\sigma r_k)(X_i)+o_{\mathbb P}(1).
\label{supp-eq:supp_projection_sigma}
\end{align}
The remainders are uniform over $k=0,\ldots,M$.  Moreover, for $\alpha\in\{\mu,\sigma\}$,
\begin{equation*}
\frac1n\sum_{i=1}^n r_k^2(R_{n,i}^{(\alpha)})
\xrightarrow[n\to\infty]{\mathbb P}
\mathbb E[r_k^2(X)],
\qquad k=0,\ldots,M.
\end{equation*}
\end{lemma}

\begin{proof}
For mean estimation, Taylor's formula gives
\begin{equation*}
r_k(X_i-\overline X_n)
=
r_k(X_i)-\overline X_n r_k'(X_i)+\varepsilon_{n,i,k}^{(\mu)}.
\end{equation*}
Since $r_k''$ has polynomial growth and $\overline X_n=O_{\mathbb P}(n^{-1/2})$,
\begin{equation*}
\frac1{\sqrt n}\sum_{i=1}^n\varepsilon_{n,i,k}^{(\mu)}=o_{\mathbb P}(1),
\end{equation*}
uniformly over $k=0,\ldots,M$.  Hence
\begin{align*}
\frac1{\sqrt n}\sum_{i=1}^n r_k(R_{n,i}^{(\mu)})
&=
\frac1{\sqrt n}\sum_{i=1}^n r_k(X_i)
-
\sqrt n\,\overline X_n
\left(\frac1n\sum_{i=1}^n r_k'(X_i)\right)
+o_{\mathbb P}(1)\\
&=
\frac1{\sqrt n}\sum_{i=1}^n
\left\{r_k(X_i)-\mathbb E[r_k'(X)]X_i\right\}
+o_{\mathbb P}(1).
\end{align*}
By Gaussian integration by parts,
\begin{equation*}
\mathbb E[r_k'(X)]=\mathbb E[Xr_k(X)]=\langle r_k,q_1\rangle_\Phi.
\end{equation*}
This proves \eqref{supp-eq:supp_projection_mu}.

For variance estimation with known mean,
\begin{equation*}
\sqrt n(S_{n,0}-1)
=
\frac1{2\sqrt n}\sum_{i=1}^n(X_i^2-1)+o_{\mathbb P}(1).
\end{equation*}
On the event $|S_{n,0}-1|\le1/2$,
\begin{equation*}
\frac{X_i}{S_{n,0}}=X_i-X_i(S_{n,0}-1)+\Delta_{n,i},
\qquad
|\Delta_{n,i}|\le C|X_i|(S_{n,0}-1)^2.
\end{equation*}
Taylor's formula and the polynomial-growth assumptions give, uniformly over $k=0,\ldots,M$,
\begin{equation*}
r_k(R_{n,i}^{(\sigma)})
=
r_k(X_i)-X_i r_k'(X_i)(S_{n,0}-1)+\varepsilon_{n,i,k}^{(\sigma)},
\end{equation*}
where, for some $C,L<\infty$,
\begin{equation*}
|\varepsilon_{n,i,k}^{(\sigma)}|
\le
C(1+|X_i|^L)(S_{n,0}-1)^2.
\end{equation*}
Since $S_{n,0}-1=O_{\mathbb P}(n^{-1/2})$ and $n^{-1}\sum_i(1+|X_i|^L)=O_{\mathbb P}(1)$,
\begin{equation*}
\frac1{\sqrt n}\sum_{i=1}^n\varepsilon_{n,i,k}^{(\sigma)}=o_{\mathbb P}(1).
\end{equation*}
Consequently,
\begin{align*}
\frac1{\sqrt n}\sum_{i=1}^n r_k(R_{n,i}^{(\sigma)})
&=
\frac1{\sqrt n}\sum_{i=1}^n r_k(X_i)
-
\sqrt n(S_{n,0}-1)
\left(\frac1n\sum_{i=1}^n X_i r_k'(X_i)\right)
+o_{\mathbb P}(1)\\
&=
\frac1{\sqrt n}\sum_{i=1}^n
\left\{r_k(X_i)-\frac12\mathbb E[Xr_k'(X)](X_i^2-1)\right\}
+o_{\mathbb P}(1).
\end{align*}
Since
\begin{equation*}
\mathbb E[Xr_k'(X)]=\mathbb E[(X^2-1)r_k(X)],
\end{equation*}
we obtain \eqref{supp-eq:supp_projection_sigma}.  The convergence of the empirical second
moments follows by applying the same argument to $r_k^2$.
\end{proof}
\section{Infinite-rank details and kernel examples}
\label{supp-sec:infinite_rank_and_examples}
\setcounter{equation}{0}

The main paper states the tail condition and a sufficient summability criterion in Propositions~\ref{main-prop:data_tail_sufficient_main} and \ref{main-prop:data_infinite_rank_limit_main}.  We repeat the notation in order to verify the condition in the examples used in the paper.  Let $K$ be a symmetric canonical kernel with expansion in $L^2(\Phi\otimes\Phi)$.  We assume that each finite collection of functions appearing below satisfies the smoothness assumptions of Lemma~\ref{main-lem:data_projection_empirical_sums}.
\begin{equation}
\label{supp-eq:data_infinite_kernel_expansion}
K(x,y)=\sum_{k=0}^{\infty}\gamma_k r_k(x)r_k(y),
\end{equation}
where $\{r_k\}_{k\ge0}$ is an orthonormal system of real, centered functions in $L^2(\Phi)$.  Let
\begin{equation*}
K_M(x,y):=\sum_{k=0}^{M}\gamma_k r_k(x)r_k(y).
\end{equation*}
For $M<L$ put $K_{M,L}(x,y):=\sum_{k=M+1}^{L}\gamma_k r_k(x)r_k(y)$.  For full-standardization residuals, the passage from $K_M$ to $K$ follows if
\begin{equation}
\label{supp-eq:data_tail_condition_probability}
\lim_{M\to\infty}
\limsup_{n\to\infty}
\mathbb P\left(\left|nU_n^{(\mu,\sigma)}(K-K_M)\right|>\varepsilon\right)=0,
\qquad \varepsilon>0.
\end{equation}
For the full-standardization residuals set
\begin{equation*}
S_{n,k}:=\frac1{\sqrt n}\sum_{i=1}^n r_k(R_{n,i}),
\qquad
Q_{n,k}:=\frac1n\sum_{i=1}^n r_k^2(R_{n,i}),
\end{equation*}
where $R_{n,i}$ is defined in Section~\ref{main-sec:data_domain_projection}.  Define
\begin{equation*}
A_k(n_0):=
\sup_{n\ge n_0}
\left\{
\mathbb E[S_{n,k}^2]+\mathbb E[Q_{n,k}]
\right\},
\qquad n_0\ge4.
\end{equation*}

\begin{lemma}
\label{supp-lem:data_tail_sufficient}
If
\begin{equation}
\label{supp-eq:data_tail_sufficient_condition}
\sum_{k=0}^{\infty}|\gamma_k|A_k(n_0)<\infty
\end{equation}
for some $n_0\ge4$, then \eqref{supp-eq:data_tail_condition_probability} holds.
\end{lemma}

\begin{proof}
For $M<L$, the exact identity used in the proof of Theorem~\ref{main-thm:data_finite_rank_projection_limit} gives
\begin{equation*}
nU_n^{(\mu,\sigma)}(K_{M,L})
=
\frac{n}{n-1}\sum_{k=M+1}^{L}\gamma_k\{S_{n,k}^2-Q_{n,k}\}.
\end{equation*}
Therefore, for $n\ge n_0$,
\begin{equation*}
\mathbb E\left|nU_n^{(\mu,\sigma)}(K_{M,L})\right|
\le
2\sum_{k=M+1}^{L}|\gamma_k|A_k(n_0).
\end{equation*}
The finite partial sums are Cauchy in $L^1$, uniformly over $n\ge n_0$.  Their $L^1$ limit equals $nU_n^{(\mu,\sigma)}(K-K_M)$ by the absolute-continuity argument in Proposition~\ref{main-prop:data_tail_sufficient_main}, and
\begin{equation*}
\limsup_{n\to\infty}
\mathbb E\left|nU_n^{(\mu,\sigma)}(K-K_M)\right|
\le
2\sum_{k=M+1}^{\infty}|\gamma_k|A_k(n_0).
\end{equation*}
The right-hand side tends to zero by \eqref{supp-eq:data_tail_sufficient_condition}; Markov's inequality proves \eqref{supp-eq:data_tail_condition_probability}.
\end{proof}

For a non-polynomial example, consider
\begin{equation*}
r_k(x)=\sqrt2\cos(k\pi\Phi(x)),\qquad k\ge1.
\end{equation*}
These functions are centered and orthonormal because $\Phi(X)$ is uniform on $(0,1)$. They satisfy $\|r_k\|_\infty\le\sqrt2$, $\|r_k'\|_\infty\le Ck$ and $\|r_k''\|_\infty\le Ck^2$. The constants follow from boundedness of $f_0$ and $f_0'$.

The required bounds are bounds on expectations, not only remainder estimates in probability. Let $R$ denote any of the three residual arrays and set $F_{n,k}=n^{-1/2}\sum_i r_k(R_{n,i})$. For full standardization, write $X_i=\overline X_n+S_nR_{n,i}$. Independence of the mean, radius and direction gives
\begin{equation*}
\EE[(R_{n,1}-X_1)^2]=n^{-1}+\EE[(S_n-1)^2]\le 3/n.
\end{equation*}
Here $(\sqrt x-1)^2\le(x-1)^2$ and $nS_n^2\sim\chi^2_{n-1}$ give the bound. For mean-only and variance-only residuals the corresponding bounds are respectively $1/n$ and $2/n$. Consequently, since $\EE[r_k(X)]=0$,
\begin{equation*}
|\EE[F_{n,k}]|\le\sqrt n\|r_k'\|_\infty\EE|R_{n,1}-X_1|
\le\sqrt3\|r_k'\|_\infty.
\end{equation*}
For $n\ge4$, the residual Jacobians and the Gaussian Poincar\'e inequality give $\Var(F_{n,k})\le4\|r_k'\|_\infty^2$. For the full and variance-only arrays this uses respectively $\EE[S_n^{-2}]=n/(n-3)\le4$ and $\EE[S_{n,0}^{-2}]=n/(n-2)\le2$; the centered-array Jacobian has norm one. For the radial arrays, the cutoff construction in Section~\ref{supp-sec:full_derivative_tail} applies in dimensions $n-1$ and $n$, respectively; its squared-gradient contribution is $O(\varepsilon^{n-3})$ and $O(\varepsilon^{n-2})$. The centered array is a linear function of $X$ and needs no radial cutoff. Thus, uniformly over the three arrays,
\begin{equation*}
\sup_{n\ge4}\left\{\EE[F_{n,k}^2]+\EE\left[n^{-1}\sum_i r_k^2(R_{n,i})\right]\right\}
\le2+7\|r_k'\|_\infty^2\le C(1+k^2).
\end{equation*}
It follows that $\sum_{k\ge1}k^2|\gamma_k|<\infty$ is sufficient. In particular, the condition holds for $\gamma_k=a^k$, $0<a<1$. This provides a non-polynomial example with a uniform expectation bound.

\section{Tail condition for the odd part of the rho-kernel under full standardization}
\label{supp-sec:tail_full_Krho}
\setcounter{equation}{0}

\begin{lemma}\label{supp-lem:data_standardized_residual_density_bound}
For $n\ge4$, the density of $R_{n,1}$ is bounded by $Cf_0(x)$, with a numerical constant
$C$ independent of $n$.
\end{lemma}

\begin{proof}
The random variable $R_{n,1}$ has density
\begin{equation*}
f_n(x)=c_n\left(1-\frac{x^2}{n-1}\right)_+^{(n-4)/2},
\qquad |x|<\sqrt{n-1}.
\end{equation*}
For $0\le u<1$, the inequality $\log(1-u)\le -u$ gives
\begin{equation*}
\left(1-\frac{x^2}{n-1}\right)^{(n-4)/2}\exp\!\left\{\frac{x^2}{2}\right\}
\le
\exp\!\left\{\frac{3x^2}{2(n-1)}\right\}\le \exp\!\left\{3/2\right\}.
\end{equation*}
Here $c_n=\Gamma\{(n-1)/2\}/[\sqrt{\pi(n-1)}\Gamma\{(n-2)/2\}]$. Log-convexity of the gamma function gives $\Gamma(z+1/2)^2\le z\Gamma(z)^2$. With $z=(n-2)/2$, this yields $\sqrt{2\pi}c_n\le\sqrt{(n-2)/(n-1)}\le1$. Thus one may take $C=\exp\!\{3/2\}$, also outside the support where the density is zero.
\end{proof}

\begin{lemma}\label{supp-lem:data_g_derivative_bound}
For each fixed $0<\rho<1$ there exists $C_\rho<\infty$ such that
\begin{equation*}
\int_{-\infty}^{\infty}\{g_m'(x;\rho)\}^2f_0(x)\,dx
\le C_\rho(m+1),
\qquad m=0,1,2,\ldots .
\end{equation*}
\end{lemma}

\begin{proof}
Put
\begin{equation*}
A_\rho:=\frac{\rho}{2(1-\rho)},
\qquad
D_\rho:=\left(\frac{1+\rho}{1-\rho}\right)^{1/2},
\qquad
c_\rho:=\left(\frac{1+\rho}{1-\rho}\right)^{1/4}.
\end{equation*}
Then
\begin{equation*}
g_m(x;\rho)=c_\rho\exp\!\left\{-A_\rho x^2\right\}h_m(D_\rho x).
\end{equation*}
Since $h_m'=\sqrt m\,h_{m-1}$, with the second term absent when $m=0$,
\begin{equation*}
g_m'(x;\rho)
=
c_\rho\exp\!\left\{-A_\rho x^2\right\}
\left\{-2A_\rho xh_m(D_\rho x)+D_\rho\sqrt m\,h_{m-1}(D_\rho x)\right\}.
\end{equation*}
By $(a+b)^2\le2a^2+2b^2$, it is enough to bound the two resulting integrals.  With
$y=D_\rho x$, the identity
\begin{equation*}
c_\rho^2\exp\!\left\{-2A_\rho x^2\right\}f_0(x)\,dx
=
f_0(y)\,dy
\end{equation*}
follows from $D_\rho^2=(1+\rho)/(1-\rho)$ and $2A_\rho=\rho/(1-\rho)$.  Hence
\begin{equation*}
c_\rho^2\int x^2\exp\!\left\{-2A_\rho x^2\right\}h_m^2(D_\rho x)f_0(x)\,dx
=
D_\rho^{-2}\int y^2h_m^2(y)f_0(y)\,dy.
\end{equation*}
The Hermite recurrence $yh_m(y)=\sqrt{m+1}h_{m+1}(y)+\sqrt m h_{m-1}(y)$ gives
\begin{equation*}
\int y^2h_m^2(y)f_0(y)\,dy=2m+1.
\end{equation*}
For $m\ge1$, also
\begin{equation*}
c_\rho^2\int\exp\!\left\{-2A_\rho x^2\right\}h_{m-1}^2(D_\rho x)f_0(x)\,dx=1,
\end{equation*}
and for $m=0$ the corresponding term is zero.

Combining the preceding bounds gives
\begin{equation*}
\int\{g_m'(x;\rho)\}^2f_0(x)\,dx
\le C_\rho(m+1),
\end{equation*}
where $C_\rho<\infty$ depends only on $\rho$.
\end{proof}

\begin{proposition}\label{supp-prop:data_Krho_tail_verified}
Let $r_j=g_{2j+1}(\cdot;\rho)$ and $\gamma_j=\rho^{2j+1}$, $j=0,1,2,\ldots$.  For the
fully standardized residuals, condition \eqref{supp-eq:data_tail_sufficient_condition} holds with
$n_0=4$.  Hence \eqref{supp-eq:data_tail_condition_probability} holds for the odd part of $K_\rho$.
\end{proposition}

\begin{proof}
It suffices to show that $A_j(4)\le C_\rho(j+1)$.  Put
\begin{equation*}
F_{n,j}:=\frac1{\sqrt n}\sum_{i=1}^n g_{2j+1}(R_{n,i};\rho).
\end{equation*}
The function $g_{2j+1}$ and its derivative belong to $L^2(\Phi)$. The Sobolev approximation in the proof of Proposition~\ref{main-prop:derivative_tail_criterion} in Section~\ref{supp-sec:full_derivative_tail} therefore applies with $r=g_{2j+1}$ and gives
\begin{equation*}
\operatorname{Var}(F_{n,j})\le \mathbb E\|\nabla F_{n,j}\|^2.
\end{equation*}
Since $g_{2j+1}$ is odd and the standardized residual vector is symmetric,
$\mathbb E[F_{n,j}]=0$.  The Jacobian of $R_n=(R_{n,1},\ldots,R_{n,n})^{\top}$ is
\begin{equation*}
S_n^{-1}\left(P-\frac1nR_nR_n^{\top}\right),
\end{equation*}
where $P$ is the orthogonal projection onto the hyperplane orthogonal to $(1,\ldots,1)$.  The
operator norm of the matrix in parentheses is at most one.  Thus
\begin{equation*}
\|\nabla F_{n,j}\|^2
\le
S_n^{-2}\frac1n\sum_{i=1}^n\{g_{2j+1}'(R_{n,i};\rho)\}^2.
\end{equation*}
Moreover $R_n=(R_{n,1},\ldots,R_{n,n})^{\top}$ is independent of $S_n$ and
$nS_n^2\sim\chi^2_{n-1}$, so $\mathbb E[S_n^{-2}]=n/(n-3)\le4$ for $n\ge4$.
Lemmas~\ref{supp-lem:data_standardized_residual_density_bound} and \ref{supp-lem:data_g_derivative_bound} imply
\begin{equation*}
\sup_{n\ge4}\mathbb E[F_{n,j}^2]\le C_\rho(j+1).
\end{equation*}
The same density bound and the orthonormality of $g_{2j+1}$ give
\begin{equation*}
\sup_{n\ge4}\mathbb E\left[\frac1n\sum_{i=1}^n g_{2j+1}^2(R_{n,i};\rho)\right]
\le C.
\end{equation*}
Therefore $A_j(4)\le C_\rho(j+1)$.  Since $\sum_{j\ge0}\rho^{2j+1}(j+1)<\infty$,
Lemma~\ref{supp-lem:data_tail_sufficient} applies.
\end{proof}
\section{The odd part of the rho-kernel}
\label{supp-sec:K_rho_data_domain}

Let $B_\rho$ denote the operator induced by $K_\rho$:
\begin{equation*}
(B_\rho f)(x)=\int_{-\infty}^{\infty}K_\rho(x,y)f(y)\,d\Phi(y).
\end{equation*}
On the odd subspace of $L^2(\Phi)$, the even terms in $K_\rho$ do not contribute, and
\begin{equation*}
B_{\rho,\mathrm{odd}}g_{2j+1}(\cdot;\rho)
=
\rho^{2j+1}g_{2j+1}(\cdot;\rho),
\qquad j=0,1,2,\ldots .
\end{equation*}
The tail condition for this odd component under full standardization is established in Section~\ref{supp-sec:tail_full_Krho}.

\begin{proposition}\label{supp-prop:data_mean_score_g_basis}
For $j=0,1,2,\ldots$, define
\begin{equation*}
d_j(\rho):=\langle q_1,g_{2j+1}(\cdot;\rho)\rangle_\Phi.
\end{equation*}
Then
\begin{equation}
\label{supp-eq:data_dj_formula}
d_j(\rho)
=
(1-\rho)^{3/4}(1+\rho)^{3/4}
\frac{\sqrt{(2j+1)!}}{2^j j!}\rho^j,
\qquad j=0,1,2,\ldots .
\end{equation}
In particular $d_j(\rho)>0$ for all $j\ge0$ and $0<\rho<1$.
\end{proposition}

\begin{proof}
By \eqref{main-eq:gn_en_link},
\begin{equation*}
g_{2j+1}(x;\rho)
=
c_\rho\exp\!\left\{-\frac{\rho}{2(1-\rho)}x^2\right\}
h_{2j+1}\left(\sqrt{\frac{1+\rho}{1-\rho}}\,x\right).
\end{equation*}
Set $A_\rho:=\rho/(2(1-\rho))$ and $D_\rho:=\sqrt{(1+\rho)/(1-\rho)}$.  The generating
function gives
\begin{equation*}
\mathbb E\left[X\exp\!\left\{-A_\rho X^2\right\}\exp\!\left\{sD_\rho X-s^2/2\right\}\right]
=
s(1-\rho)\sqrt{1+\rho}\exp\!\left\{\frac{\rho s^2}{2}\right\}.
\end{equation*}
The coefficient of $s^{2j+1}$ on the right-hand side is
\begin{equation*}
(1-\rho)\sqrt{1+\rho}\frac{1}{j!}\left(\frac{\rho}{2}\right)^j.
\end{equation*}
Comparing coefficients and using $h_{2j+1}=\mathrm{He}_{2j+1}/\sqrt{(2j+1)!}$ gives
\eqref{supp-eq:data_dj_formula} after multiplication by $c_\rho$.
\end{proof}

\begin{proposition}\label{supp-prop:data_domain_mean_secular}
Let $\omega_{m,\mathrm{odd}}^{(\mu,\sigma)}(\rho)$, $m\ge1$, denote the positive odd eigenvalues
of $\Pi_{\mu,\sigma}B_\rho\Pi_{\mu,\sigma}$, ordered decreasingly.  Then
$\omega_{m,\mathrm{odd}}^{(\mu,\sigma)}(\rho)$ is the unique root in
\begin{equation}
\label{supp-eq:data_domain_intervals}
\omega\in(\rho^{2m+1},\rho^{2m-1}),
\qquad m\ge1,
\end{equation}
of
\begin{equation}
\label{supp-eq:data_domain_mean_secular}
\sum_{j=0}^{\infty}\frac{d_j^2(\rho)}{\omega-\rho^{2j+1}}=0.
\end{equation}
Consequently,
\begin{equation*}
\rho>\omega_{1,\mathrm{odd}}^{(\mu,\sigma)}(\rho)>\rho^3>
\omega_{2,\mathrm{odd}}^{(\mu,\sigma)}(\rho)>\rho^5>\cdots .
\end{equation*}
\end{proposition}

\begin{proof}
On the odd subspace, $q_2$ is orthogonal to every odd function.  Thus $\Pi_{\mu,\sigma}$
coincides with $\Pi_\mu$ there.  Also
\begin{equation*}
B_{\rho,\mathrm{odd}}
=
\sum_{j=0}^{\infty}\rho^{2j+1}g_{2j+1}\otimes g_{2j+1},
\qquad
q_1=\sum_{j=0}^{\infty}d_j(\rho)g_{2j+1}.
\end{equation*}
An eigenfunction $f=\sum_j c_jg_{2j+1}$ with positive eigenvalue $\omega$ belongs to $q_1^\perp$, since the range of $\Pi_\mu B_{\rho,\mathrm{odd}}\Pi_\mu$ is contained in that subspace. Consequently,
\begin{equation*}
(B_{\rho,\mathrm{odd}}-\omega I)f=cq_1
\end{equation*}
for some $c$. If $\omega=\rho^{2\ell+1}$, the $\ell$th coordinate gives $cd_\ell=0$, hence $c=0$. The remaining coordinates imply that $f$ is a multiple of $g_{2\ell+1}$, contrary to $f\perp q_1$ and $d_\ell\ne0$. Thus no pole is an eigenvalue. Away from the poles, $c\ne0$ and the coordinate equation gives \eqref{supp-eq:data_domain_mean_secular}.

Conversely, for a positive off-pole solution $\omega$ of \eqref{supp-eq:data_domain_mean_secular}, set
\begin{equation*}
f=\sum_{j\ge0}\frac{d_j(\rho)}{\rho^{2j+1}-\omega}\,g_{2j+1}.
\end{equation*}
The coefficients are square summable because $\sum_jd_j^2=1$ and $\omega$ is separated from the pole set. The secular equation gives $f\perp q_1$, and $(B_{\rho,\mathrm{odd}}-\omega I)f=q_1$ gives the projected eigenvalue equation. The left-hand side of \eqref{supp-eq:data_domain_mean_secular} is strictly decreasing on every interval in \eqref{supp-eq:data_domain_intervals}, with limits $+\infty$ and $-\infty$ at the endpoints; it is positive for $\omega>\rho$. Thus these intervals contain all positive eigenvalues, exactly one in each, and the coordinate equation proves simplicity.
\end{proof}

\subsection{Equivalence with the frequency-domain equation}
\label{supp-subsec:data_frequency_equivalence}

\begin{proposition}\label{supp-prop:data_frequency_odd_equivalence}
Let $0<\rho<1$.  A positive number $\omega$ satisfies \eqref{supp-eq:data_domain_mean_secular} if and only if
\begin{equation*}
\lambda=(1-\rho)\omega
\end{equation*}
satisfies
\begin{equation*}
1+
\sum_{j=0}^{\infty}
\frac{a_{2j+1,1}^2(\rho)}{\lambda-(1-\rho)\rho^{2j+1}}=0.
\end{equation*}
Consequently,
\begin{equation*}
\lambda_{m,\mathrm{odd}}^{(\mu,\sigma)}(\rho)
=
(1-\rho)\omega_{m,\mathrm{odd}}^{(\mu,\sigma)}(\rho),
\qquad m\ge1.
\end{equation*}
The same relation holds in the mean-only odd sector.
\end{proposition}

\begin{proof}
Fix $\omega>0$ outside $\{\rho^{2j+1}:j\ge0\}$ and set $\lambda=(1-\rho)\omega$. The coefficient formulas in Propositions~\ref{main-prop:overlaps_closed_form} and \ref{supp-prop:data_mean_score_g_basis} give
\begin{equation*}
a_{2j+1,1}^2(\rho)=(1-\rho)\rho^{2j+1}d_j^2(\rho),\qquad j\ge0.
\end{equation*}
Also $\sum_jd_j^2(\rho)=\|q_1\|_\Phi^2=1$, by Parseval's identity in the odd Hermite basis. This normalization can be checked directly: with $x=\rho^2$ and $c_j=\binom{2j}{j}4^{-j}$,
\begin{equation*}
\sum_{j\ge0}d_j^2(\rho)
=(1-x)^{3/2}\sum_{j\ge0}(2j+1)c_jx^j
=(1-x)^{3/2}(1-x)^{-3/2}=1.
\end{equation*}
The pole set, completed by zero, is compact and does not contain $\omega$. Hence its distance from $\omega$ is positive. Both series in the following calculation are therefore absolutely convergent:
\begin{align*}
1+\sum_{j\ge0}\frac{a_{2j+1,1}^2(\rho)}{\lambda-(1-\rho)\rho^{2j+1}}
&=1+\sum_{j\ge0}\frac{\rho^{2j+1}d_j^2(\rho)}{\omega-\rho^{2j+1}}\\
&=1-\sum_{j\ge0}d_j^2(\rho)
+\omega\sum_{j\ge0}\frac{d_j^2(\rho)}{\omega-\rho^{2j+1}}\\
&=\omega\sum_{j\ge0}\frac{d_j^2(\rho)}{\omega-\rho^{2j+1}}.
\end{align*}
Since $\omega>0$, the two secular equations are equivalent. Multiplication by $1-\rho>0$ maps the interval $(\rho^{2m+1},\rho^{2m-1})$ onto $((1-\rho)\rho^{2m+1},(1-\rho)\rho^{2m-1})$. The unique-root statements in Proposition~\ref{supp-prop:data_domain_mean_secular} and Theorem~\ref{main-thm:mu_sigma_main} therefore identify roots with the same index $m$. The two odd restrictions are unchanged by the even scale correction, so the same identity holds in the mean-only case. Values at the poles are excluded by Proposition~\ref{supp-prop:data_domain_mean_secular} and the rank-one part of Proposition~\ref{main-prop:pole_criterion}.
\end{proof}
\section{Tail condition for one-parameter arrays}
\label{supp-sec:supp_tail_one_parameter}
\setcounter{equation}{0}

Let $K$ have the spectral expansion \eqref{supp-eq:data_infinite_kernel_expansion}.  For
$\alpha\in\{\mu,\sigma\}$ set
\begin{equation*}
S_{n,k}^{(\alpha)}:=\frac1{\sqrt n}\sum_{i=1}^n r_k(R_{n,i}^{(\alpha)}),
\qquad
Q_{n,k}^{(\alpha)}:=\frac1n\sum_{i=1}^n r_k^2(R_{n,i}^{(\alpha)}).
\end{equation*}
Define
\begin{equation*}
A_k^{(\alpha)}
:=
\sup_{n\ge4}
\left\{
\mathbb E\left[(S_{n,k}^{(\alpha)})^2\right]+
\mathbb E[Q_{n,k}^{(\alpha)}]
\right\}.
\end{equation*}

\begin{lemma}\label{supp-lem:supp_tail_one_parameter}
If
\begin{equation}
\label{supp-eq:supp_tail_sufficient_one_parameter}
\sum_{k=0}^{\infty}|\gamma_k|A_k^{(\alpha)}<\infty,
\end{equation}
then the full tails satisfy, for every $\varepsilon>0$,
\begin{equation*}
\lim_{M\to\infty}
\limsup_{n\to\infty}
\mathbb P
\left(
\left|nU_n^{(\alpha)}(K-K_M)\right|>\varepsilon
\right)=0,
\end{equation*}
Here $K-K_M$ is evaluated at the residual pairs; for $n\ge4$ these pairs have an absolutely continuous law, as in the main proof. Moreover, for each $M$ the variables $nU_n^{(\alpha)}(K-K_M)$ are well defined as $L^1$ limits of $nU_n^{(\alpha)}(K_{M,L})$ as $L\to\infty$.
\end{lemma}

\begin{proof}
The proof is the same as that of Lemma~\ref{supp-lem:data_tail_sufficient}.  For $M<L$,
\begin{equation*}
nU_n^{(\alpha)}(K_{M,L})
=
\frac{n}{n-1}\sum_{k=M+1}^{L}\gamma_k
\left\{(S_{n,k}^{(\alpha)})^2-Q_{n,k}^{(\alpha)}\right\}.
\end{equation*}
Taking expectations of absolute values and using $n/(n-1)\le2$ gives
\begin{equation*}
\sup_{n\ge4}\mathbb E\left|nU_n^{(\alpha)}(K_{M,L})\right|
\le
2\sum_{k=M+1}^{L}|\gamma_k|A_k^{(\alpha)}.
\end{equation*}
The estimate applied to differences of finite tails makes them Cauchy in $L^1$ for fixed $n$. Their limit equals the statistic evaluated at $K-K_M$: Section~\ref{supp-subsec:residual_pairs} gives absolute continuity of each fixed residual-pair law, and $L^2(\Phi\otimes\Phi)$ convergence then gives convergence in probability under that law. The finite sum over pairs preserves this convergence. Passing to $L\to\infty$ in the bound gives
\begin{equation*}
\sup_{n\ge4}\EE|nU_n^{(\alpha)}(K-K_M)|
\le2\sum_{k>M}|\gamma_k|A_k^{(\alpha)}\longrightarrow0.
\end{equation*}
Markov's inequality proves the asserted full-tail condition.
\end{proof}
\section{One-parameter tail verification for the odd part of the rho-kernel}
\label{supp-sec:one_parameter_Krho_tail}
\setcounter{equation}{0}

We verify the one-parameter versions of the tail condition for the odd part of
$K_\rho$.  The full-standardization verification is given in Proposition~\ref{supp-prop:data_Krho_tail_verified}.  We use the notation
\begin{equation*}
r_j=g_{2j+1}(\cdot;\rho),
\qquad
\gamma_j=\rho^{2j+1},
\qquad j=0,1,2,\ldots .
\end{equation*}

\begin{proposition}\label{supp-prop:supp_Krho_mean_tail}
For centered residuals $R_{n,i}^{(\mu)}=X_i-\overline X_n$, the one-parameter tail condition
\eqref{supp-eq:supp_tail_sufficient_one_parameter} holds with $\alpha=\mu$ for the odd expansion
of $K_\rho$.
\end{proposition}

\begin{proof}
Let
\begin{equation*}
F_{n,j}^{(\mu)}:=\frac1{\sqrt n}\sum_{i=1}^n g_{2j+1}(X_i-\overline X_n;\rho).
\end{equation*}
Since $g_{2j+1}$ is odd and the centered Gaussian vector is symmetric,
$\mathbb E[F_{n,j}^{(\mu)}]=0$.  By the Gaussian Poincar\'e inequality,
\begin{equation*}
\mathbb E[(F_{n,j}^{(\mu)})^2]
\le
\mathbb E\|\nabla F_{n,j}^{(\mu)}\|^2.
\end{equation*}
Differentiation gives
\begin{equation*}
\|\nabla F_{n,j}^{(\mu)}\|^2
\le
\frac1n\sum_{i=1}^n
\left\{g_{2j+1}'(X_i-\overline X_n;\rho)\right\}^2.
\end{equation*}
For $n\ge4$, the random variable $X_1-\overline X_n$ is normal with variance
$1-1/n\in[1/2,1]$; its density is bounded by $\sqrt2\,f_0$.  Therefore, by
Lemma~\ref{supp-lem:data_g_derivative_bound},
\begin{equation*}
\sup_{n\ge4}\mathbb E[(F_{n,j}^{(\mu)})^2]
\le
C_\rho(j+1).
\end{equation*}
The same density comparison gives
\begin{equation*}
\sup_{n\ge4}\mathbb E\left[
\frac1n\sum_{i=1}^n g_{2j+1}^2(X_i-\overline X_n;\rho)
\right]
\le C.
\end{equation*}
Thus $A_j^{(\mu)}\le C_\rho(j+1)$, and
\begin{equation*}
\sum_{j=0}^{\infty}\rho^{2j+1}A_j^{(\mu)}<\infty.
\end{equation*}
Lemma~\ref{supp-lem:supp_tail_one_parameter} proves the assertion.
\end{proof}

\begin{proposition}\label{supp-prop:supp_Krho_sigma_tail}
For the residuals $R_{n,i}^{(\sigma)}=X_i/S_{n,0}$ normalized by the scale estimate, the tail
condition \eqref{supp-eq:supp_tail_sufficient_one_parameter} holds with $\alpha=\sigma$ for the
odd expansion of $K_\rho$.
\end{proposition}

\begin{proof}
Let
\begin{equation*}
F_{n,j}^{(\sigma)}:=\frac1{\sqrt n}\sum_{i=1}^n g_{2j+1}(X_i/S_{n,0};\rho).
\end{equation*}
The vector $(X_1/S_{n,0},\ldots,X_n/S_{n,0})$ has norm $\sqrt n$ and uniform direction on the sphere in $\mathbb R^n$. Its first coordinate has density
\begin{equation*}
\frac{\Gamma(n/2)}{\sqrt{\pi n}\,\Gamma((n-1)/2)}
(1-x^2/n)^{(n-3)/2},\qquad |x|<\sqrt n,
\end{equation*}
and zero elsewhere. The calculation in Lemma~\ref{supp-lem:residual_density_rate}, with radial dimension $d=n$, bounds it by $\exp\!\{3/2\}f_0$ for every $n\ge3$. Hence
\begin{equation*}
\sup_{n\ge3}\mathbb E\left[
\frac1n\sum_{i=1}^n g_{2j+1}^2(X_i/S_{n,0};\rho)
\right]
\le C.
\end{equation*}
The oddness of $g_{2j+1}$ gives $\EE[F_{n,j}^{(\sigma)}]=0$. For the sum, apply the Gaussian Poincar\'e inequality. The cutoff argument in Section~\ref{supp-sec:full_derivative_tail}, with $P=I_n$, applies here: the small-radius probability is $O(\varepsilon^n)$ and the additional squared-gradient integral is $O(\varepsilon^{n-2})\to0$ for $n>2$. The Jacobian of
$x\mapsto x/S_{n,0}$ is
\begin{equation*}
S_{n,0}^{-1}\left(I-\frac1nRR^{\top}\right),
\qquad
R=(X_1/S_{n,0},\ldots,X_n/S_{n,0})^{\top},
\end{equation*}
and the matrix in parentheses is an orthogonal projection.  Therefore
\begin{equation*}
\|\nabla F_{n,j}^{(\sigma)}\|^2
\le
S_{n,0}^{-2}\frac1n\sum_{i=1}^n
\left\{g_{2j+1}'(X_i/S_{n,0};\rho)\right\}^2.
\end{equation*}
The radius $S_{n,0}$ is independent of the direction and $nS_{n,0}^2\sim\chi_n^2$, so
$\mathbb E[S_{n,0}^{-2}]=n/(n-2)\le3$ for $n\ge3$.  Combining this with the density
comparison and Lemma~\ref{supp-lem:data_g_derivative_bound} gives
\begin{equation*}
\sup_{n\ge3}\mathbb E[(F_{n,j}^{(\sigma)})^2]
\le C_\rho(j+1).
\end{equation*}
Thus $A_j^{(\sigma)}\le C_\rho(j+1)$ and
\begin{equation*}
\sum_{j=0}^{\infty}\rho^{2j+1}A_j^{(\sigma)}<\infty.
\end{equation*}
Lemma~\ref{supp-lem:supp_tail_one_parameter} proves the result.
\end{proof}

\Needspace{8\baselineskip}
\section{Numerical evaluation and null calibration}
\label{supp-sec:numerical_use}
\setcounter{equation}{0}

\subsection{Finite Hermite approximation}
All numerical results concern the fully standardized statistic $T_{n,\beta}^{(\mu,\sigma)}$ in the normalized EP scale. They do not concern the shifted $U$-statistic limit of Corollary~\ref{main-cor:full_U_V_limit}. Let $P_M$ denote projection onto $\mathrm{span}\{e_0,\ldots,e_{M-1}\}$. The matrix of $P_M A^{(\mu,\sigma)}P_M$ is
\begin{equation}\label{supp-eq:numerical_Hermite_matrix}
A_M=\mathrm{diag}\{(1-\rho)\rho^j:0\le j<M\}
-\sum_{k=0}^2 a_k^{(M)}(a_k^{(M)})^\top,
\qquad a_k^{(M)}:=(a_{0,k},\ldots,a_{M-1,k})^\top.
\end{equation}
This is the Rayleigh--Ritz approximation in the Gaussian Hermite basis; see Ebner et~al.~\cite[Sections~3 and 4.5]{EbnerJimenezGameroMilosevic2025} for applications to covariance operators and BHEP statistics. The coefficients of Proposition~\ref{main-prop:overlaps_closed_form} determine all entries explicitly. Diagonalization includes both parity subspaces and remains applicable when an eigenvalue coincides with an unperturbed pole.

Let $\lambda_j$ be the ordered positive eigenvalues of $A^{(\mu,\sigma)}$, and let $\lambda_j^{[M]}$ be the ordered eigenvalues of $A_M$, padded with zeros for $j\ge M$. Positivity and the min--max principle give $0\le\lambda_j^{[M]}\le\lambda_j$. Moreover,
\begin{align*}
0\le\sum_{j\ge0}(\lambda_j-\lambda_j^{[M]})
&=\operatorname{tr}(A^{(\mu,\sigma)})-\operatorname{tr}(A_M)\\
&=\sum_{j\ge M}\left\{(1-\rho)\rho^j-\sum_{k=0}^2a_{j,k}^2(\rho)\right\}
\le\rho^M.
\end{align*}
The equality is the trace identity in the Hermite basis; the final bound follows by omitting the nonnegative squared overlaps. Coupling the two quadratic forms with the same independent normal variables therefore gives
\begin{equation*}
\EE\left|\sum_{j\ge0}\lambda_jN_j^2-\sum_{j=0}^{M-1}\lambda_j^{[M]}N_j^2\right|\le\rho^M.
\end{equation*}
If $H$ and $H_M$ denote their respective distribution functions, then, for $\varepsilon>0$,
\begin{equation*}
H_M(x-\varepsilon)-\rho^M/\varepsilon\le H(x)\le H_M(x).
\end{equation*}
These bounds concern the exact eigenvalues of $A_M$. For the four values of $\rho$ considered below, we use $M=160$, giving $\rho^{160}\le2^{-160}$. For inversion of the quadratic-form distribution, only computed eigenvalues greater than $10^{-14}$ are retained.

\subsection{Eigenvalues and percentage points}
Table~\ref{supp-tab:first10} lists the first ten eigenvalues in decreasing order for $\rho=0.2,0.3,0.4,0.5$, with rank one corresponding to index zero in the main text. By \eqref{main-eq:full_parity_indexing}, odd-numbered and even-numbered ranks correspond respectively to odd and even eigenfunctions of the fully standardized covariance operator. The associated weight parameter is $\beta=\sqrt\rho/(1-\rho)$.

The percentage points in Table~\ref{supp-tab:quantiles} were obtained by numerical inversion of the finite weighted chi-square distribution. The calculations used R~4.5.3 and \texttt{CompQuadForm}~1.4.4, with the algorithm of Davies~\cite{Davies1980} and Imhof~\cite{Imhof1961} inversion as an alternative. The requested probability and root-finding tolerances were $10^{-10}$ and $10^{-9}$, respectively.

The matrix eigenvalues were checked against 240-point Gauss--Hermite quadrature of the covariance operator. The maximum absolute difference for the first ten eigenvalues was below $1.3\times10^{-15}$.

\begin{table}[htbp]
\centering
\caption{The first ten eigenvalues of the fully standardized covariance operator in the normalized EP scale.}
\label{supp-tab:first10}
\begin{tabular}{rcccc}
\hline
Rank & $\rho=0.2$ & $\rho=0.3$ & $\rho=0.4$ & $\rho=0.5$\\
\hline
1 & $1.5348745\times10^{-2}$ & $4.3082082\times10^{-2}$ & $8.1676872\times10^{-2}$ & $1.2194605\times10^{-1}$\\
2 & $5.2162634\times10^{-3}$ & $2.1200089\times10^{-2}$ & $5.1126249\times10^{-2}$ & $9.0125820\times10^{-2}$\\
3 & $4.4347270\times10^{-4}$ & $2.9038280\times10^{-3}$ & $1.0241549\times10^{-2}$ & $2.5130509\times10^{-2}$\\
4 & $1.3130431\times10^{-4}$ & $1.2676992\times10^{-3}$ & $5.8130446\times10^{-3}$ & $1.7234335\times10^{-2}$\\
5 & $1.5299698\times10^{-5}$ & $2.2736910\times10^{-4}$ & $1.4447350\times10^{-3}$ & $5.6412450\times10^{-3}$\\
6 & $4.1800254\times10^{-6}$ & $9.2325557\times10^{-5}$ & $7.7081260\times10^{-4}$ & $3.6818227\times10^{-3}$\\
7 & $5.6211994\times10^{-7}$ & $1.8852182\times10^{-5}$ & $2.1404172\times10^{-4}$ & $1.3162385\times10^{-3}$\\
8 & $1.4561880\times10^{-7}$ & $7.2863136\times10^{-6}$ & $1.0930755\times10^{-4}$ & $8.2829957\times10^{-4}$\\
9 & $2.1276266\times10^{-8}$ & $1.6075746\times10^{-6}$ & $3.2522294\times10^{-5}$ & $3.1369709\times10^{-4}$\\
10 & $5.3072199\times10^{-9}$ & $5.9947450\times10^{-7}$ & $1.6075325\times10^{-5}$ & $1.9193174\times10^{-4}$\\
\hline
\end{tabular}
\end{table}
\begin{table}[htbp]
\centering
\caption{Asymptotic percentage points for the fully standardized EP statistic.}
\label{supp-tab:quantiles}
\begin{tabular}{cccc}
\hline
$\rho$ & 0.90 & 0.95 & 0.99\\
\hline
0.2 & 0.04929173 & 0.06651722 & 0.10910634\\
0.3 & 0.15432205 & 0.20328447 & 0.32245797\\
0.4 & 0.32696859 & 0.42309367 & 0.65224449\\
0.5 & 0.54902691 & 0.69873256 & 1.05060541\\
\hline
\end{tabular}
\end{table}

\subsection{Normal-null simulations}
For each $n\in\{50,100,200\}$, we generated 50,000 independent samples from the standard normal law. Each sample was used for all four values of $\rho$, so comparisons across weights are paired. Both parameters were estimated, with the sample variance defined using divisor $n$.

For each standardized sample, the statistic was evaluated as
\begin{align*}
T_{n,\beta}^{(\mu,\sigma)}
&=\frac1n\sum_{j,k=1}^n
\exp\!\left\{-\frac{\beta^2}{2}(V_{n,j}-V_{n,k})^2\right\}\\
&\quad-\frac{2}{\sqrt{1+\beta^2}}\sum_{j=1}^n
\exp\!\left\{-\frac{\beta^2V_{n,j}^2}{2(1+\beta^2)}\right\}
+\frac{n}{\sqrt{1+2\beta^2}}.
\end{align*}
The test rejects when the statistic exceeds the asymptotic 0.95 percentage point reported in Table~\ref{supp-tab:quantiles}. If $\widehat p$ is the observed rejection proportion and $B=50{,}000$, its marginal Monte Carlo standard error is $\{\widehat p(1-\widehat p)/B\}^{1/2}$.

The rejection frequencies in Table~\ref{supp-tab:null_size} range from 0.04670 to 0.05054. The test is mildly conservative at $n=100$, with departures from 0.05 exceeding two Monte Carlo standard errors for each weight. At $n=200$, the frequencies lie between 0.04936 and 0.05054 and are close to the nominal level. These results indicate accurate null calibration for the sample sizes and fixed weights considered; no power comparison is made.

\Needspace{20\baselineskip}
\begin{table}[htbp]
\centering
\caption{Normal-null rejection proportions at nominal level 0.05, based on 50,000 replications per sample size. The last column gives the Monte Carlo standard error.}
\label{supp-tab:null_size}
\begin{tabular}{rrrr}
\hline
$n$ & $\rho$ & Rejection proportion & Standard error\\
\hline
50 & 0.2 & 0.04926 & 0.000968\\
50 & 0.3 & 0.04778 & 0.000954\\
50 & 0.4 & 0.04868 & 0.000962\\
50 & 0.5 & 0.04866 & 0.000962\\
100 & 0.2 & 0.04670 & 0.000944\\
100 & 0.3 & 0.04716 & 0.000948\\
100 & 0.4 & 0.04742 & 0.000950\\
100 & 0.5 & 0.04772 & 0.000953\\
200 & 0.2 & 0.04970 & 0.000972\\
200 & 0.3 & 0.04980 & 0.000973\\
200 & 0.4 & 0.04936 & 0.000969\\
200 & 0.5 & 0.05054 & 0.000980\\
\hline
\end{tabular}
\end{table}

\section{Operator facts used in the main paper}\label{supp-sec:supp_operator_facts}\setcounter{equation}{0}
\label{supp-subsec:operator_background}

Let $(\mathcal H,\langle\cdot,\cdot\rangle_{\mathcal H})$ be a real separable Hilbert space.
A bounded linear operator $A:\mathcal H\to\mathcal H$ is \emph{self-adjoint} if
$\langle Af,g\rangle_{\mathcal H}=\langle f,Ag\rangle_{\mathcal H}$ for all $f,g\in\mathcal H$ and
\emph{positive} if $\langle Af,f\rangle_{\mathcal H}\ge 0$ for all $f\in\mathcal H$.

\medskip
\noindent\textbf{Compact operators and the spectral theorem.}
An operator $A$ is \emph{compact} if it maps the unit ball of $\mathcal H$ to a relatively compact set.
If $A$ is compact and self-adjoint, then the nonzero spectrum of $A$ consists of real eigenvalues of finite
multiplicity with $0$ as the only possible accumulation point, and $\mathcal H$ admits an orthonormal basis
of eigenfunctions of $A$ (possibly completed by an orthonormal basis of $\ker(A)$).
For the spectral theorem and covariance operators, see
Horv{\'a}th and Kokoszka~\cite[Chapter~2]{horvath2012inference} or Hsing and Eubank~\cite[Sections~4.2--4.3]{hsing2015theoretical}.

\medskip
\noindent\textbf{Courant--Fischer variational principle.}
If $A$ is compact, self-adjoint and positive, list its positive eigenvalues in decreasing order, repeated according to multiplicity, and append zeros when its rank is finite. Then, for every integer $k$ with $0\le k<\dim\mathcal H$,
\begin{equation*}
\lambda_k
=
\max_{\substack{L\subset\mathcal H\\ \dim(L)=k+1}}
\ \min_{\substack{f\in L\\ \|f\|_{\mathcal H}=1}}
\langle Af,f\rangle_{\mathcal H}
=
\min_{\substack{L\subset\mathcal H\\ \dim(L)=k}}
\ \max_{\substack{f\perp L\\ \|f\|_{\mathcal H}=1}}
\langle Af,f\rangle_{\mathcal H}.
\end{equation*}
This min--max principle is commonly attributed to Fischer~\cite{fischer1905quadratische} and Courant~\cite{courant1920eigenwerte}; see, for example,
Teschl~\cite{teschl2014mathematical} for a proof in the compact-operator setting.

\medskip
\noindent\textbf{Hilbert--Schmidt and trace-class operators.}
If $\{e_j\}_{j\ge 1}$ is an orthonormal basis of $\mathcal H$, the operator $A$ is called
\emph{Hilbert--Schmidt} if $\sum_{j\ge 1}\|Ae_j\|_{\mathcal H}^2<\infty$, in which case the
Hilbert--Schmidt norm is given by
\begin{equation*}
\|A\|_{\mathrm{HS}}^2:=\sum_{j=1}^\infty \|Ae_j\|_{\mathcal H}^2,
\end{equation*}
and does not depend on the chosen basis. Every Hilbert--Schmidt operator is compact.
An operator is \emph{trace class} if $\sum_{j\ge1}\langle |A|e_j,e_j\rangle_{\mathcal H}<\infty$, where
$|A|=(A^*A)^{1/2}$. Its trace norm is $\|A\|_1:=\sum_j\langle |A|e_j,e_j\rangle_{\mathcal H}$. For such an operator, $\operatorname{tr}(A):=\sum_j\langle Ae_j,e_j\rangle_{\mathcal H}$. Every trace-class operator is Hilbert--Schmidt.
Covariance operators of square-integrable $\mathcal H$-valued random elements are positive, self-adjoint
and trace class; see again Horv{\'a}th and Kokoszka~\cite[Chapter~2]{horvath2012inference}.

\medskip
\noindent\textbf{Integral operators.}
For $\mathcal H=L^2(\mu)$ and a kernel $K\in L^2(\mu\otimes\mu)$, the integral operator
\begin{equation*}
(A_K f)(s):=\int K(s,t)\,f(t)\,\mu(dt)
\end{equation*}
is Hilbert--Schmidt with $\|A_K\|_{\mathrm{HS}}=\|K\|_{L^2(\mu\otimes\mu)}$; see
Hsing and Eubank~\cite[Sections~4.4 and 4.6]{hsing2015theoretical}.
If additionally $K(s,t)=K(t,s)$ $\mu\otimes\mu$-a.e., then $A_K$ is self-adjoint, and if $K$ is positive
definite then $A_K$ is positive.

\medskip
\noindent\textbf{Rank-one and finite-rank operators.}
For $g,h\in\mathcal H$ we write $g\otimes h$ for the rank-one operator
\begin{equation*}
(g\otimes h)f:=\langle f,h\rangle_{\mathcal H}\,g,\qquad f\in\mathcal H.
\end{equation*}
Finite sums of such terms are finite-rank operators.  In our setting, parameter estimation enters through
subtracting a small number of rank-one terms from an unperturbed covariance operator; the resulting eigenvalues
are then characterized by secular equations (Section~\ref{main-subsec:finite_rank_lemma}).

\subsection{Gaussian limits and quadratic forms}

The limit theory for the weighted $L^2$ statistics is formulated on the frequency-domain space $L^2(\varphi_\beta)$.

We shall use the following central limit theorem in a separable Hilbert space.

\begin{theorem}
\label{supp-thm:HilbertCLT}
Let $\mathcal H$ be a separable Hilbert space over $\mathbb R$.
Let $\{\xi_j\}_{j\ge1}$ be i.i.d.\ $\mathcal H$-valued random elements with $\EE[\xi_1]=0$ and
$\EE\|\xi_1\|_{\mathcal H}^2<\infty$.
Then
\begin{equation*}
\frac{1}{\sqrt n}\sum_{j=1}^n \xi_j \ \xrightarrow[n\to\infty]{d}\ \mathcal Z
\qquad\text{in }\mathcal H,
\end{equation*}
where $\mathcal Z$ is a centered Gaussian element in $\mathcal H$ with covariance operator
$C:\mathcal H\to\mathcal H$ defined by
\begin{equation*}
\langle C f, g\rangle_{\mathcal H}
=
\EE\bigl[\langle \xi_1,f\rangle_{\mathcal H}\langle \xi_1,g\rangle_{\mathcal H}\bigr],
\qquad f,g\in\mathcal H.
\end{equation*}
\end{theorem}

We shall use the following Karhunen--Lo\`eve representation for the squared norm.

\begin{theorem}
\label{supp-thm:KLquadratic}
Let $\mathcal Z$ be a centered Gaussian element in a separable Hilbert space $\mathcal H$ over $\mathbb R$,
with covariance operator $C$.
Let $\{\lambda_n\}_{n\ge0}$ be its nonzero eigenvalues, repeated by multiplicity and completed by zeros if the rank is finite.
Then
\begin{equation*}
\|\mathcal Z\|_{\mathcal H}^2\ \stackrel{d}{=}\ \sum_{n=0}^\infty \lambda_n\,N_n^2,
\end{equation*}
where $\{N_n\}_{n\ge0}$ are i.i.d.\ $\mathcal N(0,1)$ and the series converges almost surely and in $L^1$.
\end{theorem}

For these two results see, for example, Gin{\'e} and Nickl~\cite{gine2021mathematical}, Horv{\'a}th and Kokoszka~\cite{horvath2012inference}, and Hsing and Eubank~\cite{hsing2015theoretical}. Their application to characteristic functions uses the real space \eqref{main-eq:Hermitian_space}.

\section{A derivative estimate for the complete spectral tail}
\label{supp-sec:full_derivative_tail}
\setcounter{equation}{0}
This section proves Proposition~\ref{main-prop:derivative_tail_criterion}. It supplies a bound for centered functions of either parity. In combination with Corollary~\ref{main-cor:weighted_kernel_full_tail}, it verifies the tail condition for the whole Gaussian kernel. The earlier odd-sector estimates remain useful because they give bounds for the explicit functions $g_{2j+1}$.

\subsection{Residual-pair laws and kernel representatives}
\label{supp-subsec:residual_pairs}
Put $d=n-1$ and let $U$ be uniform on the unit sphere in the $d$-dimensional subspace $\boldsymbol1^\perp$. Then $R_n=\sqrt n U$. For distinct coordinates, the corresponding linear forms have Gram matrix $(d/n)\Sigma_d$, where
\begin{equation*}
\Sigma_d=\begin{pmatrix}1&-1/d\\-1/d&1\end{pmatrix}.
\end{equation*}
By an orthogonal change of coordinates, $(R_{n,1},R_{n,2})^\top$ has the law of $\sqrt d\,\Sigma_d^{1/2}(U_1,U_2)^\top$, with $U_1,U_2$ the first two coordinates of a uniform direction in $\mathbb R^d$. For $d>2$, integrating the remaining $d-2$ coordinates gives the density
\begin{equation}\label{supp-eq:residual_pair_density}
p_n(z)=\frac{\Gamma(d/2)}{\pi d\,\Gamma((d-2)/2)\sqrt{\det\Sigma_d}}
\left(1-\frac{z^\top\Sigma_d^{-1}z}{d}\right)^{(d-4)/2},
\quad z^\top\Sigma_d^{-1}z<d,
\end{equation}
and zero outside this ellipse. Its normalization follows by setting $y=d^{-1/2}\Sigma_d^{-1/2}z$ and integrating $(1-\|y\|^2)^{(d-4)/2}$ over the unit disc in polar coordinates. Since the normal product density is strictly positive, \eqref{supp-eq:residual_pair_density} proves absolute continuity with respect to $\Phi\otimes\Phi$. For $n=4$ the exponent is $-1/2$: the density is integrable but unbounded at the boundary. No bounded joint-density comparison is needed.

Here is the measure argument behind Proposition~\ref{main-prop:data_tail_sufficient_main}. Let $\mu=\Phi\otimes\Phi$, let $\nu_n$ be the law of one residual pair, and put $h_n=d\nu_n/d\mu$. If $K_L\to K$ in $L^2(\mu)$, then, for $\delta,B>0$,
\begin{equation*}
\nu_n\{|K_L-K|>\delta\}
\le B\delta^{-2}\|K_L-K\|_{L^2(\mu)}^2
+\int_{\{h_n>B\}}h_n\,d\mu.
\end{equation*}
Fix $n$, let $L\to\infty$, and then let $B\to\infty$. The last integral tends to zero because $h_n\in L^1(\mu)$. A union bound over the finitely many off-diagonal pairs gives convergence in probability of the kernel statistics for this fixed $n$. The separate summability estimate of Proposition~\ref{main-prop:data_tail_sufficient_main} supplies the uniform $L^1$ bound. These two arguments do not interchange the spectral and sample-size limits.

For completeness, centered residual pairs have a nonsingular Gaussian covariance matrix with diagonal $1-1/n$ and off-diagonal $-1/n$ when $n>2$. Variance-normalized pairs with known mean are two coordinates of a uniform sphere in dimension $n$ and have a density for $n>2$. This justifies the corresponding representative identifications in Lemma~\ref{supp-lem:supp_tail_one_parameter} for the stated cutoff $n_0\ge4$.

\subsection{The marginal density and the derivative estimate}

\begin{lemma}\label{supp-lem:residual_density_rate}
Let $f_n$ be the density of $R_{n,1}$ under full standardization. There is a numerical constant $C$ such that, for every $n\ge4$,
\begin{equation}\label{supp-eq:residual_density_rate}
0\le f_n(x)\le Cf_0(x),\qquad
\left|\frac{f_n(x)}{f_0(x)}-1\right|\le\frac{C}{n}(1+x^4),\qquad x\in\mathbb R.
\end{equation}
In particular, $\|f_n/f_0-1\|_\Phi\le C/n$.
\end{lemma}
\begin{proof}
Write $d=n-1\ge3$. The spherical marginal density is
\begin{equation*}
f_n(x)=\frac{\Gamma(d/2)}{\sqrt{\pi d}\,\Gamma((d-1)/2)}
(1-x^2/d)^{(d-3)/2},\qquad |x|<\sqrt d,
\end{equation*}
and is zero outside this interval. Put $a_d=\sqrt{2/d}\,\Gamma(d/2)/\Gamma((d-1)/2)$. Log-convexity of the gamma function, which follows by Cauchy--Schwarz from its integral representation, gives, for $z\ge1$,
\begin{equation*}
\Gamma(z+1/2)^2\le\Gamma(z)\Gamma(z+1)=z\Gamma(z)^2,
\qquad
\Gamma(z)^2\le\Gamma(z-1/2)\Gamma(z+1/2)
=\frac{\Gamma(z+1/2)^2}{z-1/2}.
\end{equation*}
Taking $z=(d-1)/2$ yields
\begin{equation*}
\sqrt{1-2/d}\le a_d\le\sqrt{1-1/d}\le1,
\qquad
|\log a_d|\le-\tfrac12\log(1-2/d)\le3/d.
\end{equation*}
For $u=x^2/d<1$, the inequality $\log(1-u)\le-u$ implies
\begin{equation*}
\frac{f_n(x)}{f_0(x)}
=a_d\exp\!\left\{x^2/2+(d-3)\log(1-u)/2\right\}
\le\exp\!\left\{3x^2/(2d)\right\}\le\exp\!\left\{3/2\right\}.
\end{equation*}
This proves the first bound uniformly in $d$. If $x^2\le d/2$, then
\begin{align*}
\left|\log\frac{f_n(x)}{f_0(x)}\right|
&\le\frac3d+\frac{3x^2}{2d}
+\frac{d-3}{2}|\log(1-u)+u|\\
&\le\frac3d+\frac{3x^2}{2d}+\frac{x^4}{2d}
\le\frac{C}{d}(1+x^4),
\end{align*}
where we used $|\log(1-u)+u|\le u^2/\{2(1-u)\}$. Since
$|\exp\!\{v\}-1|\le\max(1,\exp\!\{v\})|v|$ for real $v$, the uniform bound on the density ratio gives the second estimate in \eqref{supp-eq:residual_density_rate} on this region. If $x^2>d/2$, the ratio is still bounded, including outside the support, and
\begin{equation*}
\frac{1+x^4}{d}\ge\frac d4\ge\frac34.
\end{equation*}
The same estimate follows there after increasing $C$. Finally, $d=n-1\ge3n/4$ and
\begin{equation*}
\int_{\mathbb R}(1+x^4)^2\,d\Phi(x)=1+2\EE[X^4]+\EE[X^8]=112.
\end{equation*}
Replacing $d$ by $n$ and integrating the squared pointwise bound proves the asserted $L^2(\Phi)$ rate.
\end{proof}

\begin{proof}[Proof of Proposition~\ref{main-prop:derivative_tail_criterion}]
Let $C_0,C_1$ be constants, independent of $n\ge4$, such that
$f_n\le C_0f_0$ and $\|f_n/f_0-1\|_\Phi\le C_1/n$, as in Lemma~\ref{supp-lem:residual_density_rate}. Put
\begin{equation*}
F_n=n^{-1/2}\sum_{i=1}^n r(R_{n,i}),\qquad
Q_n=n^{-1}\sum_{i=1}^n r^2(R_{n,i}).
\end{equation*}
We bound the mean, variance and empirical second moment separately.

\emph{Mean and empirical second moment.} Exchangeability, centering of $r$ under $\Phi$, and Cauchy--Schwarz give
\begin{align*}
|\EE F_n|
&=\sqrt n\left|\int r(x)\left(\frac{f_n(x)}{f_0(x)}-1\right)d\Phi(x)\right|
\le C_1n^{-1/2}\|r\|_\Phi,\\
\EE Q_n&=\int r(x)^2 f_n(x)\,dx\le C_0\|r\|_\Phi^2.
\end{align*}
In particular, $F_n\in L^2$ for fixed $n$, since $F_n^2\le nQ_n$.

\emph{Derivative and variance.} Let $P=I_n-n^{-1}\boldsymbol1\boldsymbol1^\top$, set $Y=PX$, and write
\begin{equation*}
S_n=\|Y\|/\sqrt n,\qquad R_n=\sqrt n\,Y/\|Y\|.
\end{equation*}
On $\{Y\ne0\}$, differentiation yields
\begin{equation*}
\frac{\partial R_n}{\partial X}=S_n^{-1}Q,
\qquad Q=P-n^{-1}R_nR_n^\top.
\end{equation*}
The identities $PR_n=R_n$ and $R_n^\top R_n=n$ imply $Q^\top=Q$, $Q^2=Q$ and $QR_n=0$. Hence $\|Q\|_{\mathrm{op}}\le1$ and, with $r'(R_n)$ denoting the vector of coordinatewise derivatives,
\begin{equation*}
\nabla F_n=\frac1{\sqrt n S_n}Q\,r'(R_n),\qquad
\|\nabla F_n\|^2\le S_n^{-2}\frac1n\sum_{i=1}^n|r'(R_{n,i})|^2,
\qquad Y^\top\nabla F_n=0.
\end{equation*}
The vector $Y$ is standard Gaussian on $\boldsymbol1^\perp$, of dimension $d=n-1$. Its polar-coordinate density factors into a uniform direction and a radial density proportional to $u^{d-1}\exp\!\{-u^2/2\}$. Thus $R_n$ and $S_n$ are independent and $nS_n^2\sim\chi_d^2$. Direct integration gives
\begin{equation*}
\EE S_n^{-2}
=\frac{n\Gamma((d-2)/2)}{2\Gamma(d/2)}
=\frac{n}{n-3}\le4.
\end{equation*}
Using this independence, exchangeability and the density bound, we obtain
\begin{equation*}
\EE\|\nabla F_n\|^2
\le\EE S_n^{-2}\int |r'(x)|^2f_n(x)\,dx
\le4C_0\|r'\|_\Phi^2.
\end{equation*}
It remains to justify the Gaussian Poincar\'e inequality across $Y=0$. Choose a smooth function $\chi:[0,\infty)\to[0,1]$ with $\chi=0$ on $[0,1]$ and $\chi=1$ on $[2,\infty)$. Put
\begin{equation*}
\eta_\varepsilon(X)=\chi(\|PX\|/\varepsilon),
\qquad F_{n,\varepsilon}=\eta_\varepsilon F_n,
\end{equation*}
with $F_{n,\varepsilon}=0$ where $PX=0$. For fixed $n$, each residual belongs to $[-\sqrt{n-1},\sqrt{n-1}]$. The continuity of $r,r'$ therefore makes $F_{n,\varepsilon}$ a continuously differentiable function with bounded value and gradient. The Gaussian Poincar\'e inequality applies to it.

Let $M_{n,r}=\sqrt n\sup_{|x|\le\sqrt{n-1}}|r(x)|<\infty$. The radial density gives $\PP(\|PX\|\le2\varepsilon)\le C_n\varepsilon^d$ for $0<\varepsilon\le1$. Consequently,
\begin{align*}
\EE|F_{n,\varepsilon}-F_n|^2
&\le M_{n,r}^2\PP(\|PX\|\le2\varepsilon)\longrightarrow0,\\
\EE[F_n^2\|\nabla\eta_\varepsilon\|^2]
&\le\frac{M_{n,r}^2\|\chi'\|_\infty^2}{\varepsilon^2}
\PP(\|PX\|\le2\varepsilon)
\le C_{n,r}\varepsilon^{n-3}\longrightarrow0.
\end{align*}
Furthermore, $\nabla\eta_\varepsilon$ is parallel to $Y$, whereas $Y^\top\nabla F_n=0$. It follows that
\begin{equation*}
\EE\|\nabla F_{n,\varepsilon}-\nabla F_n\|^2
=\EE[(1-\eta_\varepsilon)^2\|\nabla F_n\|^2]
+\EE[F_n^2\|\nabla\eta_\varepsilon\|^2]\longrightarrow0.
\end{equation*}
The first term tends to zero by dominated convergence and the integrable gradient bound. Passing to the limit in the Poincar\'e inequality for $F_{n,\varepsilon}$ now proves
\begin{equation*}
\operatorname{Var}(F_n)\le\EE\|\nabla F_n\|^2
\le4C_0\|r'\|_\Phi^2.
\end{equation*}
The cutoff limit was taken for each fixed $n$; its constants do not enter this uniform estimate.

Combining the three bounds gives, for every $n\ge4$,
\begin{equation*}
\EE F_n^2+\EE Q_n
\le4C_0\|r'\|_\Phi^2+(C_0+C_1^2/n)\|r\|_\Phi^2,
\end{equation*}
which proves \eqref{main-eq:Gaussian_derivative_bound}. For the centered orthonormal functions $r_k$, this implies
$A_k(4)\le C(1+\|r_k'\|_\Phi^2)$. Multiplication by $|\gamma_k|$ and summation prove \eqref{main-eq:data_tail_sufficient_condition} under \eqref{main-eq:derivative_summability}; Proposition~\ref{main-prop:data_tail_sufficient_main} then gives the full-tail condition \eqref{main-eq:data_tail_condition_probability}.
\end{proof}

\section{Full empirical characteristic-function expansion}
\label{supp-sec:supp_full_ecf}
\setcounter{equation}{0}
The following gives the probabilistic part of Theorem~\ref{main-thm:mu_sigma_main}, with the notation of its setup.
\begin{proposition}\label{supp-prop:supp_full_ecf_limit}
Under independent standard normal observations,
\begin{equation}\label{supp-eq:supp_full_ecf_expansion}
\sqrt n(\psi_n-\phi_0)=n^{-1/2}\sum_{j=1}^n\zeta_{\mu,\sigma}(\cdot;X_j)+o_{\PP}(1)
\quad\text{in }\mathcal H_\beta,
\end{equation}
where $\zeta_{\mu,\sigma}(t;x)=\exp\!\left\{itx\right\}-\phi_0(t)-it\phi_0(t)x+t^2\phi_0(t)(x^2-1)/2$. The summands are centered and square integrable, and their covariance is \eqref{main-eq:mu_sigma_kernel}.
\end{proposition}
\begin{proof}

Let
\begin{equation*}
\phi_n(t):=\frac1n\sum_{j=1}^n\exp\!\left\{itX_j\right\},
\qquad
m_n(t):=\frac1n\sum_{j=1}^nX_j\exp\!\left\{itX_j\right\}.
\end{equation*}
For fixed $x,t\in\mathbb R$ set
\begin{equation*}
h_{t,x}(u,v):=\exp\!\left\{it\frac{x-u}{1+v}\right\},
\qquad (u,v)\in\mathbb R\times(-1,\infty).
\end{equation*}
Taylor's formula at $(0,0)$ gives, for $|v|\le1/2$,
\begin{equation*}
h_{t,x}(u,v)
=
\exp\!\left\{itx\right\}-it\exp\!\left\{itx\right\}u-itx\exp\!\left\{itx\right\}v
+R_{t,x}(u,v).
\end{equation*}
A direct differentiation gives
\begin{align*}
\partial_{uu}h_{t,x}(u,v)
&=-\frac{t^2}{(1+v)^2}h_{t,x}(u,v),\\
\partial_{uv}h_{t,x}(u,v)
&=\left\{\frac{it}{(1+v)^2}-\frac{t^2(x-u)}{(1+v)^3}\right\}h_{t,x}(u,v),\\
\partial_{vv}h_{t,x}(u,v)
&=\left\{\frac{2it(x-u)}{(1+v)^3}-\frac{t^2(x-u)^2}{(1+v)^4}\right\}h_{t,x}(u,v).
\end{align*}
Since $|h_{t,x}|=1$ and $|1+v|^{-k}\le2^k$ on $|v|\le1/2$, the Hessian norm is bounded by
$C(1+t^2)(1+(x-u)^2)$.  Hence, possibly changing $C$,
\begin{equation*}
|R_{t,x}(u,v)|
\le
C(1+t^2)(1+x^2+u^2)(u^2+v^2),
\qquad |v|\le1/2.
\end{equation*}
Apply the expansion with $u=\overline X_n$, $v=S_n-1$, $x=X_j$, and average over $j$.  On a set whose probability tends to one,
\begin{equation*}
\psi_n(t)
=
\phi_n(t)-it\overline X_n\phi_n(t)-it(S_n-1)m_n(t)+r_n(t),
\end{equation*}
where
\begin{equation*}
|r_n(t)|
\le
C(1+t^2)(\overline X_n^{\,2}+(S_n-1)^2)
\left(1+\overline X_n^{\,2}+\frac1n\sum_{j=1}^nX_j^2\right).
\end{equation*}
Since $\overline X_n=O_{\mathbb P}(n^{-1/2})$, $S_n-1=O_{\mathbb P}(n^{-1/2})$ and $n^{-1}\sum_jX_j^2=O_{\mathbb P}(1)$, while
$\int(1+t^2)^2\varphi_\beta(t)\,dt<\infty$, it follows that
\begin{equation*}
\sqrt n\|r_n\|_\beta=o_{\mathbb P}(1).
\end{equation*}
We next replace the empirical factors in the two linear terms by their deterministic limits.  By Fubini and the variance bound,
\begin{equation*}
\mathbb E\|t(\phi_n-\phi_0)\|_\beta^2
\le
\frac1n\int t^2\varphi_\beta(t)\,dt,
\end{equation*}
and, with $m(t):=\mathbb E[X\exp\!\left\{itX\right\}]=it\phi_0(t)$,
\begin{equation*}
\mathbb E\|t(m_n-m)\|_\beta^2
\le
\frac1n\int t^2\mathbb E[X^2]\varphi_\beta(t)\,dt.
\end{equation*}
Thus $\|t(\phi_n-\phi_0)\|_\beta=O_{\mathbb P}(n^{-1/2})$ and $\|t(m_n-m)\|_\beta=O_{\mathbb P}(n^{-1/2})$.  Since $\overline X_n$ and $S_n-1$ are $O_{\mathbb P}(n^{-1/2})$, the replacement errors are $o_{\mathbb P}(n^{-1/2})$ in $L^2(\varphi_\beta)$.  Moreover,
\begin{equation*}
S_n-1=\frac12(S_n^2-1)+O_{\mathbb P}(n^{-1}).
\end{equation*}
Consequently,
\begin{align*}
\psi_n(t)-\phi_0(t)
&=\phi_n(t)-\phi_0(t)-it\phi_0(t)\overline X_n
+\frac{t^2}{2}\phi_0(t)(S_n^2-1)+\widetilde r_n(t),\\
\sqrt n\|\widetilde r_n\|_\beta&=o_{\mathbb P}(1).
\end{align*}
Since
\begin{equation*}
S_n^2-1=\frac1n\sum_{j=1}^n(X_j^2-1)-\overline X_n^{\,2},
\qquad
\sqrt n\,\overline X_n^{\,2}=o_{\mathbb P}(1),
\end{equation*}
we obtain
\begin{equation*}
\sqrt n(\psi_n-\phi_0)
=
\frac1{\sqrt n}\sum_{j=1}^n\zeta_{\mu,\sigma}(\cdot;X_j)+o_{\mathbb P}(1)
\quad\hbox{in }L^2(\varphi_\beta),
\end{equation*}
where
\begin{equation*}
\zeta_{\mu,\sigma}(t;x)
:=
\exp\!\left\{itx\right\}-\phi_0(t)-it\phi_0(t)x+\frac{t^2}{2}\phi_0(t)(x^2-1).
\end{equation*}
The summands are centered and satisfy $\mathbb E\|\zeta_{\mu,\sigma}(\cdot;X)\|_\beta^2<\infty$, since
$|\zeta_{\mu,\sigma}(t;X)|\le 2+|t||X|+t^2(1+X^2)/2$ and the Gaussian weight has finite moments of all orders.  The Hilbert-space central limit theorem and Remark~\ref{main-rem:complex_H} give convergence to a centered Gaussian element $Z$ with covariance
\begin{equation*}
\mathbb E[Z(s)\overline{Z(t)}]
=
\mathbb E[\zeta_{\mu,\sigma}(s;X)\overline{\zeta_{\mu,\sigma}(t;X)}].
\end{equation*}
The identities
\begin{align*}
\mathbb E[\exp\!\left\{iuX\right\}]
&=\exp\!\left\{-u^2/2\right\},\\
\mathbb E[X\exp\!\left\{iuX\right\}]
&=iu\exp\!\left\{-u^2/2\right\},\\
\mathbb E[X^2\exp\!\left\{iuX\right\}]
&=(1-u^2)\exp\!\left\{-u^2/2\right\}.
\end{align*}
yield \eqref{main-eq:mu_sigma_kernel}.  The continuous mapping theorem and Theorem~\ref{supp-thm:KLquadratic} give \eqref{main-eq:mu_sigma_limit_series}.

The expansion established above is \eqref{supp-eq:supp_full_ecf_expansion}. All its terms are Hermitian symmetric, so the norm convergence is also convergence in the real space $\mathcal H_\beta$.
\end{proof}

\section{Verification of the coincidence at an even pole}
\label{supp-sec:pole_validation}
\setcounter{equation}{0}
Proposition~\ref{main-prop:even_pole_exists} proves existence by analytic inequalities, and Theorem~\ref{main-thm:unique_pole_sequence} establishes uniqueness. We locate the exceptional parameter by rational interval bounds and construct a corresponding eigenfunction.

For rational $\rho$, put $x=\rho^2$ and $c_k=\binom{2k}{k}4^{-k}$. Let $P_N(x)$ be the right-hand side of \eqref{main-eq:even_pole_E1} with the positive series stopped at $k=N$. The remainder satisfies
\begin{align}\label{supp-eq:pole_series_tail}
0\le E_1(\rho)-P_N(x)
&\le\frac{2(1-x)^{5/2}}{x^2(1-x^N)}\,y^{N+1}\nonumber\\
&\quad\times\left\{\frac{N^2}{1-y}+\frac{2Ny}{(1-y)^2}
+\frac{y(1+y)}{(1-y)^3}\right\},\qquad y=x^2.
\end{align}
Indeed $c_k\le1$ and $1-x^{k-1}\ge1-x^N$ for $k\ge N+1$; the remaining power series is summed exactly. All terms of $P_N$ and of the bound are rational multiples of $1$ or $\sqrt{1-x}$. Rational lower and upper bounds for the square root therefore give bounds for the finite sum and its remainder.

Using \eqref{supp-eq:pole_series_tail} with $N=80$ and rational bounds of width $10^{-60}$ for the square root gives
\begin{equation*}
E_1(0.71351870517303)<0,
\qquad E_1(0.71351870517305)>0.
\end{equation*}
The endpoints are exact rational numbers. The change of sign and Theorem~\ref{main-thm:unique_pole_sequence} identify the interval containing the unique zero.

\Needspace{12\baselineskip}
\subsection{An eigenfunction at the coincident value}
Let $\rho_*$ be the unique zero of $E_1$ located in Proposition~\ref{main-prop:even_pole_exists}. All coefficients below are evaluated at $\rho_*$. Put
\begin{equation*}
v_k=a_{2k,0}(1,b_k)^\top,\qquad q=(-b_1,1)^\top,\qquad
H_1=I_2+\sum_{k\ne1}\frac{v_kv_k^\top}{\kappa_1-\kappa_k}.
\end{equation*}
Since $q^\top H_1q=E_1(\rho_*)=0$, the vector $H_1q$ is parallel to $v_1$. Define
\begin{align*}
\ell_k&=-\frac{a_{2k,0}(b_k-b_1)}{\kappa_1-\kappa_k},\qquad k\ne1,\\
\ell_1&=\frac{v_1^\top H_1q}{\|v_1\|^2},\qquad
f_*(t)=\sum_{k\ge0}\ell_ke_{2k}(t;\rho_*).
\end{align*}
The coefficients are square summable, since $\kappa_1$ is isolated and $\sum_k\|v_k\|^2<\infty$. Moreover,
\begin{equation*}
\sum_{k\ge0}v_k\ell_k=v_1\ell_1-(H_1-I_2)q=q.
\end{equation*}
For $k\ne1$ the coordinate equation is $(\kappa_k-\kappa_1)\ell_k=v_k^\top q$; for $k=1$ it follows from $v_1^\top q=0$. Hence $A_{\mathrm{even}}^{(\mu,\sigma)}f_*=\kappa_1f_*$. Since $q\ne0$, the displayed overlap identity also proves that $f_*\ne0$. This constructs an eigenfunction of the infinite-dimensional operator.

In particular, the assertion is not that $e_2$ remains an eigenfunction. In fact,
\begin{equation*}
\langle(A_0-A^{(\mu,\sigma)})e_2,e_2\rangle_\beta
=a_{2,0}^2+a_{2,2}^2>0.
\end{equation*}
Thus the eigenvalue coincides with $\kappa_1$, but its eigendirection differs from that of the unperturbed operator.

\section{Multiplicity in the off-pole secular equation}
\label{supp-sec:secular_multiplicity}
\setcounter{equation}{0}
We complete the multiplicity assertion in Lemma~\ref{main-lem:finite_rank_secular}. Use $d_n$ and $v_n$ as in Proposition~\ref{main-prop:pole_criterion}, and put
\begin{equation*}
M(\lambda)=I_r+\sum_{n\ge0}\frac{v_nv_n^\top}{\lambda-d_n}.
\end{equation*}
Fix a real root $\lambda_0\notin\{0,d_0,d_1,\ldots\}$. The series and all its derivatives converge locally uniformly near $\lambda_0$, also for the complex variable, since $\sum_n\|v_n\|^2<\infty$ and the denominators are bounded away from zero. Put $V=\ker M(\lambda_0)$ and $k=\dim V$. The reconstruction in Lemma~\ref{main-lem:finite_rank_secular} is the linear bijection
\begin{equation*}
z\longmapsto f_z=-\sum_{n\ge0}\frac{v_n^\top z}{\lambda_0-d_n}e_n,
\qquad V\longrightarrow\ker(A-\lambda_0I).
\end{equation*}
For $z\in V\setminus\{0\}$, differentiation gives
\begin{equation*}
z^\top M'(\lambda_0)z
=-\sum_{n\ge0}\frac{(v_n^\top z)^2}{(\lambda_0-d_n)^2}
=-\|f_z\|_\beta^2<0.
\end{equation*}
Thus $M'(\lambda_0)$ restricted as a quadratic form to $V$ is negative definite. Choose orthonormal bases of $V$ and $V^\perp$. In the resulting block representation, the $V^\perp$ block $M_{22}(\lambda_0)$ is invertible, and both off-diagonal blocks vanish at $\lambda_0$. Taylor expansion of the Schur complement gives
\begin{equation*}
M_{11}(\lambda)-M_{12}(\lambda)M_{22}(\lambda)^{-1}M_{21}(\lambda)
=(\lambda-\lambda_0)B+O((\lambda-\lambda_0)^2),
\end{equation*}
where $B$ is the negative definite matrix of the restricted quadratic form. Consequently
\begin{equation*}
\det M(\lambda)
=(\lambda-\lambda_0)^k\det M_{22}(\lambda_0)\det B
+O((\lambda-\lambda_0)^{k+1}),
\end{equation*}
with nonzero leading coefficient. If $V^\perp=\{0\}$, omit that block and interpret its determinant as one. Hence the order of the zero is exactly $k$. Self-adjointness identifies the eigenvalue multiplicity with the eigenspace dimension.

In particular, a determinant root is counted with its order in the even spectrum. For a general parity-preserving operator, multiplicities are added if the blocks share an eigenvalue. Theorem~\ref{main-thm:simple_full_spectrum} rules this out for the fully standardized covariance; it is possible in the variance-only case, as shown in Proposition~\ref{supp-prop:variance_only_double}. At an unperturbed pole this argument is inapplicable: the bordered matrix of Proposition~\ref{main-prop:pole_criterion}, not $M(\lambda)$, gives the multiplicity.

\section{Total positivity, simple eigenvalues and parity}
\label{supp-sec:total_positivity}
\setcounter{equation}{0}
We prove the operator result used in Theorem~\ref{main-thm:simple_full_spectrum}. The argument concerns a consecutive initial block of powers of $st$. This restriction is important: the variance-only correction removes powers zero and two, but not power one.

\begin{theorem}\label{supp-thm:consecutive_corrections}
Let $r\ge1$ be an integer and let $\theta$ be an even, integrable weight, positive almost everywhere on $\mathbb R$. On $L^2(\theta;\mathbb R)$ consider the operator $A^{[r]}_\theta$ with kernel
\begin{equation}\label{supp-eq:consecutive_kernel}
\mathcal K_r(s,t)=\exp\!\left\{-\frac{s^2+t^2}{2}\right\}
\left(\exp\!\left\{st\right\}-\sum_{k=0}^{r-1}\frac{(st)^k}{k!}\right).
\end{equation}
Its positive eigenvalues can be ordered as $\nu_0>\nu_1>\cdots>0$, and all are simple. A real eigenfunction for $\nu_n$ has parity $(-1)^{r+n}$. In particular, the positive even and odd spectra are disjoint and alternate. These assertions do not concern the multiplicity of zero.
\end{theorem}

\begin{proof}
\emph{Positivity and trace class.} Write $d\mu(t)=\theta(t)\,dt$ and put
\begin{equation*}
v_k(t)=\exp\!\left\{-t^2/2\right\}\frac{t^k}{\sqrt{k!}},\qquad k\ge r.
\end{equation*}
Tonelli's theorem gives
\begin{equation*}
\sum_{k=r}^{\infty}\|v_k\|_\theta^2
=\int_{\mathbb R}\left(1-\exp\!\left\{-t^2\right\}
\sum_{k=0}^{r-1}\frac{t^{2k}}{k!}\right)d\mu(t)
\le\mu(\mathbb R)<\infty.
\end{equation*}
Since $\|v_k\otimes v_k\|_1=\|v_k\|_\theta^2$, the series of rank-one operators converges in trace norm. Moreover,
\begin{equation*}
\sum_{k=r}^{\infty}|v_k(s)v_k(t)|
\le\left(\sum_{k=r}^{\infty}v_k(s)^2\right)^{1/2}
\left(\sum_{k=r}^{\infty}v_k(t)^2\right)^{1/2}\le1.
\end{equation*}
The scalar series sums to \eqref{supp-eq:consecutive_kernel} and converges in $L^2(\mu\otimes\mu)$ by dominated convergence. Thus
\begin{equation*}
A^{[r]}_\theta=\sum_{k=r}^{\infty}v_k\otimes v_k,
\qquad
\langle A^{[r]}_\theta f,f\rangle_\theta
=\sum_{k=r}^{\infty}\langle f,v_k\rangle_\theta^2\ge0.
\end{equation*}
In particular,
\begin{equation*}
\ker A^{[r]}_\theta
=\left(\overline{\operatorname{span}}\{v_k:k\ge r\}\right)^\perp.
\end{equation*}
Every finite collection of the $v_k$ is linearly independent: a relation between them would give a polynomial vanishing $\mu$-almost everywhere, hence Lebesgue-almost everywhere because $\theta>0$ almost everywhere. The operator therefore has infinite rank. Until simplicity is proved, denote its positive eigenvalues, with multiplicities, by $\nu_0\ge\nu_1\ge\cdots>0$.

\emph{Strict positivity of the determinants.} Set $\sigma_r(t)=\operatorname{sgn}(t)^r$ for $t\ne0$ and $\sigma_r(0)=1$. Multiplication by $\sigma_r$ is unitary on $L^2(\mu;\mathbb R)$. The conjugated operator $B$ has kernel
\begin{equation}\label{supp-eq:signed_total_positive_kernel}
\widehat{\mathcal K}_r(s,t)=\sigma_r(s)\sigma_r(t)\mathcal K_r(s,t).
\end{equation}
We shall prove that, for all $m\ge1$ and all strictly increasing nonzero nodes $s_1<\cdots<s_m$ and $t_1<\cdots<t_m$,
\begin{equation}\label{supp-eq:strict_total_positivity}
\det[\widehat{\mathcal K}_r(s_i,t_j)]_{i,j=1}^m>0.
\end{equation}
We first establish the required strictness for the exponential kernel. Let $a_1<\cdots<a_q$ be real and let $p_\ell$ be polynomials of degree less than $m_\ell$, where $m_\ell\ge1$. A nonzero function
\begin{equation*}
F(u)=\sum_{\ell=1}^q p_\ell(u)\exp\!\left\{a_\ell u\right\}
\end{equation*}
has at most $N-1$ real zeros, counted with multiplicity, where $N=\sum_\ell m_\ell$. Here a nonzero coefficient vector cannot define the zero function: after division by the largest exponential and passage to $u\to+\infty$, the corresponding polynomial must vanish, and successive removal of that exponent proves linear independence. The zero bound follows by induction on $N$. For $N=1$, a nonzero constant multiple of an exponential has no real zero. For the induction step,
\begin{equation*}
\frac{d}{du}\left(\exp\!\left\{-a_1u\right\}F(u)\right)
=p_1'(u)+\sum_{\ell=2}^q
\{p_\ell'(u)+(a_\ell-a_1)p_\ell(u)\}
\exp\!\left\{(a_\ell-a_1)u\right\}.
\end{equation*}
The first polynomial has degree less than $m_1-1$ and is absent if $m_1=1$; the other degree bounds are unchanged. If this derivative is nonzero, induction gives at most $N-2$ zeros for it. Rolle's theorem, including multiplicities, then gives at most $N-1$ zeros for $F$. If the derivative vanishes identically, $F$ is a nonzero constant multiple of $\exp\!\left\{a_1u\right\}$ and has no zero.

For strictly increasing real nodes $x_1<\cdots<x_m$ and $y_1<\cdots<y_m$, the determinant $\det[\exp\!\left\{x_i y_j\right\}]$ is consequently nonzero: singularity would give a nonzero linear combination of the $m$ column functions with $m$ distinct zeros. Its sign is constant on the connected set of ordered nodes. For fixed $u_1<\cdots<u_m$, Taylor expansion at zero gives
\begin{equation*}
\lim_{\varepsilon\downarrow0}
\frac{\det[\exp\!\left\{\varepsilon u_i y_j\right\}]_{i,j=1}^m}
{\varepsilon^{m(m-1)/2}\prod_{i<j}(u_j-u_i)}
=\frac{\prod_{i<j}(y_j-y_i)}{\prod_{k=0}^{m-1}k!}>0.
\end{equation*}
Thus the exponential determinant is strictly positive for every pair of ordered node sets.

We also need derivative evaluations at zero. Append evaluations of orders $0,\ldots,r-1$ at zero to both node sets, and place both derivative blocks first. The resulting $(r+m)$-dimensional matrix is
\begin{equation*}
\mathcal E=
\begin{pmatrix}
D_r&T\\
S&E
\end{pmatrix},\qquad
\begin{gathered}
D_r=\operatorname{diag}(0!,\ldots,(r-1)!),\\
T_{k+1,j}=t_j^k,\quad S_{i,k+1}=s_i^k,\\
E_{ij}=\exp\!\left\{s_it_j\right\},\quad 0\le k<r.
\end{gathered}
\end{equation*}
In increasing-node order, its determinant is nonnegative. To verify this, first replace each derivative block by $r$ distinct ordered nodes tending to zero. Successive divided differences in those rows and columns divide the determinant by positive within-cluster Vandermonde factors. Their limits are the successive derivatives divided by positive factorials. Multiplication by these factorials gives the asserted nonnegative limit.

The limiting determinant is nonzero. If it were singular, some nonzero function in
\begin{equation*}
\operatorname{span}\{1,u,\ldots,u^{r-1},
\exp\!\left\{t_1u\right\},\ldots,\exp\!\left\{t_mu\right\}\}
\end{equation*}
would vanish to order at least $r$ at zero and at each of $s_1,\ldots,s_m$. These are at least $r+m$ zeros counted with multiplicity, contradicting the zero bound with $N=r+m$. Thus the ordered derivative-evaluation determinant is strictly positive.

Let $N_s$ and $N_t$ count the negative nodes in the two sets. Moving the derivative blocks from the first positions to their increasing-node positions uses $rN_s$ row interchanges and $rN_t$ column interchanges. Hence
\begin{equation*}
(-1)^{r(N_s+N_t)}\det\mathcal E>0,
\qquad
\det\mathcal E=\det D_r\,
\det\left[\exp\!\left\{s_it_j\right\}
-\sum_{k=0}^{r-1}\frac{s_i^kt_j^k}{k!}\right]_{i,j=1}^m.
\end{equation*}
Since $\prod_i\sigma_r(s_i)\prod_j\sigma_r(t_j)=(-1)^{r(N_s+N_t)}$, it follows that
\begin{equation*}
\det[\widehat{\mathcal K}_r(s_i,t_j)]
=\frac{\exp\!\left\{-\frac12\sum_i(s_i^2+t_i^2)\right\}}
{\prod_{k=0}^{r-1}k!}
(-1)^{r(N_s+N_t)}\det\mathcal E>0.
\end{equation*}
This proves \eqref{supp-eq:strict_total_positivity}.

\emph{Exterior powers and simplicity.} For $m\ge1$, let
\begin{equation*}
\mathcal W_m=\{s\in\mathbb R^m:s_1<\cdots<s_m,\ s_i\ne0\},
\qquad
\mu_m=\mu^{\otimes m}|_{\mathcal W_m}.
\end{equation*}
Tuples with a repeated or zero coordinate have $\mu^{\otimes m}$-measure zero. Restriction to $\mathcal W_m$, multiplied by $\sqrt{m!}$, is a unitary map from the antisymmetric subspace of $L^2(\mu^{\otimes m})$ onto $L^2(\mu_m)$. Denote this map by $V_m$. For an antisymmetric function $f$, averaging the tensor-product integral over permutations of the integration variables gives
\begin{align*}
(B^{\otimes m}f)(s)
&=\frac1{m!}\int_{\mathbb R^m}
\det[\widehat{\mathcal K}_r(s_i,t_j)]f(t)\,d\mu^{\otimes m}(t)\\
&=\int_{\mathcal W_m}
\det[\widehat{\mathcal K}_r(s_i,t_j)]f(t)\,d\mu_m(t).
\end{align*}
The second equality holds because the product of the determinant and $f$ is symmetric in $t$. Consequently the operator $B_m=V_m(B^{\otimes m}|_{\mathrm{alt}})V_m^{-1}$ has exactly the kernel
\begin{equation*}
\Delta_m(s,t)=\det[\widehat{\mathcal K}_r(s_i,t_j)]_{i,j=1}^m.
\end{equation*}
Expansion of the determinant and Cauchy--Schwarz give
\begin{equation*}
\int_{\mathcal W_m}\int_{\mathcal W_m}|\Delta_m(s,t)|^2\,
 d\mu_m(s)\,d\mu_m(t)
\le(m!)^2\|\widehat{\mathcal K}_r\|_{L^2(\mu\otimes\mu)}^{2m}<\infty.
\end{equation*}
Thus $B_m$ is compact and self-adjoint. It is positive because it is unitarily equivalent to the restriction of the positive operator $B^{\otimes m}$.

Choose real orthonormal eigenfunctions $h_0,h_1,\ldots$ of $B$ for the positive eigenvalues, still counted with multiplicities. For $I=(i_1<\cdots<i_m)$, put
\begin{equation*}
\Psi_I(s)=\det[h_{i_a}(s_b)]_{a,b=1}^m.
\end{equation*}
Expansion of both determinants, followed by integration, yields
\begin{align*}
\int_{\mathcal W_m}\Psi_I(s)\Psi_J(s)\,d\mu_m(s)
&=\det[\langle h_{i_a},h_{j_b}\rangle_\theta]_{a,b=1}^m
=\boldsymbol1_{\{I=J\}},\\
B_m\Psi_I&=\left(\prod_{a=1}^m\nu_{i_a}\right)\Psi_I.
\end{align*}
Completing the $h_n$ by an orthonormal basis of $\ker B$ gives a complete eigenbasis of $B$; its antisymmetrized tensor products give a complete eigenbasis of $B_m$. Products involving a nullspace vector have eigenvalue zero. The largest eigenvalue of $B_m$ is therefore
\begin{equation*}
\eta_m=\prod_{n=0}^{m-1}\nu_n>0.
\end{equation*}
It is simple. Indeed, if an eigenfunction $f$ for $\eta_m$ changed sign on sets of positive $\mu_m$-measure, then \eqref{supp-eq:strict_total_positivity} would give
\begin{align*}
\langle B_m|f|,|f|\rangle-\langle B_mf,f\rangle
&=4\int_{\{f>0\}}\int_{\{f<0\}}
\Delta_m(s,t)f(s)(-f(t))\,d\mu_m(t)\,d\mu_m(s)\\
&>0.
\end{align*}
This contradicts the variational characterization of $\eta_m$, since $\||f|\|=\|f\|$. Hence every such eigenfunction has one sign. The equation $f=\eta_m^{-1}B_mf$ makes that sign strict almost everywhere. Two orthogonal eigenfunctions at $\eta_m$ are therefore impossible.

If $\nu_k=\nu_{k+1}$ for some $k\ge0$, take $m=k+1$. The determinants formed from $h_0,\ldots,h_k$ and from $h_0,\ldots,h_{k-1},h_{k+1}$ are orthonormal eigenfunctions of $B_m$ with eigenvalue $\eta_m$. This contradicts its simplicity. Thus $\nu_0>\nu_1>\cdots>0$.

\emph{Parity.} Let $(Jh)(t)=h(-t)$. Evenness of $\theta$ and the identity $\widehat{\mathcal K}_r(-s,-t)=\widehat{\mathcal K}_r(s,t)$ imply $JB=BJ$. Since each positive eigenspace is one-dimensional and $J^2=I$, there are $\epsilon_n\in\{-1,1\}$ such that $h_n(-t)=\epsilon_nh_n(t)$ almost everywhere. The function
\begin{equation*}
\Psi_m(s)=\det[h_{i-1}(s_j)]_{i,j=1}^m
\end{equation*}
is an eigenfunction of $B_m$ for $\eta_m$, and therefore has one strict sign almost everywhere on $\mathcal W_m$. The transformation $\mathcal R_m(s)=(-s_m,\ldots,-s_1)$ preserves $\mathcal W_m$ and $\mu_m$. Reversing the columns of the determinant gives
\begin{equation*}
\Psi_m(\mathcal R_m s)
=(-1)^{m(m-1)/2}\left(\prod_{n=0}^{m-1}\epsilon_n\right)\Psi_m(s).
\end{equation*}
Its two sides have the same strict sign almost everywhere. Thus
\begin{equation*}
\prod_{n=0}^{m-1}\epsilon_n=(-1)^{m(m-1)/2},\qquad m\ge1.
\end{equation*}
Taking $m=1$, and then dividing consecutive identities, gives $\epsilon_0=1$ and $\epsilon_n=(-1)^n$. The original eigenfunction is $\sigma_rh_n$, whose parity is $(-1)^r\epsilon_n=(-1)^{r+n}$. Unitary conjugation preserves the eigenvalues and their multiplicities, proving all the assertions.
\end{proof}

For the classical determinant and eigenvalue theory of totally positive matrices, see Pinkus~\cite{Pinkus2010}. The proof above supplies the derivative-evaluation determinants and the compact-operator argument for the kernel \eqref{supp-eq:signed_total_positive_kernel}.

\subsection{Scope for one-parameter standardization}
For $r=1$, Theorem~\ref{supp-thm:consecutive_corrections} applies to $A_0-\phi_0^*\otimes\phi_0^*$, the known-parameter covariance. For $r=2$ it applies to $A^{(\mu)}$ of Theorem~\ref{supp-thm:mu_only_main_updated}. Thus the latter spectrum is simple and ordered as
\begin{equation*}
\lambda_{1,\mathrm{even}}^{(\mu)}>
\lambda_{1,\mathrm{odd}}^{(\mu)}>
\lambda_{2,\mathrm{even}}^{(\mu)}>
\lambda_{2,\mathrm{odd}}^{(\mu)}>\cdots>0.
\end{equation*}
The even spectrum of $A^{(\sigma)}$ is also simple, since it equals the fully standardized even spectrum. Its odd spectrum is the simple unperturbed odd spectrum. However, these two spectra need not be disjoint, as the following result shows.

\begin{proposition}\label{supp-prop:variance_only_double}
There exists $\rho\in(1/10,1/2)$ such that
\begin{equation*}
\lambda_{1,\mathrm{even}}^{(\sigma)}(\rho)=(1-\rho)\rho^3.
\end{equation*}
At such a value the variance-only operator has an eigenvalue of multiplicity exactly two.
\end{proposition}
\begin{proof}
The largest even eigenvalue is bounded above by the even trace. At $\rho=1/10$, where $\beta^2=10/81$, this gives
\begin{equation*}
\lambda_{1,\mathrm{even}}^{(\sigma)}
\le\frac{10}{11}-\frac{93159\sqrt{101}}{1030301}
<\frac9{10000}=(1-\rho)\rho^3.
\end{equation*}
The trace is $1/(1+\rho)-\|\phi_0^*\|_\beta^2-\|\phi_2^*\|_\beta^2$. For a lower bound at $\rho=1/2$, put $v_k(t)=\exp\!\left\{-t^2/2\right\}t^k/\sqrt{k!}$. The even operator is $\sum_{k\ge4,\ k\text{ even}}v_k\otimes v_k$. Its Rayleigh quotient at $v_4$ is therefore at least
\begin{equation*}
\|v_4\|_\beta^2+\frac{\langle v_4,v_6\rangle_\beta^2}{\|v_4\|_\beta^2}
=\frac{2506\sqrt5}{78125}>\frac1{16}=(1-\rho)\rho^3.
\end{equation*}
The two strict inequalities follow by squaring positive rational quantities. Under the isometry $t=\beta x$, the even kernels are continuous in $\rho$ in Hilbert--Schmidt norm on compact subintervals of $(0,1)$, by dominated convergence. Hence their largest eigenvalues are continuous. The intermediate value theorem gives equality somewhere between the endpoints. The two parity eigenvalues are each simple, so their common eigenvalue has total multiplicity two.
\end{proof}

\section{Unique and ordered parameters for coincidences at even poles}
\label{supp-sec:unique_poles}
\setcounter{equation}{0}
We prove Theorem~\ref{main-thm:unique_pole_sequence}. All series below, and their derivatives of each fixed order, converge locally uniformly on $0<x<1$; this follows either from a geometric majorant times a polynomial or from the finite negative part and the positive tail displayed below.

Put $x=\rho^2$, $c_k=\binom{2k}{k}4^{-k}$, and retain $E_j$, $b_j$ and $\kappa_j$ from the main paper. The generating function $\sum_{k\ge0}c_kx^k=(1-x)^{-1/2}$ gives, by two differentiations,
\begin{equation}\label{supp-eq:moment_pole_identity}
1+b_j^2=\frac{2(1-x)^{5/2}}{x}\sum_{k\ge0}c_k(k-j)^2x^k.
\end{equation}
Since
\begin{equation*}
\frac{a_{2k,0}^2}{1-\rho}=\sqrt{1-x}\,c_kx^{2k},\qquad
(b_k-b_j)^2=\frac{2(1-x)^2}{x}(k-j)^2,
\end{equation*}
combining \eqref{supp-eq:moment_pole_identity} with the pole criterion gives
\begin{align}\label{supp-eq:Lambert_pole_identity}
E_j(\rho)&=2(1-x)^{5/2}x^{j-1}\mathcal H_j(x),\\
\mathcal H_j(x)&=\sum_{h=1}^\infty c_{j+h}h^2\frac{x^h}{1-x^h}
-\sum_{h=1}^j c_{j-h}h^2\frac1{1-x^h}.
\label{supp-eq:Lambert_Hj}
\end{align}
For example, this follows by combining the two terms at each $k\ne j$ using
\begin{equation*}
x^k+\frac{x^{2k}}{x^j-x^k}=\frac{x^{j+k}}{x^j-x^k}.
\end{equation*}
The $k=j$ term in \eqref{supp-eq:moment_pole_identity} is zero and is omitted. The formula in \eqref{supp-eq:Lambert_Hj} then separates $k>j$ and $k<j$.

\subsection{A unique simple zero for each fixed pole}
For $h\ge1$, set $p_h(x)=x^h/(1-x^h)$ and $C_j=\sum_{h=1}^jc_{j-h}h^2>0$. Since $c_{k+1}/c_k=(2k+1)/(2k+2)<1$, we can write, for $j\ge1$,
\begin{equation}\label{supp-eq:normalized_pole_function}
\frac{\mathcal H_j(x)}{p_j(x)}
=-\frac{C_j}{p_j(x)}
+\sum_{h=1}^j(c_{j+h}-c_{j-h})h^2\frac{p_h(x)}{p_j(x)}
+\sum_{h=j+1}^\infty c_{j+h}h^2\frac{p_h(x)}{p_j(x)}.
\end{equation}
The first term is strictly increasing. For fixed $x\in(0,1)$,
\begin{equation*}
\frac{d\log p_h(x)}{d\log x}=\frac{h}{1-x^h}
\end{equation*}
is strictly increasing in the positive real variable $h$. Indeed, if $a=-\log x>0$, the derivative of $h/(1-\exp\!\left\{-ah\right\})$ has positive numerator $1-(1+ah)\exp\!\left\{-ah\right\}$. In particular,
\begin{equation*}
\frac{d}{dx}\left(\frac{p_h(x)}{p_j(x)}\right)
=\frac{p_h(x)}{xp_j(x)}
\left(\frac{h}{1-x^h}-\frac{j}{1-x^j}\right).
\end{equation*}
Thus $p_h/p_j$ decreases strictly for $h<j$ and increases strictly for $h>j$; for $h=j$ it is constant. The coefficients of the finite sum in \eqref{supp-eq:normalized_pole_function} are negative, whereas those of the last sum are positive. On a compact interval $[a,b]\subset(0,1)$, the derivative of the $h$th term in the final sum is bounded in absolute value by $C_{a,b,j}h^3b^h$. The geometric majorant is summable, so this differentiation is locally uniform. Every term other than $-C_j/p_j$ has nonnegative derivative, while
\begin{equation*}
\frac{d}{dx}\left(-\frac{C_j}{p_j(x)}\right)=C_jj x^{-j-1}>0.
\end{equation*}
Termwise differentiation consequently gives
\begin{equation}\label{supp-eq:normalized_pole_monotone}
\frac{d}{dx}\left\{\frac{\mathcal H_j(x)}{p_j(x)}\right\}>0,
\qquad 0<x<1.
\end{equation}
In particular, this is a statement about the normalized function, not an assertion that $E_j$ itself is globally monotone.

For $x\downarrow0$, dominate the positive series by its summable value at any fixed $b\in(0,1)$. Dominated convergence gives $\mathcal H_j(x)\to-C_j$, hence $\mathcal H_j(x)/p_j(x)\to-\infty$. At the other endpoint,
\begin{equation*}
\lim_{x\uparrow1}(1-x)\sum_{h=1}^j\frac{c_{j-h}h^2}{1-x^h}
=\sum_{h=1}^j c_{j-h}h<\infty.
\end{equation*}
For every integer $N\ge1$, positivity of the remaining series gives
\begin{equation*}
\liminf_{x\uparrow1}(1-x)\sum_{h\ge1}c_{j+h}h^2p_h(x)
\ge\sum_{h=1}^N h c_{j+h}.
\end{equation*}
The right-hand side tends to infinity with $N$, since $\sum_{h\ge1}h c_{j+h}=\infty$.
For completeness, $c_k\ge1/(2k+1)$: the central coefficient is the largest of the $2k+1$ binomial coefficients whose sum is $4^k$. This lower bound already proves the divergence. Since $(1-x)p_j(x)\to1/j$, we obtain $\mathcal H_j(x)/p_j(x)\to+\infty$. Thus \eqref{supp-eq:normalized_pole_monotone} gives exactly one zero $x_j\in(0,1)$. At that zero,
\begin{equation*}
\mathcal H_j'(x_j)=p_j(x_j)\left.\frac{d}{dx}\{\mathcal H_j(x)/p_j(x)\}\right|_{x=x_j}>0.
\end{equation*}
For $\rho_j=\sqrt{x_j}$, differentiation of \eqref{supp-eq:Lambert_pole_identity} at the zero gives
\begin{equation*}
E_j'(\rho_j)=4\rho_j(1-x_j)^{5/2}x_j^{j-1}\mathcal H_j'(x_j)>0.
\end{equation*} For $j=0$, the original expression for $E_0$ is strictly positive, because $1+b_0^2>0$ and every denominator $\kappa_0-\kappa_k$, $k>0$, is positive. Hence that pole never occurs in the even perturbed spectrum.

\subsection{Ordering and the only accumulation point}
Fix $0<x<1$, put $a=-\log x$, and define
\begin{equation*}
f_x(u)=\frac{u^2}{\exp\!\left\{au\right\}-1}\quad(u\ne0),\qquad f_x(0)=0.
\end{equation*}
Then \eqref{supp-eq:Lambert_Hj} is $\mathcal H_j(x)=\sum_{k\ge0}c_kf_x(k-j)$. For $h>0$, the function $h^2/(1-\exp\!\left\{-ah\right\})$ is strictly increasing: its derivative has numerator $h\{2-(2+ah)\exp\!\left\{-ah\right\}\}>0$. Therefore $f_x(k-j)>f_x(-j-1)$ for every $k\ge0$. Set $\delta_k=c_k-c_{k+1}>0$. Since $c_k=\prod_{\ell=1}^k\{1-(2\ell)^{-1}\}\to0$, $\sum_k\delta_k=1$. Shifting indices gives
\begin{equation}\label{supp-eq:pole_ordering_difference}
\mathcal H_{j+1}(x)-\mathcal H_j(x)
=f_x(-j-1)-\sum_{k\ge0}\delta_kf_x(k-j)<0.
\end{equation}
The negative terms are finite in number and the positive terms decay geometrically, so the rearrangement is justified. At $x_j$, \eqref{supp-eq:pole_ordering_difference} implies $\mathcal H_{j+1}(x_j)<0$. The unique-zero and sign-change result above yields $x_{j+1}>x_j$.

For fixed $x<1$, keeping just the negative term $h=j$ and using $c_k\le1$ gives
\begin{equation*}
\mathcal H_j(x)\le\sum_{h\ge1}\frac{h^2x^h}{1-x^h}-j^2
\le\frac{x(1+x)}{(1-x)^4}-j^2.
\end{equation*}
This is negative for all sufficiently large $j$, so $x_j>x$ eventually. It follows that $x_j\uparrow1$ and hence $\rho_j\uparrow1$.

Finally, the two-step bounds in the main paper locate a coincidence at $\kappa_j$ only at the $j$th even eigenvalue: the $(j-1)$st, when present, is strictly above $\kappa_j$, and the $(j+1)$st is strictly below it. Simplicity in the full standardized space follows from Theorem~\ref{main-thm:simple_full_spectrum}. The strict ordering of the $\rho_j$ also shows that at a fixed parameter at most one even eigenvalue equals an unperturbed even pole. This last assertion does not exclude coincidences between eigenvalues and unperturbed poles of the opposite parity.

\subsection{Rational enclosures without square roots}
Formula~\eqref{supp-eq:Lambert_Hj} gives an alternative enclosure that uses only rational quantities. Truncate the positive sum at $h=N\ge j$. Its remainder is nonnegative and at most
\begin{align}\label{supp-eq:Lambert_rational_tail}
\frac{c_{j+N+1}x^{N+1}}{1-x^{N+1}}
\left\{\frac{(N+1)^2}{1-x}
+\frac{2(N+1)x}{(1-x)^2}
+\frac{x(1+x)}{(1-x)^3}\right\}.
\end{align}
This follows from $c_{j+h}\le c_{j+N+1}$ and $1-x^h\ge1-x^{N+1}$ for $h>N$. For rational $\rho$, every quantity in \eqref{supp-eq:Lambert_Hj} and \eqref{supp-eq:Lambert_rational_tail} is rational. Evaluating the finite sums with rational bounds and adding the remainder bound gives opposite signs at the following endpoints, with $N=200,300,400$, respectively:
\begin{align*}
0.71351870517303&<\rho_1<0.71351870517305,\\
0.83019379455034&<\rho_2<0.83019379455037,\\
0.87931690045595&<\rho_3<0.87931690045599.
\end{align*}
All endpoints are exact rational numbers. Bounds of width $10^{-50}$ for each retained positive summand, together with \eqref{supp-eq:Lambert_rational_tail}, suffice to determine the signs. The uniqueness proved above then gives the stated enclosure for each exceptional parameter.

\section{Comparison with the parameter-substitution theorem of Arcones}
\label{supp-sec:Arcones_comparison}
\setcounter{equation}{0}
The relevant general result in the published article of Arcones~\cite{Arcones2007} is Theorem~3.3, on pp.~185--187. It concerns integrated products of generalized one-sample $U$-statistics with estimated parameters. Its normality-test application is Theorem~2.3, on pp.~181--182, with proof on pp.~191--193. For block size one that application includes the classical BHEP statistic. We give the correspondence explicitly; the covariance projection itself is not a new parameter-substitution principle.

Use a variance coordinate $v>0$ and put
\begin{equation*}
g(x,t;\mu,v)=\cos\{t(x-\mu)/\sqrt v\}+\sin\{t(x-\mu)/\sqrt v\}-\phi_0(t).
\end{equation*}
At $(\mu,v)=(0,1)$ its expectation under $X\sim\Phi$ is zero. With $G(t;\mu,v)=\EE[g(X,t;\mu,v)]$, direct normal integration gives
\begin{equation*}
G(t;\mu,v)=\exp\!\left\{-\frac{t^2}{2v}\right\}
\left(\cos\{t\mu/\sqrt v\}-\sin\{t\mu/\sqrt v\}\right)-\phi_0(t).
\end{equation*}
Consequently the two derivatives at $(0,1)$ are $-t\phi_0(t)$ and $t^2\phi_0(t)/2$. The sample mean and sample variance have influence functions $x$ and $x^2-1$. Arcones's corrected summand therefore becomes
\begin{equation}\label{supp-eq:Arcones_corrected_summand}
g_*(x,t)=\cos(tx)+\sin(tx)-\phi_0(t)-t\phi_0(t)x
+\frac{t^2}{2}\phi_0(t)(x^2-1).
\end{equation}
The expression in \eqref{supp-eq:Arcones_corrected_summand} is exactly $J_\beta^{-1}\zeta_{\mu,\sigma}(\cdot;x)$ from Proposition~\ref{supp-prop:supp_full_ecf_limit}. Since the weight is even, the odd cross terms integrate to zero, and
\begin{equation*}
\int g_*(x,t)g_*(y,t)\theta(t)\,dt
=\langle\zeta_{\mu,\sigma}(\cdot;x),\zeta_{\mu,\sigma}(\cdot;y)\rangle_\theta.
\end{equation*}
Thus the covariance and its limiting quadratic form agree at the operator level as well as at the level of the corrected summand.

In the Gaussian-weight case, boundedness of $g$, local differentiability in $(\mu,v)$, the normal moments and the finite weight moments verify the integrability and local continuity conditions in that application. More explicitly, in a neighborhood with $v$ bounded away from zero, the first parameter derivatives of $g$ are bounded by $C|t|(1+|x|)$. This gives the integrated quadratic Lipschitz bound. The expectation $G$ is differentiable in the weighted $L^2$ norm, and its derivative at the null has finite fourth weighted moment. Boundedness gives the remaining dominated-convergence conditions for the second and fourth powers of local increments. The asymptotic linear representations of the estimators are precisely those used above. These bounds verify the conditions needed for the correspondence. The proofs in the main paper are independent of Arcones's theorem.

For the $U$-form one must additionally remove the diagonal of the original statistic, as in \eqref{main-eq:UV_identity}. That diagonal has limit $\EE[K_\theta(X,X)]$, not the trace of the projected covariance. This agrees with the original-trace centering in Cupari{\'c} et~al.~\cite{CuparicMilosevicObradovic2022}. Accordingly, neither covariance projection nor original-trace subtraction is claimed as a new general result here.

The additional assertion in Proposition~\ref{main-prop:derivative_tail_criterion} is a bound on the residual empirical sums, uniform in sample size, expressed in the Gaussian $L^2$ norms of each function and its derivative. Combined with absolute summability it controls the entire observation-space spectral expansion, including signed kernels. Corollary~\ref{main-cor:weighted_kernel_full_tail} verifies that condition for the complete characteristic-function kernel without requiring closed even eigenfunctions. This is a different sufficient-condition argument, not a claim that the hypotheses of Theorem~3.3 of Arcones~\cite{Arcones2007} have been weakened in full generality. The pole classification and the simple-spectrum results concern the further spectral analysis of the common limiting operator.

\section{Proofs of supporting results}\label{supp-sec:preliminary_proofs}
\setcounter{equation}{0}
This section gives the proofs of Propositions~\ref{main-prop:Hermite_basis}, \ref{main-prop:data_infinite_rank_limit_main} and \ref{main-prop:data_frequency_factorization}, Lemmas~\ref{main-lem:finite_rank_secular}, \ref{main-lem:minmax_rank1} and \ref{main-lem:data_projection_empirical_sums}, and Corollary~\ref{main-cor:full_U_V_limit}. The statements and notation remain in the main paper.

\subsection{Gaussian Hermite diagonalization}\label{supp-subsec:proof_Hermite_basis}
\begin{proof}[Proof of Proposition~\ref{main-prop:Hermite_basis}]
For a polynomial $p$ and a standard normal variable $Z$, completion of squares gives
\begin{equation*}
A_0\!\left[\exp\!\left\{-(1-\rho)t^2/2\right\}p(t)\right](s)
=(1-\rho)\exp\!\left\{-(1-\rho)s^2/2\right\}
\EE[p(\rho s+\sqrt\rho Z)].
\end{equation*}
The Hermite generating function implies
\begin{equation*}
\EE[h_n(\rho u+\sqrt{1-\rho^2}Z)]=\rho^n h_n(u).
\end{equation*}
Since $\alpha_\rho\sqrt\rho=\sqrt{1-\rho^2}$, the choice $p(t)=h_n(\alpha_\rho t)$ proves the eigenvalue equation. For orthonormality, set $y=\alpha_\rho t$. Direct substitution gives
\begin{equation*}
c_\rho^2\exp\!\left\{-(1-\rho)t^2\right\}\varphi_\beta(t)\,dt=f_0(y)\,dy.
\end{equation*}
Consequently the map $p\mapsto c_\rho\exp\!\left\{-(1-\rho)t^2/2\right\}p(\alpha_\rho t)$ is a surjective isometry from $L^2(f_0)$ onto $L^2(\varphi_\beta)$. Completeness follows from completeness of the Hermite polynomials. The last assertion follows by $t=\beta x$.
\end{proof}

\subsection{Finite-rank secular equations}\label{supp-subsec:proof_finite_rank_secular}
\begin{proof}[Proof of Lemma~\ref{main-lem:finite_rank_secular}]
Let $\lambda\notin\{0\}\cup\{\lambda_n^{(0)}(\rho)\}$ and suppose $f\neq0$ satisfies $Af=\lambda f$.
Expand $f=\sum_{n\ge0}c_n e_n$ in the orthonormal basis of eigenfunctions of $A_0$. Taking inner products with $e_n$ gives
\begin{equation*}
(\lambda_n^{(0)}(\rho)-\lambda)\,c_n=\sum_{\ell=1}^r a_{n,\ell}\,\langle f,u_\ell\rangle_\beta,
\qquad n\ge0,
\end{equation*}
hence
\begin{equation*}
c_n=-\,\frac{1}{\lambda-\lambda_n^{(0)}(\rho)}\,\sum_{\ell=1}^r a_{n,\ell}\,\langle f,u_\ell\rangle_\beta.
\end{equation*}
Substituting into $\langle f,u_m\rangle_\beta=\sum_{n\ge0}c_n\, a_{n,m}$, $m=1,\ldots,r$, yields the linear system
$(I_r+G(\lambda))\,v=0$ for the vector $v=(\langle f,u_1\rangle_\beta,\dots,\langle f,u_r\rangle_\beta)^\top$.
Conversely, a nonzero solution $v$ defines $f$ by the displayed coefficients. The denominators are bounded away from zero, so $f\in L^2(\varphi_\beta;\mathbb R)$. The system gives $\langle f,u_m\rangle_\beta=v_m$, and hence $f\ne0$ and $Af=\lambda f$. The two constructions are inverse linear maps, so they also identify the eigenspace dimension. The series defining $G$, and their derivatives, converge absolutely and locally uniformly off $\{0\}\cup\{\lambda_n^{(0)}\}$, by Cauchy--Schwarz and Parseval's identity. The assertion about the order of a determinant zero is proved in Section~\ref{supp-sec:secular_multiplicity}.

For $r=1$, \eqref{main-eq:finite_rank_det} reduces to \eqref{main-eq:finite_rank_rank1}.
On any interval free of poles, the derivative exists and
\begin{equation*}
F'(\lambda)=-\sum_{n=0}^\infty \frac{a_{n,1}^2}{(\lambda-\lambda_n^{(0)}(\rho))^2}<0,
\end{equation*}
so $F$ is strictly decreasing.
Fix $m\ge 1$ and consider $(\kappa_m,\kappa_{m-1})$.
As $\lambda\downarrow \kappa_m$, the term indexed by the pole $\kappa_m$ behaves like
$a^2/(\lambda-\kappa_m)\to +\infty$, while all other terms remain finite; hence $F(\lambda)\to +\infty$.
As $\lambda\uparrow \kappa_{m-1}$, the term indexed by $\kappa_{m-1}$ behaves like
$a^2/(\lambda-\kappa_{m-1})\to -\infty$ (since $\lambda-\kappa_{m-1}\uparrow 0$ from the negative side),
so $F(\lambda)\to -\infty$.
By strict monotonicity, there is exactly one (necessarily simple) root in $(\kappa_m,\kappa_{m-1})$.
This gives strict interlacing. For $\lambda>\kappa_0$ every summand in $F(\lambda)-1$ is positive. At a pole $\kappa_j$, write $z=\langle f,u_1\rangle$. The $j$th coordinate equation is $a_{j,1}z=0$, so $z=0$. Every other coordinate is then zero; finally $z=a_{j,1}c_j=0$ forces $c_j=0$. Thus neither a pole nor a value above $\kappa_0$ is a positive eigenvalue. The preceding equivalence proves exhaustion.
\end{proof}

\subsection{Rank-one interlacing}\label{supp-subsec:proof_minmax_rank1}
\begin{proof}[Proof of Lemma~\ref{main-lem:minmax_rank1}]
Choose orthonormal eigenvectors for the largest $m+2$ eigenvalues of $A$, taking the required number at a multiple eigenvalue, and denote their span by $E_{m+1}$. Then $\dim E_{m+1}=m+2$, and $E_{m+1}\cap u^\perp$ has dimension at least $m+1$. For every unit vector $f$ in this intersection,
\begin{equation*}
\langle\widetilde Af,f\rangle=\langle Af,f\rangle\ge\lambda_{m+1}(A)>0.
\end{equation*}
The min--max principle therefore gives at least $m+1$ positive eigenvalues of $\widetilde A$ and the lower bound $\lambda_m(\widetilde A)\ge\lambda_{m+1}(A)$. Since $m$ is arbitrary, the positive spectrum is infinite. The upper bound follows from the same principle and $\langle\widetilde Af,f\rangle\le\langle Af,f\rangle$.
\end{proof}

\subsection{Projection of residual empirical sums}\label{supp-subsec:proof_projection_empirical}
\begin{proof}[Proof of Lemma~\ref{main-lem:data_projection_empirical_sums}]
Put $a_n:=\overline X_n$ and $b_n:=S_n-1$.  On the set $|b_n|\le 1/2$ and $|a_n|\le1$, whose probability tends to one,
\begin{equation*}
R_{n,i}=\frac{X_i-a_n}{1+b_n}.
\end{equation*}
For fixed $x$, define $H_{k,x}(u,v):=r_k\{(x-u)/(1+v)\}$.  Taylor's formula at $(0,0)$ gives
\begin{equation*}
H_{k,x}(u,v)
=
r_k(x)-u r_k'(x)-v x r_k'(x)+E_{k,x}(u,v).
\end{equation*}
Since the family $r_0,\ldots,r_M$ is finite and the first two derivatives have polynomial growth, there are constants $C,L<\infty$ such that, for $|v|\le1/2$ and $|u|\le1$,
\begin{equation*}
|E_{k,x}(u,v)|
\le
C(1+|x|^L)(u^2+v^2),
\qquad k=0,\ldots,M.
\end{equation*}
Consequently,
\begin{equation*}
\frac1{\sqrt n}\sum_{i=1}^n
\left|E_{k,X_i}(a_n,b_n)\right|
\le
\sqrt n(a_n^2+b_n^2)\,C\frac1n\sum_{i=1}^n(1+|X_i|^L)
=o_{\mathbb P}(1),
\end{equation*}
uniformly in $k$.  Hence
\begin{align*}
\frac1{\sqrt n}\sum_{i=1}^n r_k(R_{n,i})
&=
\frac1{\sqrt n}\sum_{i=1}^n r_k(X_i)
-
\sqrt n\,\overline X_n
\left(\frac1n\sum_{i=1}^n r_k'(X_i)\right)\\
&\quad
-
\sqrt n(S_n-1)
\left(\frac1n\sum_{i=1}^n X_i r_k'(X_i)\right)
+o_{\mathbb P}(1).
\end{align*}
Furthermore,
\begin{equation*}
\sqrt n\,\overline X_n=\frac1{\sqrt n}\sum_{i=1}^n X_i,
\qquad
\sqrt n(S_n-1)=\frac1{2\sqrt n}\sum_{i=1}^n(X_i^2-1)+o_{\mathbb P}(1).
\end{equation*}
By the law of large numbers and Gaussian integration by parts,
\begin{equation*}
\mathbb E[r_k'(X)]=\mathbb E[Xr_k(X)]=\langle r_k,q_1\rangle_\Phi,
\end{equation*}
and
\begin{equation*}
\mathbb E[Xr_k'(X)]=\mathbb E[(X^2-1)r_k(X)]=\sqrt2\,\langle r_k,q_2\rangle_\Phi.
\end{equation*}
Substitution proves \eqref{main-eq:data_projection_musigma}.  The convergence of the empirical second moments follows by applying the same Taylor argument to the finite family $r_0^2,\ldots,r_M^2$.
\end{proof}

\subsection{The infinite-rank quadratic-form limit}\label{supp-subsec:proof_infinite_rank_limit}
\begin{proof}[Proof of Proposition~\ref{main-prop:data_infinite_rank_limit_main}]
Let $\mathcal W$ be an isonormal Gaussian process on the real space $L^2(\Phi)$, that is, a centered Gaussian family with covariance $\EE[\mathcal W(f)\mathcal W(g)]=\langle f,g\rangle_\Phi$. The finite-rank limits in Theorem~\ref{main-thm:data_finite_rank_projection_limit} can be realized on this common probability space as
\begin{equation*}
Q_M:=\sum_{k=0}^{M}\gamma_k\{\mathcal W^2(\Pi_{\mu,\sigma}r_k)-1\}.
\end{equation*}
Since $\|\Pi_{\mu,\sigma}r_k\|_\Phi\le1$, the expected absolute value of the $k$th summand is at most $2|\gamma_k|$. Hence $Q_M$ converges absolutely almost surely and in $L^1$.

Put $B_M=\Pi_{\mu,\sigma}T_{K_M}\Pi_{\mu,\sigma}$ and $B=\Pi_{\mu,\sigma}T_K\Pi_{\mu,\sigma}$. Then
\begin{equation*}
\|B-B_M\|_1\le\sum_{k>M}|\gamma_k|\longrightarrow0.
\end{equation*}
For a finite-rank self-adjoint operator $D=\sum_\ell d_\ell u_\ell\otimes u_\ell$ with orthonormal eigenvectors, set $Q(D)=\sum_\ell d_\ell\mathcal W^2(u_\ell)$. On a common finite-dimensional subspace this is a Gaussian quadratic form; hence the definition is independent of the eigenbasis and is linear in $D$. Diagonalization gives $\EE|Q(D)|\le\|D\|_1$. It therefore extends uniquely, by $L^1$ continuity, to trace-class self-adjoint operators. Applying this extension to $B_M\to B$ and then using an eigenbasis of $B$ identifies the limit of $Q_M$ as the right-hand side of \eqref{main-eq:data_infinite_rank_limit_main}. This argument applies also to signed $\gamma_k$ and does not require a decreasing enumeration of a mixed-sign spectrum.

For completeness, put $T_n=nU_n^{(\mu,\sigma)}(K)$ and $T_{n,M}=nU_n^{(\mu,\sigma)}(K_M)$, and let $Q$ be the $L^1$ limit of $Q_M$. If $f$ is bounded and Lipschitz with constant $L_f$, then, for every $\varepsilon>0$,
\begin{align*}
|\EE f(T_n)-\EE f(Q)|
&\le L_f\varepsilon+2\|f\|_\infty\PP(|T_n-T_{n,M}|>\varepsilon)\\
&\quad+|\EE f(T_{n,M})-\EE f(Q_M)|+L_f\EE|Q_M-Q|.
\end{align*}
For fixed $M$ the third term tends to zero by Theorem~\ref{main-thm:data_finite_rank_projection_limit}. Take $\limsup_{n\to\infty}$, then let $M\to\infty$ using \eqref{main-eq:data_tail_condition_probability} and $L^1$ convergence, and finally let $\varepsilon\downarrow0$. Convergence against bounded Lipschitz functions proves the assertion.
\end{proof}

\subsection{The real Hilbert-space factorization}\label{supp-subsec:proof_factorization}
\begin{proof}[Proof of Proposition~\ref{main-prop:data_frequency_factorization}]
For $a\in\mathcal H_\theta$, the adjoint is
\begin{equation*}
(L_\theta^*a)(x)=\langle a,\xi_x\rangle_\theta.
\end{equation*}
Using the symmetry and reality of $K_\theta$,
\begin{equation*}
(L_\theta^*L_\theta f)(x)
=
\int_{-\infty}^{\infty}K_\theta(x,y)f(y)\,d\Phi(y).
\end{equation*}
Thus $L_\theta^*L_\theta=T_{K_\theta}$.  Since $\Pi_{\mu,\sigma}$ is self-adjoint,
\eqref{main-eq:factorization_data_side} follows.

Self-adjointness of $\Pi_{\mu,\sigma}$ and the definition of the projected summand give
\begin{equation*}
L_{\theta,\mu,\sigma}f=\EE_\Phi[f(X)\xi_X^{(\mu,\sigma)}],
\qquad
(L_{\theta,\mu,\sigma}^*a)(x)=\langle a,\xi_x^{(\mu,\sigma)}\rangle_\theta.
\end{equation*}
The inner product is real on $\mathcal H_\theta$. Therefore, for $a\in\mathcal H_\theta$,
\begin{align*}
(L_{\theta,\mu,\sigma}L_{\theta,\mu,\sigma}^*a)(s)
&=\EE_\Phi[\langle a,\xi_X^{(\mu,\sigma)}\rangle_\theta
\xi_X^{(\mu,\sigma)}(s)]\\
&=\int C_\theta^{(\mu,\sigma)}(s,t)a(t)\theta(t)\,dt,
\end{align*}
where $C_\theta^{(\mu,\sigma)}(s,t)=\EE_\Phi[\xi_X^{(\mu,\sigma)}(s)
\overline{\xi_X^{(\mu,\sigma)}(t)}]$. The interchange is justified by Cauchy--Schwarz and $\EE_\Phi\|\xi_X^{(\mu,\sigma)}\|_\theta^2<\infty$. This is the real covariance operator, proving \eqref{main-eq:factorization_frequency_side}.

Finally, for a compact operator $L$ between real Hilbert spaces, $L^*L$ and $LL^*$ have the same positive eigenvalues.  If $L^*Lh=\lambda h$ with $\lambda>0$, then $\lambda^{-1/2}Lh$ is an eigenfunction of $LL^*$ with eigenvalue $\lambda$; the inverse on the corresponding eigenspaces is $\lambda^{-1/2}L^*$, which also proves equality of multiplicities. The polar decomposition gives a unitary equivalence on the orthogonal complements of the respective kernels.

These identifications also preserve parity. The projection $\Pi_{\mu,\sigma}$ commutes with reflection, since its two removed directions are respectively odd and even. Symmetry of $\Phi$ and $\xi_{-x}(t)=\xi_x(-t)$ give
\begin{equation*}
\bigl(L_{\theta,\mu,\sigma}[f(-\cdot)]\bigr)(t)
=\bigl(L_{\theta,\mu,\sigma}f\bigr)(-t).
\end{equation*}
Taking adjoints gives the corresponding identity for $L_{\theta,\mu,\sigma}^*$. Thus the eigenspace bijections preserve even and odd parity, with odd Fourier functions understood as imaginary-valued members of $\mathcal H_\theta$.
\end{proof}

\subsection{The original-trace subtraction}\label{supp-subsec:proof_U_V_limit}
\begin{proof}[Proof of Corollary~\ref{main-cor:full_U_V_limit}]
The Gaussian kernel formula gives
\begin{equation*}
K_\beta(x,x)=1-\frac{2}{\sqrt{1+\beta^2}}
\exp\!\left\{-\frac{\beta^2x^2}{2(1+\beta^2)}\right\}
+\frac1{\sqrt{1+2\beta^2}}.
\end{equation*}
This function is bounded and has bounded derivative. Since
\begin{equation*}
\frac1n\sum_j|V_{n,j}-X_j|
\le|S_n^{-1}-1|\frac1n\sum_j|X_j|+S_n^{-1}|\overline X_n|
=O_{\PP}(n^{-1/2}),
\end{equation*}
its empirical mean at the residuals differs by $o_{\PP}(1)$ from its empirical mean at the independent observations. The law of large numbers and a Gaussian integral give
\begin{equation*}
\frac1n\sum_jK_\beta(V_{n,j},V_{n,j})
\xrightarrow[n\to\infty]{\PP}\EE[K_\beta(X,X)]
=1-(1+2\beta^2)^{-1/2}=\tau_\beta.
\end{equation*}
Apply the exact identity \eqref{main-eq:UV_identity} at these residuals, then Theorem~\ref{main-thm:mu_sigma_main} and the scaling \eqref{main-eq:kernel_scaling_rho_beta}. This proves the asserted limit, with the original trace $\tau_\beta$ subtracted, in agreement with Corollary~\ref{main-cor:weighted_kernel_full_tail}.

Finally, the expectation of $T_\infty^{(\mu,\sigma)}$ is the trace of its covariance operator. From \eqref{main-eq:mu_sigma_operator_form},
\begin{align*}
\operatorname{tr}(A^{(\mu,\sigma)})-\tau_\beta
&=-\|\phi_1^*\|_\beta^2-\|\phi_2^*\|_\beta^2\\
&=-\frac{\beta^2}{(1+2\beta^2)^{3/2}}
-\frac{3\beta^4}{2(1+2\beta^2)^{5/2}}.
\end{align*}
Multiplication by $(1-\rho)^{-1}$ proves the formula for the limiting mean.
\end{proof}

\begingroup
\small
\bibliographystyle{imsart-number-issue}
\bibliography{references}
\endgroup

\end{document}